\documentclass[smallextended,numbook,runningheads]{svjour3}
\smartqed  
\usepackage{xcolor}
\usepackage{amsmath,amssymb,bm,cancel}
\usepackage{graphicx,graphics,color}
\usepackage[colorlinks,linkcolor=blue]{hyperref}
\usepackage{verbatim}

\usepackage{subcaption}
\usepackage{enumerate}

\journalname{...}
\date{ \phantom{b} \vspace{45mm}\phantom{e}}
\def\makeheadbox{\hspace{-4mm}Version of \today}

\newcommand{\tnorm}{|\!|\!|}
\newcommand{\R}{\mathbb{R}}
\newcommand{\N}{\mathbb{N}}
\DeclareMathOperator{\Span}{span}
\newcommand{\Rn}[1][3]{\mathbb{R}^#1}
\newcommand{\abs}[1]{\lvert #1 \rvert}
\newcommand{\tr}[1]{\text{tr}(#1)}
\newcommand{\norm}[1]{\left\|#1\right\|}
\newcommand{\normcuad}[1]{\left\|#1\right\|^2}

\newcommand{\Gt}{\mathcal{G}_{T}} 
\newcommand{\Ght}{\mathcal{G}_{h,T}} 
\newcommand{\Gammat}[1][t]{\Gamma(#1)}
\newcommand{\Gammah}{\Gamma_h}
\newcommand{\Gammaht}[1][t]{\Gammah(#1)}
\newcommand{\Gammas}{{\Gamma^*}}
\newcommand{\Gammast}[1][t]{\Gamma^*(#1)}
\newcommand{\Gammatilden}{\widetilde{\Gamma}_h^k}
\newcommand{\Gammahj}[1][k]{\Gammah^{#1}}
\newcommand{\Gammasj}[1][k]{\Gamma_{*}^{#1}}
\newcommand{\Gammatildesn}{\widetilde{\Gamma}_{*}^k}

\newcommand{\Gammabar}{\Lambda}
\newcommand{\Gammabart}{\Lambda(\theta)}
\newcommand{\gradgbarra}{\nabla_{\Gammabar}}

\newcommand{\boundaryG}{\partial \Gamma_0}
\newcommand{\boundaryGt}[1][t]{\partial \Gamma(#1)}

\newcommand{\boundaryGht}[1][t]{\partial \Gammah(#1)}
\newcommand{\boundaryGhcero}{\partial \Gamma_{h,0}}

\newcommand{\boundaryGs}{\partial \Gammas}
\newcommand{\boundaryGscero}{\partial {\Gamma}_0^*}

\newcommand{\gradg}{\nabla_{\Gamma}}

\newcommand{\gradgh}{\nabla_{\Gammah}}
\newcommand{\gradght}{\nabla_{\Gammaht}}
\newcommand{\gradghs}{\nabla_{\Gammas}}
\newcommand{\gradghst}{\nabla_{\Gammat}}

\DeclareMathOperator{\Div}{div}
\newcommand{\Divg}{\Div_{\Gamma}}

\newcommand{\Divgh}{\Div_{\Gamma_h}}
\newcommand{\Divgs}{\Div_{\Gamma^*}}
\newcommand{\Divgbar}{\Div_{\Gammabar}}

\newcommand{\laplg}{\Delta_{\Gamma}}

\newcommand{\dermat}{\partial^{\bullet}}
\def\tensor{\otimes}

\newcommand{\dt}{\frac{d}{d t}}
\newcommand{\partdt}{\frac{\partial}{\partial t}}

\newcommand{\argu}[1][t]{(#1)}

\newcommand{\Otau}{\mathcal{O}(\bs\tau)}

\newcommand{\Lp}[1][2]{L^{#1}}
\newcommand{\Hk}[1][1]{H^{#1}}
\newcommand{\Wunoinf}{W^{1,\infty}}

\newcommand{\Linfomega}{L^\infty(\Omega)}
\newcommand{\Ldosgamma}{\Lp(\Gamma)}
\newcommand{\Ldosgammah}{\Lp(\Gammah)}
\newcommand{\Ldosgammas}{\Lp(\Gamma^*)}

\newcommand{\Hunogamma}{\Hk(\Gamma)}

\newcommand{\Hunogammas}{\Hk(\Gamma^*)}

\newcommand{\Hunogammalambda}{\Hk_{\lambda}(\Gamma)}

\newcommand{\Ldosgammast}[1][t]{\Lp(\Gamma^*(#1))}
\newcommand{\Hunogammat}[1][t]{\Hk(\Gamma(#1))}

\newcommand{\Hunogammast}[1][t]{\Hk(\Gamma^*(#1))}

\newcommand{\Lpgammat}[1][t]{\Lp(\Gamma(#1))}

\newcommand{\normHunog}[1]{\norm{#1}_{\Hk (\Gamma)}}
\newcommand{\normHunogs}[1]{\norm{#1}_{\Hk (\Gammas)}}

\newcommand{\normHunogt}[2][t]{\norm{#2}_{\Hk(\Gammat[#1])}}
\newcommand{\normHunogst}[1]{\norm{#1}_{\Hk (\Gammast)}}

\newcommand{\normHunogtvec}[2][t]{\norm{#2}_{\Hk(\Gammat[#1])}}
\newcommand{\normHunogstvec}[1]{\norm{#1}_{\Hk (\Gammast)}}

\newcommand{\normLinfomega}[1]{\norm{#1}_{\Linfomega}}
\newcommand{\normLdosg}[1]{\norm{#1}_{\Lp (\Gamma)}}
\newcommand{\normLdosgs}[1]{\norm{#1}_{\Lp (\Gammas)}}
\newcommand{\normLdosgh}[1]{\norm{#1}_{\Lp (\Gammah)}}

\newcommand{\normWunogt}[1]{\norm{#1}_{W^{1,\infty} (\Gammat)}}
\newcommand{\normWunomega}[1]{\norm{#1}_{W^{1,\infty} (\Omega)}}

\newcommand{\normWunost}[1]{\norm{#1}_{W^{1,\infty} (\Gammast)}}

\newcommand{\normLinfs}[1]{\norm{#1}_{L^{\infty} (\Gammas)}}

\newcommand{\intgamma}{\int_{\Gamma}}
\newcommand{\intgammat}{\int_{\Gammat}}
\newcommand{\intgammas}{\int_{\Gamma^*}}

\newcommand{\intgammah}{\int_{\Gammah}}

\newcommand{\intgammabar}{\int_{\Gammabar}}

\newcommand{\intbgamma}{\int_{\boundaryG}}

\newcommand{\intbgammas}{\int_{\boundaryGscero}}

\newcommand{\X}{\bs{X}}
\newcommand{\Xcero}{\X_0}
\newcommand{\invX}{(\X)^{-1}}
\newcommand{\invXt}{(\X\argu)^{-1}}
\newcommand{\Xb}{\bs{Y}}
\newcommand{\Xh}{\X_h}
\newcommand{\Xhcero}{\X_{h,0}}
\newcommand{\invXh}{(\Xh)^{-1}}
\newcommand{\Xhj}[1][k]{\Xh^{#1}}
\newcommand{\Xs}{{\bs{X}^*}}
\newcommand{\Xsj}[1][k]{\X_{*}^{#1}}
\newcommand{\invXs}{(\Xs)^{-1}}
\newcommand{\Xscero}{\Xs_0}
\newcommand{\Xhtilden}[1][k]{\widetilde{\X}_h^{#1}}
\newcommand{\Xstilden}[1][k]{\widetilde{\X}_{*}^{#1}}

\newcommand{\funs}[1]{#1^*_h}
\newcommand{\funss}[1]{#1^*}

\newcommand{\liftg}[1]{(#1)^{\Gamma}}
\newcommand{\liftgamma}[1]{(#1)^{\Gamma}}
\newcommand{\liftgbarra}[1]{(#1)^{\Gammabart}}
\newcommand{\liftlambda}[1]{(#1)^{\Gammabar}}
\newcommand{\lifts}[1]{(#1)^{\Gamma^*}}
\newcommand{\lifth}[1]{(#1)^{\Gammah}}
\newcommand{\liftb}[1]{(#1)^{\boundaryGhcero}}

\newcommand{\liftgammatilden}[1]{(#1)^{\Gammatilden}}
\newcommand{\liftgammatildesn}[1]{(#1)^{\Gammatildesn}}
\newcommand{\liftgammasj}[1]{(#1)^{\Gammasj}}

\newcommand{\bs}[1]{\boldsymbol{#1}}
\newcommand{\vel}{\bs{v}}

\newcommand{\testv}{\bs{\varphi}}
\newcommand{\testnu}{\bs{\psi}}
\newcommand{\testH}{\eta}

\newcommand{\idg}{\bs{x}}
\newcommand{\idh}{\bs{x}_h}
\newcommand{\ids}{\bs{x}^*}

\newcommand{\Hm}{\kappa}
\newcommand{\Hmcero}{\Hm_0}
\newcommand{\Hh}{\Hm_h}

\newcommand{\Hs}{\Hm^*}

\newcommand{\n}{\bs{\nu}}
\newcommand{\ncero}{\n_0}
\newcommand{\ns}{\bs{\nu}^*}

\newcommand{\nh}{\bs{\nu}_h}
\newcommand{\nhcero}{\n_{h,0}}

\newcommand{\conor}{\bs{\mu}}

\newcommand{\conors}{\conor^*}
\newcommand{\conorh}{\conor_h}

\newcommand{\vs}{\vel^*}
\newcommand{\vh}{\vel_h}

\newcommand{\btau}{\bs\tau}
\newcommand{\btauh}{\btau_h}
\newcommand{\btaus}{\btau^*}

\newcommand{\btauhunit}{\hat{\btau}_h}

\newcommand{\curvbdry}{\bs\kappa_{\partial}}
\newcommand{\curvbdrys}{\bs\kappa^*_{\partial}}
\newcommand{\curvbdryh}{\bs\kappa_{\partial,h}}

\newcommand{\Sparamunid}{\mathbb{S}_{p,\ell,h}}
\newcommand{\Sparam}{\mathbb{S}_{h}}

\newcommand{\Sparamcero}{\mathbb{S}_{h,0}}

\newcommand{\Sh}{\mathcal{S}_{h}(t)}
\newcommand{\Shcero}{\mathcal{S}_{h,0}(t)}
\newcommand{\Shs}{\mathcal{S}_{h}^*(t)}
\newcommand{\Shscero}{\mathcal{S}^*_{h,0}(t)}

\newcommand{\Otauh}{\mathcal{O}_h(\btauh)}
\newcommand{\Otauhb}{\mathcal{O}_h^{\partial}(\btauh)}
\newcommand{\Otauhs}{\mathcal{O}_h^*(\btauh)}
\newcommand{\Otauhsb}{\mathcal{O}_h^{*,\partial}(\btauh)}
\newcommand{\Ptau}{\mathcal{P}(\btau)}

\newcommand{\basisfo}[1][j]{B_{#1}}
\newcommand{\basisfg}[1][j]{b_{#1}}

\newcommand{\testnuh}{\testnu_h}
\newcommand{\testnuhs}{\testnu_h^*}
\newcommand{\testHh}{\testH_h}
\newcommand{\testHhs}{\testH_h^*}

\newcommand{\Q}{\mathcal{Q}}
\newcommand{\Qs}[1][\Gammas]{\mathcal{Q}_{#1}}

\newcommand{\Rs}[1][\Gammas]{\mathcal{R}_{#1}}
\newcommand{\Rscero}[1][\Gammas]{\mathcal{R}_{#1}^0}
\newcommand{\Ps}[1][h]{\mathcal{P}_{#1}}

\newcommand{\deltat}{\Delta t}
\newcommand{\idhn}[1][k]{\bs{x}_h^{#1}} 
\newcommand{\idsn}[1][k]{\bs{x}_{*}^{#1}}

\newcommand{\Hhn}[1][k]{\Hm_h^{#1}}
\newcommand{\Hsn}[1][k]{\Hm_{*}^{#1}}

\newcommand{\nhn}[1][k]{\bs{\nu}_h^{#1}}
\newcommand{\nsn}[1][k]{\bs{\nu}_{*}^{#1}}

\newcommand{\vhn}[1][k]{\vh^{#1}}
\newcommand{\vsn}[1][k]{\vel_{*}^{#1}}

\newcommand{\extraid}{\widetilde{\idg}_h^k}
\newcommand{\extrakappa}{\widetilde{\Hm}_h^k}
\newcommand{\extranu}{\widetilde{\bs{\nu}}_h^k}
\newcommand{\extraA}{\widetilde{A}_h^k}

\newcommand{\extrakappas}{\widetilde{\Hm}_{*}^k}
\newcommand{\extranus}{\widetilde{\bs{\nu}}_{*}^k}
\newcommand{\extraAs}{\widetilde{A}_{*}^k}

\newcommand{\phin}[1][k]{\testv^k}
\newcommand{\psin}[1][k]{\testnu^k}
\newcommand{\wn}[1][k]{w^k}

\newcommand{\xdotn}[1][k]{\dot{\idg}_h^{#1}}
\newcommand{\hdotn}[1][k]{\dot{\Hm}_h^{#1}}
\newcommand{\nudotn}[1][k]{\dot{\n}_h^{#1}}

\newcommand{\xsdotn}[1][k]{\dot{\idg}_h^{#1}}
\newcommand{\hsdotn}[1][k]{\dot{\Hm}_h^{#1}}
\newcommand{\nusdotn}[1][k]{\dot{\n}_h^{#1}}

\newcommand{\EX}{E_{\bs{X}}}
\newcommand{\ex}{e_{\bs{x}}}
\newcommand{\ederx}{\dermat \ex}
\newcommand{\ev}{e_{\vel}}
\newcommand{\en}{e_{\n}}
\newcommand{\eH}{e_\Hm}

\newcommand{\eR}{e_{\mathcal{R}}}

\newcommand{\dn}{d_{\n}}
\newcommand{\dv}{d_{\vel}}
\newcommand{\dH}{d_{\Hm}}

\newcommand{\dnk}{d_{\n}^k}
\newcommand{\dvk}{d_{\vel}^k}
\newcommand{\dHk}{d_{\Hm}}
\newcommand{\dxk}{d_{\idg}^k}

\newcommand{\exk}[1][k]{e_{\bs{x}}^{#1}}
\newcommand{\evk}[1][k]{e_{\bs{v}}^{#1}}
\newcommand{\enk}[1][k]{e_{\n}^{#1}}
\newcommand{\eHk}[1][k]{e_{\Hm}^{#1}}

\newcommand{\exdotn}[1][k]{\dot{e}_{\bs{x}}^{#1}}
\newcommand{\ehdotn}[1][k]{\dot{e}_{\Hm}^{#1}}
\newcommand{\enudotn}[1][k]{\dot{e}_{\n}^{#1}}

\begin{document}

\title{A convergent algorithm for mean curvature flow of surfaces with Dirichlet boundary conditions}

\titlerunning{A convergent algorithm for mean curvature flow of surfaces Dirichlet b.c.}        

\author{Bárbara S.\ Ivaniszyn \and Pedro Morin \and M.\ Sebastián Pauletti}

\authorrunning{} 

\institute{Universidad Nacional del Litoral and CONICET, 
Departamento de Matem\'atica, 
Facultad de Ingenier\'{\i}a Qu\'{\i}mica, Santiago del Estero 2829, S3000AOM Santa Fe, Argentina. 
Emails: 
\href{mailto:pmorin@fiq.unl.edu.ar}{pmorin@fiq.unl.edu.ar},
\href{mailto:ivaniszyn@gmail.com}{bivaniszyn@fiq.unl.edu.ar},
\href{mailto:spauletti@fiq.unl.edu.ar}{spauletti@fiq.unl.edu.ar}
}

\date{\today}

\maketitle

\begin{abstract}

    We establish convergence results for a spatial semidiscretization of Mean Curvature Flow (MCF) for surfaces with fixed boundaries. 
    Our analysis is based on Huisken’s evolution equations for the mean curvature and the normal vector, enabling precise control of discretization errors and yielding optimal error estimates for discrete spaces with piecewise polynomials of degree $p \geq 2$.
    Building on techniques recently developed by Kovács, Li, Lubich, and collaborators for closed surfaces, we extend these ideas to surfaces with boundaries by formulating appropriate boundary conditions for both the mean curvature and the normal vector. These boundary treatments are essential for proving convergence.
    The core of our analysis involves a classical error splitting strategy using auxiliary discrete functions that approximate the surface geometry, the mean curvature, and the normal vector. We estimate two types of errors for each variable to rigorously assess both stability and consistency.
    To effectively handle boundary conditions for the normal vector, we introduce a nonlinear Ritz projection into the analysis. As a result, we derive optimal $H^1$ error estimates for the surface position, velocity, mean curvature, and normal vector. Our theoretical findings are corroborated by numerical experiments.

	\keywords{mean curvature flow \and surfaces with boundary \and geometric evolution equations \and boundary conditions \and isogeometric analysis \and linearly implicit backward difference formula \and stability \and error estimates}
\end{abstract}

\tableofcontents 

\section{Introduction}
The mean curvature flow is a geometric evolution which has received much attention during the last decades; see~\cite{KLL2019} and the references therein. In this article we are interested in the mean curvature flow with a fixed boundary.

From the point of view of analysis, Stone~\cite{Sto1996} studied the mean curvature flow of parametrized surfaces with a fixed boundary, finding conditions for the existence of smooth solutions over finite time intervals. 
From the computational point of view, various methods have been proposed to approximate the solution of this problem in the case of parametrized surfaces with and without boundaries. In 1990, Dzuik presented the first algorithm based on a weak formulation from the following equation:
\begin{equation*}
	\dermat \idg =\laplg \idg,\quad \text{on }\Gammat,
\end{equation*}
where $\idg$ is the identity function on $\Gammat$. He then presents the Evolving Surface Finite Element Method (ESFEM)~\cite{Dzi1991} along with a numerical method for parametrized surfaces, with or without boundaries, where the movement of the nodes of the finite element mesh determines the approximate evolving surface. Later, Barrett, Garcke, and Nürnberg proposed another method~\cite{BGN2007,BGN2008a}, also using finite elements. It is also valid for the case of surfaces with fixed boundaries, as explained in~\cite{BGN2020275}. However, their scheme considers a different weak formulation than Dziuk's, as it is based on the following equations that relate the velocity $\vel$, the normal vector $\n$, the mean curvature $\Hm$, and the identity function:
\begin{align*}
	\vel\cdot \n = \Hm,\quad\ \Hm\n=\laplg\idg.
\end{align*}
In 2016, an adaptation of the semi-discrete scheme proposed by Dziuk in the context of Isogeometric Analysis (IgA)~\cite{BDQ2016} was presented, where the authors report good numerical results although without error estimates.

Regarding the analysis and the error estimation, significant progress has been made in recent years for the case of closed parametrized surfaces, that is, compact and without boundaries. The first convergence proof was presented in 2019 by Kov\'acs, Li, and Lubich in~\cite{KLL2019}. In this scheme, they include the evolution equations of the normal vector and the mean curvature (presented in~\cite{Huisken1984}) into the set of equations from where they obtain their weak formulation:
\begin{align*}
	\vel = \Hm\n, \quad \dermat \Hm=\laplg \Hm + |\gradg\n|^2\Hm,\quad \dermat \n=\laplg \n + |\gradg\n|^2\n.
\end{align*}
With this modification, although they increase the number of equations in the scheme, they manage to obtain the first error estimates for a discrete method for mean curvature flow of closed parametric surfaces. To this date, it is the only numerical scheme for which estimates have been obtained considering polynomial degree $p\geq 2$.

It is worth noting that in 2021, Buyang Li~\cite{Li2021} presented the first convergence proof for Dziuk's semi-discrete scheme using a matrix-vector formulation and some techniques previously used in~\cite{KLLP2017,KLL2019} under the restriction of using polynomial finite elements of degree $p\geq 6$. Later, in 2023, the same author and Genming Bai~\cite{BL2024} proposed a new approach to analyze not only the method proposed by Dziuk but also other methods such as those proposed in~\cite{BGN2007,BGN2008a,BGN2020275}, for which the error analysis is still an open problem. In this new approach, they perform the error analysis for the fully discrete scheme considering the projected distance between the calculated numerical surface and the exact surface instead of analyzing the error considering the distance between the two surfaces as done in~\cite{Li2021} or in~\cite{KLL2019}. With this new approach, they manage to improve the order of convergence by relaxing the polynomial degree restriction to $p\geq 3$. 

Previous studies on error analysis have focused on problems in closed surfaces.
As far as we know, no error estimates have yet been obtained for the numerical schemes proposed on surfaces with boundaries.
Following the ideas of~\cite{KLL2019}, in this article we present a convergent numerical method for mean curvature flow of parametric surfaces with boundaries. In particular, we consider the case where the boundary remains fixed during the evolution, that is, a Dirichlet-type boundary condition. We work in the framework of isogeometric analysis, but our analysis is valid also for the finite element method, because the required global regularity of the spline spaces can be taken as $C^0$. 

The outline of the paper is the following.
In Section~\ref{sec: MCF}, we propose a continuous weak formulation for the mean curvature flow problem that will be the basis of our method. 
We use Huisken's equations~\cite{Huisken1984} for the mean curvature and the normal, and derive suitable boundary conditions for the normal vector and its derivative in the conormal direction.

In Section~\ref{sec: spline approximation estimates}, we introduce the different discrete surfaces that we will use for the numerical method and its analysis, as well as the evolving spline spaces defined on these surfaces. Additionally, we study the geometric consistency errors, which are the errors made when approximating the continuous surface by a theoretical discrete one, as well as the continuous solutions by certain discrete approximations. 
We define a linear and a nonlinear Ritz projection operator. The latter is needed to appropriately handle the new boundary conditions and obtain the necessary estimates for the consistency of the method.

In Section~\ref{sec: semidisc espacial}, we present the semi-discrete scheme using tensor product spline spaces of polynomial degree $p\ge 2$ and study its consistency and stability. We obtain optimal order error estimates in the $\Hk$ norm for the errors in position, velocity, mean curvature, and normal vector under sufficient regularity conditions in the solution.

Finally, we furthermore detail relevant aspects regarding the implementation of the method and present numerical experiments that reflect its behavior.

\paragraph{Main Challenges.}
Here are the main challenges that we faced during the development of this article.

\begin{itemize}
	\item Boundary condition for the normal vector: the consideration of the evolution equation of the normal vector in the weak formulation of the problem, made it necessary to determine an adequate boundary condition for the normal vector. More precisely, to determine some relationship regarding the behavior, at the boundary of the surface, of the derivative of the normal vector with the conormal direction.
	\item Geometric errors: We do not use the oriented distance function to define the lift operators used to compare functions and bilinear forms defined on the surface $\Gammat$ and on a discrete surface $\Gammast$. This prevented us to rely on the usual results from the literature of evolutionary problems on surfaces~\cite{DE2007,K2018} regarding geometric perturbation errors. We thus had to extend results such as those used for the stationary Laplace-Beltrami problem in~\cite{BDN20201}.
	\item Ritz projection for the normal vector: one of the main challenges in proving the consistency of the semi-discrete scheme was the definition of an appropriate Ritz projection for the normal vector. The challenge was originated from the boundary conditions for the normal vector and the structure of the semi-discretization.
	\item Approximation properties of the nonlinear Ritz projection: 
  Due to our definition of the discrete ansatz space for the normal vector, which turns out to be non-conforming, the usual interpolation operators cannot be used to obtain an optimal order approximation error in $\Hk$.

	\item Stability analysis in the spatial semi-discretization: due to the boundary condition for the normal vector, the main difficulty in obtaining stability for the semi-discrete scheme was the appearance of a critical term in the energy estimate for the error in this variable. The appearance of the $L^2$ norm in the boundary of the error for the normal vector and the decision of how to handle it, to control it, was an important point in this article.
	
\end{itemize}

\paragraph{Main Contributions.}
These are the main contributions, sparked by the aforementioned challenges, that we have developed.
\begin{itemize}
\item In Section~\ref{sec: MCF}, Lemma~\ref{Lem: BC for nu}, we present an adequate boundary condition for the derivative in the conormal direction of the normal vector.
\item In Section~\ref{sec: weak formulation}, we present a weak formulation for the problem with Dirichlet boundary. 
\item In Section~\ref{sec: geometric error}, we present results corresponding to the geometric errors, without resorting to the oriented distance function to obtain optimal order estimates.
\item In Section~\ref{sec: Linear Ritz projection}, we define a linear Ritz projection with zero trace values and present in Proposition~\ref{prop: error estimate linear Ru} its approximation properties, as well as those of its material derivative in Proposition~\ref{prop: error estimate dermat linear Ru}.
\item In Section~\ref{sec: Nonlinear Ritz projection}, we define a nonlinear Ritz projection with optimal approximation properties in $\Hk$ norm, which are presented in Propositions~\ref{prop: error estimate non-linear Ru} and~\ref{prop: error estimate dermat Ritz proj on boundary}. 
\item In Section~\ref{sec: consistency}, we obtain optimal estimates for the consistency errors of the spatial semi-discretization proposed in Section~\ref{sec: spatial semi disc}. These are presented in Proposition~\ref{prop: consistency estimates}.
\item In Section~\ref{sec: error equations}, we present the error equations used as well as a set of key auxiliary results to compare and estimate errors between discrete surfaces and obtain stability estimates.
\item In Section~\ref{sec: stability}, in Lemmas~\ref{Lem: energy estimate ex ev} to~\ref{Lem: energy estimate en}, we obtain the energy estimates necessary to establish the stability of the semi-discrete scheme in Proposition~\ref{prop: stability}.
\item In Section~\ref{sec: convergence}, we present the main Theorem of this article where convergence for the spatial semi-discretization is established.
\end{itemize}

\subsection{Notation for evolving surfaces}
We follow the usual notation for surfaces and geometric partial differential equations from~\cite{BGN2020275}.
We consider surfaces $\Gammat$ for which there exists a $C^k$-map $\X:\Omega\times [0,T] \rightarrow \mathbb{R}^3$, such that $\X \argu := \X(\cdot,t)$ is a diffeomorphism for each $t\in [0,T]$ and
$$\Gammat=\{X(P,t): P\in \Omega\},\;\;\;\;\boundaryGt=\{X(P,t):P\in\partial \Omega\},$$
where $\Omega\subset \Rn[2]$ is a bounded domain with boundary $\partial\Omega$. 
We denote the resulting $C^k$-evolving surface as $\Gt=\underset{t\in [0,T]}{\bigcup} \Gammat\times \{t\}$.

The \textit{velocity} $\vel$ of $\Gammat$ is defined on $\Gt$ by 	
\begin{equation}\label{def: vel}
		\vel(\X(P,t),t)=\partial_t \X(P,t),\;\;\; \forall (P,t)\in \Omega\times[0,T] .
		\end{equation}

Given a smooth evolving surface $\Gt$, given by a global $C^1$ parametrization $\X:\Omega\times [0,T] \rightarrow \mathbb{R}^3$, and a smooth function $f$ defined on $\Gt$, the material derivative of $f$ is defined for $x=\X(P,t)\in \Gammat$, with $P\in\Omega$ and $0\leq t\leq T$ as
\begin{equation}
(\dermat f )(\X(P,t),t)=\partial_t f (\X(P,t),t). \label{def: dermat}
\end{equation}

When no confusion arises, we denote by $\Gamma$ one surface $\Gammat$ for a given $t$, omitting the variable $t$. 
If $\Gamma$ is a piecewise $C^1$ surface, we let $\gradg f$ ($\gradg \bs{f}$) be the tangential gradient of a scalar (vector-valued) function $f$ ($\bs{f}$), with $\bs{e}_i^T\gradg \bs{f}=\gradg (\bs f \cdot \bs{e}_i)^T$ for $i=1,2,3$; i.e., each row of $\gradg \bs{f}$ is the transpose of the surface gradient of each component $\bs{f}\cdot e_i$ of $\bs{f}$. We denote by $\Divg \bs{f}$ the surface divergence of a vector-valued function $\bs{f}$, and by $\Delta_{\Gamma} f=\Divg \gradg f$ the Laplace-Beltrami operator applied to a scalar function $f$; see \cite{BDN20201} for these notions.

We choose a smooth unit normal field $\n\argu:\Gt\rightarrow\mathbb{R}^3$ such that $\n(x,t)$ is a unit normal to $\Gammat$ at $x\in\Gammat$ for all $t\in [0,T]$. 
For $x\in\Gamma = \Gammat$, the map $A(x):=\gradg\n(x)\in\mathbb{R}^{3\times 3}$ is the Weingarten map at $x$; and we let $|A|$ denote the Frobenius norm of $A$. 
The matrix $A(x)$ is symmetric, one of its eigenvalues is zero (because $A(x)\n(x)=0$) and the others $\kappa_1(x)$ and $\kappa_2(x)$ are called the principal curvatures of $\Gamma$ at $x$, which correspond to orthogonal eigenvectors of $A(x)$ in the tangent space of $\Gamma$ at $x$. The mean curvature $\Hm$ of $\Gamma$ at $x$ is defined to be the trace of $A(x)$, i.e.,
\begin{equation*}
    \Hm(x):=\text{tr}(A(x))=\kappa_1(x)+\kappa_2(x).
\end{equation*}

\textbf{Transport formulae.} Let $\Gt$ be a $C^2$-evolving surface, let $\Omega_{\Sigma}\subseteq\Omega$ be a subdomain of $\Omega$ with a Lipschitz boundary, and let $\Sigma(t)=\X(\Omega_{\Sigma},t)$, then~\cite[Sec. 9.2, equation (9.6)]{ER2020}
\begin{equation}\label{eq: teo transp}
		\frac{d}{d t}\int_{\Sigma(t)} f = \int_{\Sigma(t)} \dermat f + f\,\textnormal{div}_{\Gammat}\, \bs{v}.
\end{equation}
If $f$ and $g$ are functions such that all the following quantities exist, we have~\cite[Sec. 9.2, equation (9.7) with $\mathcal{A}_{\Gamma}=I_3$]{ER2020}, 
\begin{equation}\label{eq: teo transp-grad-grad}
\begin{split}
	\frac{d}{d t}\int_{\Sigma(t)}\nabla_{\Gammat} f\cdot \nabla_{\Gammat} g= \int_{\Sigma(t)}&{} \nabla_{\Gammat} f\cdot \nabla_{\Gammat} \dermat f 
 +\nabla_{\Gammat} \dermat g\cdot \nabla_{\Gammat}g \\
 &+  \nabla_{\Gammat} f\cdot \mathcal{B}(\bs{v})\nabla_{\Gammat} g,
 \end{split}
\end{equation}
where $\mathcal{B}(\bs{v})=\textnormal{div}_{\Gammat}\bs{v}\,I_3\, -(\gradg \bs{v}+(\gradg \bs{v})^T)$.

\paragraph{Pullback, push-forward and transport functions.} 
Throughout this article, functions defined on the parametric domain $\Omega$ will be represented by capital letters, whereas lower case letters are reserved for functions defined on the surface $\Gamma$.
Given $f:\Gammat\rightarrow\mathbb{R}^m$, we define its \emph{pull-back} to the parametric domain as $F=f\circ\X\argu$. 
And given $F:\Omega\times[0,T]\rightarrow\mathbb{R}^m$, its \emph{push-forward} to the surface $\Gammat$ will be $f=F\circ\invXt$.  If $\Gammat$ is $C^{1,\alpha}$ for $0<\alpha\leq 1$, then we have the following chain rule \cite[Section 2.2]{BDN20201}
\begin{equation}\label{eq: chain rule Omega-Gamma}
    (\gradg f)\circ \X\argu=\nabla \X\argu G\argu^{-1}\nabla F,
\end{equation}
where $G_\Gamma(t):=(\nabla \X\argu)^T\nabla \X\argu$ is the time-dependent first fundamental form of $\Gammat$.

If the map $\X$ is sufficiently smooth in the sense that $\X\argu\in W^{r,\infty}(\Omega)$ and $\invXt\in W^{r,\infty}(\Gammat)$ for some $r\in\mathbb{N}$ and all $t\in[0,T]$, we have the following equivalence of norms between a function $f$ and its pullback $F$:
\begin{equation*}
    C^{-1}\norm{f\argu}_{W^{r,s}(\Gammat)}\leq\norm{F\argu}_{W^{r,s}(\Omega)}\leq C\norm{f\argu}_{W^{r,s}(\Gammat)},
\end{equation*}
for some constant $C=C(\X)$ and all $1\leq s\leq\infty$.

For a time independent function $F$ defined in $\Omega$, its push-forward to the surface is called \textbf{transport} and satisfies the following property
\begin{equation}\label{eq: transport property}
   \dermat f = \partial_t (f\circ \X)\circ\invXt= (\partial_t F)\circ\invXt\equiv 0.
\end{equation}

\subsection{Discrete setting: B-spline based Isogeometric Analysis}
\label{sec: splines}

Without loss of generality we assume here that the reference domain $\Omega$ is the unit square in $\R^2$, which we partition uniformly into $N^2$ squares of side length $h=1/N$, for $N\in\N$. We call this bidimensional mesh $\mathcal{M}_h$.

We fix $p\in\N$, $p\ge 2$ and given a partition of $[0,1]$ into $N$ subintervals $\{I_j\}^N_{j=1}$ of length $h=1/N$, we define the univariate spline $\Sparamunid$ of degree $p$ and smoothness $\ell\in\N_0$ by
$$\Sparamunid=\{ s \in C^{\ell}[0,1]:\, s|_{I_j}\in \mathbb{P}^p,\,\forall j=1,\dots,N\},$$
 where $\mathbb{P}^p$ is the space of polynomials of degree at most $p$. The corresponding tensor-product spline space on the parametric domain $\Omega$ of degree $ p$ and smoothness $\ell$ will be $\Sparam=\Sparamunid\otimes\Sparamunid$,
 which is spanned by tensor product B-splines $\basisfo(x,y)=B_{j_1,p}(x)B_{j_2,p}(y)$, where $\{B_{j,p}\}_{j=1}^{N_p}$ is the canonical B-spline basis of the univariate spline space.
 We present our results for two-dimensional surfaces in $\R^3$, but they naturally extend to the simpler case of curves in $\R^2$.

\medskip

 \textbf{Polynomial degree and spline global regularity.}
 From now on, we omit the subscripts $p$ and $\ell$ in the spaces because we consider them fixed. Our only assumption is $p \ge 2$ and $\ell \ge 0$. 
 The assumption $p\ge 2$ is necessary to cope with some inverse estimates in the proof of stability in Section \ref{sec: stability}.
 Furthermore, the canonical B-spline basis of $\Sparam$ will be denoted by $\{\basisfo\}_{j \in J}$, with $J$ the set of indices.

In \cite[Section 12.3]{Sch2007} a projection operator 
onto $\Sparam$ is defined as follows:
\begin{equation}\label{def: projector}
\Q: L^1 (\Omega)\rightarrow \Sparam \quad \text{given by}\quad
\Q (W)=\sum_{j\in J}\lambda_{j}(W)\basisfo,
\end{equation}
where $\lambda_{j}$ are linear functionals that define a dual basis for the B-splines basis, i.e., $\lambda_{j}(\basisfo[i])=\delta_{ij}$~\cite[Theorem 4.41]{Sch2007}. This operator is called \textit{quasi-interpolant} and has the following approximation properties.

\begin{lemma}\label{lem: globalestimatesprojector}
Let $\Q: L^1 (\Omega)\rightarrow \Sparam$ be the projector operator defined as in~\eqref{def: projector}. 
Then, if $1\le q\le \infty$, $T>0$ and $U:\Omega\times[0,T]\to \R$, satisfies $U\argu,\partial_t U\argu \in W^{l,q}(\Omega)$ for all $0\le t \le T$, we have that, for  $0\leq k \leq l \leq p+1$, if $k\le \ell+1$, 
\begin{align}\label{eq: globalapproximation}
\|U\argu-\Q U\argu\|_{W^{k,q}(\Omega)}
	&\leq C h^{l-k}\|U\argu\|_{W^{l,q}(\Omega)},
\\
\label{eq: globalapproximation dt}
\left\|\partial_t U\argu-\partial_t( \Q U\argu)\right\|_{W^{k,q}(\Omega)}
&\leq C h^{l-k}\left\|\partial_t U\argu\right\|_{W^{l,q}(\Omega)}
\end{align}
 for a constant $C$, which may depend on $k$, $l$, $q$ and $p$, but is otherwise independent of $h$ and $U$.
\end{lemma}

\begin{proof}
Estimate~\eqref{eq: globalapproximation} is well known~\cite{VBSB2014}, and~\eqref{eq: globalapproximation dt} is a consequence of the fact that the functionals $\lambda_j$ are defined in \cite{Sch2007} as local integrals $\lambda_j(W)=\int_{I_j}W \Phi_j$, for certain $I_j\subset \Omega$ and $\Phi_j\in L^\infty(I_{j})$, Whence $\partial_t \lambda_j(U\argu)=\int_{I_j} \partial_t U\argu \Phi_j$, which thereby implies that 
$\partial_t \Q U\argu
=\Q \left(\partial_t U\argu\right)$.
\end{proof}

\paragraph{Vanishing boundary values.}
We define $\Sparamcero$ as the subspace of $\Sparam$ of functions that vanish on the boundary $\partial \Omega$.
Results analogous to~\eqref{eq: globalapproximation} and~\eqref{eq: globalapproximation dt} hold if 
$U\argu,\partial_t U\argu \in W^{l,q}(\Omega)\cap W^{1,q}_0(\Omega)$, respectively.

\section{Mean curvature flow with Dirichlet boundary conditions}
\label{sec: MCF}
\label{sec: MCF with boundary}

Mean curvature flow is the evolution of a surface $\Gammat$ by the normal velocity $v$ equal to $-\Hm$, with $\Hm$ the mean curvature of $\Gamma$. If we consider the points on the surface to move along the normal direction with vector-valued velocity $\vel$, this reads as:
\begin{equation}
\label{eq:mcf with dirichlet}
		\vel =  - \Hm \n , \quad \text{on} \ \Gammat \times [0,T).
\end{equation}

In our approach we consider the mean curvature flow of the surface $\Gammat$ parametrized by the unknown function $\X(\cdot,t)$ over $\Omega\subset \mathbb{R}^2$, with a Dirichlet boundary condition. On the reference domain $\Omega$ this leads to the following equation:
\begin{equation}\label{eq:mcf on reference domain}
\begin{aligned}
    \partial_t \X &= \vel\circ \X \quad & &\text{on} \ \Omega \times [0,T) , \\
    \X\argu\rvert_{\partial \Omega}&= \X_0\rvert_{\partial \Omega}, \quad & &\text{for }t>0,
\end{aligned}
\end{equation}
which can also be written on $\Gammat$ as follows:
\begin{equation}\label{eq:mcf on surface}
\begin{split}
    \dermat \idg &= \vel \quad \text{on} \ \Gammat \times [0,T) ,\\
    \vel \argu\rvert_{\boundaryGt} &= \bs{0}, \qquad \ \text{for all}\ t>0.
\end{split}
\end{equation}

A further implication of this boundary condition is that $\dermat f=\partial_t f$ on $\boundaryGt$, for any function $f$ defined on $\partial\Gammat$.
Denoting the tangent vector by $\bs\tau:=\partial_s \idg$, where $s$ is the arc length of $\boundaryGt$, we can also conclude that
\begin{eqnarray}\label{eq: dermat tau}
    \dermat \bs \tau =0,&\ \quad \ \text{on} \ \boundaryGt &\ \text{and for all} \quad\ t>0,
\end{eqnarray}
because the $\bs \tau$ does not change during the evolution.
\subsection{Evolution equations for geometric quantities and boundary conditions}

In the interior of the surface $\Gammat$, the surface normal $\n$ and the mean curvature $\Hm$ satisfy the following evolution equations~\cite{Huisken1984}:
\begin{equation}
\label{eq: evolution equations for H and nu}
	\begin{aligned}
		\dermat\n&=\laplg\n+\abs{\gradg\n}^2\n,\\
	\dermat \Hm&=\laplg \Hm+\abs{\gradg\n}^2\Hm.\\
	\end{aligned}
	 \qquad \text{on }\;\Gammat, \;\;t\in[0,T),
\end{equation}
In \cite{Sto1996} it was shown that the mean curvature must vanish on the boundary of $\Gammat$, for all $t>0$:
\begin{equation}\label{eq: bc for k}
\begin{aligned}
    \Hm\argu\equiv 0,\ &\qquad \ \text{on} \ \partial\Gammat \times (0,T).
\end{aligned}
\end{equation}
This result can be seen as a compatibility condition at the boundary that must be satisfied between its zero velocity and its prescribed one $\vel=-\Hm\n$. 

To the best of our knowledge, we have found no references in the literature to the behaviour of the normal vector at the boundary, except the obvious one:
\begin{equation}\label{eq: bc for nu}
    \begin{aligned}
         \n\argu \cdot \bs\tau\argu  = 0 ,\ &\qquad \ \text{on} \ \partial\Gammat \times [0,T).
    \end{aligned}
\end{equation}
This condition is not sufficient for the uniqueness of solutions to a system comprised of equations~\eqref{eq:mcf with dirichlet} and~\eqref{eq: evolution equations for H and nu}.
We have developed the following result for the Weingarten mapping acting on the boundary, which is instrumental for our computational algorithm.

\begin{lemma}\label{Lem: BC for nu}
For a surface $\Gammat$ moving under mean curvature flow, with a fixed boundary, i.e., satisfying~\eqref{eq:mcf with dirichlet} and $\vel \argu\rvert_{\boundaryGt} = \bs{0}$ for all $t\in[0,T]$, we have
\begin{equation}\label{eq: bc for Dgamma nu without test}
    \partial_{\conor}\n = (\curvbdry\cdot \bs{\nu})\conor+ (\partial_{\conor}\n\cdot \btau)\btau\ \quad \text{on } \partial \Gammat,
\end{equation}
where $\bs\tau$ is the unit tangent vector to $\partial\Gammat$ at $x$ and 
$\curvbdry:=\partial_{s}\bs\tau$ is the curvature vector of $\partial\Gammat$.
As a consequence we also have
    \begin{equation}\label{eq: bc for Dgamma nu}
		\partial_{\conor}\n\cdot\testnu = (\curvbdry\cdot \bs{\nu})(\testnu\cdot\conor)\ \quad \text{for} \ \testnu\in N_x\partial \Gammat
\end{equation}
if $N_x(\partial \Gammat) = \bs\tau^\perp$ is the space of vectors which are orthogonal to $\bs\tau$ at $x$.
\end{lemma}

\begin{proof}
    At each point $x\in \boundaryGt$, we consider the 
    positively oriented orthonormal basis of $\mathbb{R}^3$ formed by 
    with $\conor=\n\times\bs\tau$ the conormal vector to $\Gammat$ on $\boundaryGt$. Then,
    \begin{equation}\label{eq: descomposition of Dnu conor}
        \partial_{\conor}\n
        =(\partial_{\conor}\n\cdot\conor)\conor +(\partial_{\conor}\n\cdot \n)\n+ (\partial_{\conor}\n\cdot \btau)\btau    = (\partial_{\conor}\n\cdot\conor)\conor+ (\partial_{\conor}\n\cdot \btau)\btau   ,
    \end{equation}    
    because $\partial_{\conor}\n\cdot\n=\gradg \n\conor\cdot\n=\conor\cdot\gradg\n\n=0$.

From~\eqref{eq: bc for k} we have
    \[
    \gradg\n\btau\cdot\btau+\gradg\n\conor\cdot\conor = \tr{\gradg\n} = \Hm  = 0.
    \]
    From~\eqref{eq: bc for nu},
    \[
    0 = \partial_s(\bs\nu\cdot\bs\tau) 
    = \partial_s\bs\nu \cdot \bs\tau + \bs\nu \cdot \partial_s\bs\tau,
    = \partial_s\bs\nu \cdot \bs\tau + \bs\nu \cdot \bs\Hm_\partial,
    \]
    which readily implies, using that $\gradg\n\btau = \partial_s\n$, that
    \begin{equation*} 
    0=\gradg\n\bs\tau\cdot\bs\tau+\gradg \n\conor\cdot\conor=\partial_s\n\cdot\bs\tau+\partial_{\conor}\n\cdot\conor=-\n\cdot\bs\Hm_\partial+ \partial_{\conor}\n\cdot\conor.
    \end{equation*}
    Therefore
    \begin{equation}\label{eq: A mu mu}
        \partial_{\conor}\n\cdot\conor=\n\cdot\curvbdry,
    \end{equation}
which together with~\eqref{eq: descomposition of Dnu conor} yields the first assertion of this lemma.   

Given $x \in \partial\Gammat$, we denote as $T_x\partial \Gammat$ and $N_x\partial \Gammat=(T_x\partial \Gammat)^{\perp}$ 
the tangent space to $\partial \Gammat$ and the normal space of $\partial \Gamma$ at $x$. 
If $\testnu\in N_x(\partial \Gammat)$ we can write $\testnu=(\testnu\cdot \conor)\conor+(\testnu\cdot \n)\n$ and then using~\eqref{eq: bc for Dgamma nu without test}, we have
\begin{equation*}
    \partial_{\conor}\n\cdot\testnu
    =(\n\cdot\curvbdry)(\conor\cdot\testnu ) +\partial_{\conor}\n\cdot \n(\n\cdot \testnu)
\end{equation*}
which immediately leads to the last assertion of this lemma.
\end{proof}

\subsection{Weak formulation}\label{sec: weak formulation}
For the weak formulation we consider~\eqref{eq:mcf with dirichlet}, the system of semilinear parabolic equations~\eqref{eq: evolution equations for H and nu} and the ODE ~\eqref{eq:mcf on reference domain} together with the boundary conditions from~\eqref{eq:mcf on surface}, \eqref{eq: bc for k}, \eqref{eq: bc for nu} and~\eqref{eq: bc for Dgamma nu}. 

The weak formulation employs standard Sobolev spaces on surfaces for the position, velocity, and mean curvature. For a Lipschitz surface $\Gammat$ in $\mathbb{R}^{3}$, we define 
$$\Hk(\Gammat)=\{ \eta \in \Lpgammat\mid \gradg \eta \in \Lpgammat^{3}\},$$
as defined in \cite[Section 2.1]{DE2007}. For sufficiently smooth surfaces, we similarly define $\Hk[k](\Gammat)$ and $W^{k,p}(\Gammat)$ for $k\in\mathbb{N}$ and $p\in [1,\infty]$. Additionally, we define  $\Hk[k]_0(\Gammat)$ and $W_0^{k,p}(\Gammat)$ in the obvious manner. For the normal vector, we define the following space,
 \begin{equation}
 \begin{split}
      \Otau:=&\left\{ \bs \phi \in H^1(\Gammat)^3: \bs \phi \cdot \bs \tau =0,\ \text{on} \ \boundaryGt \right\}\\
      =&\left\{ \bs \phi \in H^1(\Gammat)^3: \Ptau\bs \phi=\bs \phi,\ \text{on} \ \boundaryGt \right\},
 \end{split}
    \end{equation}
where $\bs\tau$ is the tangent vector of $\boundaryGt$ and $\Ptau=\bs I-\btau\tensor\btau$. 
\begin{remark}
 Note that if $\bs \phi\in \Otau$ then, $\dermat\bs\phi \in \Otau$ due to~\eqref{eq: dermat tau}.
\end{remark}

In order to obtain the weak form we multiply~\eqref{eq: evolution equations for H and nu} by appropriate test functions $\eta$, $\bs\psi$ and use Green's formula for  compact orientable $C^2$-surfaces~\cite[Remark 22]{BGN2020275}
\begin{subequations}
    \begin{alignat}{2}
    \label{eq: green H}
     \intgamma \dermat \Hm \,\testH+\intgamma \gradg \Hm \cdot \gradg \testH  &= \intgamma \abs{A}^2 \Hm \,\testH+ \intbgamma \testH \partial_{\conor}\Hm,\\
     \label{eq: green nu}
    \intgamma \dermat\n \cdot \testnu  +\intgamma \gradg\n:\gradg\bs{\psi }&= \intgamma \abs{A}^2\n \cdot \testnu+\intbgamma \testnu\cdot \partial_{\conor}\n,
    \end{alignat} 
\end{subequations}

Considering $\testH\in \Hk_0(\Gammat)$, the boundary term in~\eqref{eq: green H} is zero. For the boundary term in~\eqref{eq: green nu}, we use~\eqref{eq: bc for Dgamma nu} and test functions $\testnu\in\Otau$.
Then, by using  Lemma~\ref{Lem: BC for nu}, we get
\begin{equation*}
    \partial_{\conor}\n  \cdot\testnu=(\bs{\kappa}_{\partial \Gamma}\cdot \bs{\nu})(\testnu\cdot\conor),\quad \forall\testnu\in\Otau\quad \ \text{on} \ \boundaryGt.
\end{equation*}

We thus arrive at the following weak formulation of the problem, which is the basis for our numerical discretization.

\begin{problem}\label{weak problem MCF}
Find $\Gammat$, a sufficiently smooth surface defined by a motion $\X\argu \in H^1(\Omega)^3$, 
 $\vel\argu\in \Hk_0(\Gammat)^3$, $ \Hm \argu \in H_0^1(\Gammat)$ with $\dermat \Hm\argu\in  \Lp(\Gammat)$ and $\n\argu\in \Otau$ with $\dermat \n\argu\in \Lp[2](\Gammat)
$ such that:  
\begin{align}
 \partial_t \X &= \vel\circ \X \quad \text{on} \ \Omega,
\\
\vel&= -\Hm\n, \quad\text{on }\Gammat,\label{eq: weak velocity}\\
     \intgammat \dermat \Hm \,\testH+\intgammat \gradg \Hm \cdot \gradg \testH  &= \intgammat \abs{A}^2 \Hm \,\testH,\label{eq: weak curvature}\\
    \intgammat \dermat\n \cdot \testnu  +\intgammat \gradg\n:\gradg\bs{\psi }&= \intgammat \abs{A}^2\n \cdot \testnu+\int_{\boundaryG} (\curvbdry \cdot \n )(\conor\cdot \testnu),\label{eq: weak normal}
    \end{align} 
 $\forall\,\testH\,\in \Hk_0(\Gammat)$, $\forall\,\testnu\,\in\Otau$,  and for almost all $t\in[0,T]$.
 The initial conditions are 
 \[
 \bs X(0) = \bs X_0, \text{ on $\Omega$}, 
 \quad
 \Hm(0) = \Hm_0, \text{ on $\Gamma_0$}, 
 \quad
 \n(0) = \n_0, \text{ on $\Gamma_0$}, 
 \]
 where $\X_0, \,\Hm_0$ and $\n_0$ are, respectively, the initial mapping defining $\Gamma_0$, the mean curvature and the normal of $\Gamma_0$.
 Since we ask $\vel\argu\in \Hk_0(\Gammat)^3$, the position $\X\argu$ will satisfy the boundary condition $\X\argu|_{\partial\Omega} = \X_0|_{\partial\Omega}$.
\end{problem}

\begin{remark}
    Thanks to~\eqref{eq: bc for k} and Lemma~\eqref{Lem: BC for nu}, if the original problem has a sufficiently smooth classical solution, then it will be a weak solution according to Definition~\eqref{weak problem MCF}.
\end{remark}

\textbf{Standing regularity assumption.}
From now on we assume that, for all $t\in[0,T]$, the weak solution satisfies $\X , \partial_t X \in W^{p+1,\infty}(\Omega)$, where $p\in\N$ is the polynomial degree of the spline space defined in Section~\ref{sec: splines}. Moreover, whenever we write $A\lesssim B$ we mean $A\le C \, B$, with a constant $C$ that depends on $\max_{t\in[0,T]} \| \X\argu \|_{W^{p+1,\infty}(\Omega)} + \| \partial_t \X\argu\|_{W^{p+1,\infty}(\Omega)}$, and possibly the polynomial degree $p$ of the spline space.
Also $A\cong B$ means that $A\lesssim B$ and $B \lesssim A$.

\section{Surface approximation} \label{sec: surface approximation}

\subsection{Spline evolving surface}\label{sec:spline evolving surface}
The solution to the semidiscrete problem will determine, for each $t>0$, a mapping $\Xh\argu\in \Sparam^3$, which is an approximation to the solution $\X\argu$ and its image $\Gammaht$ approximates the surface $\Gammat$. The mapping $\Xh$ will be given by
\begin{equation*}
    \Xh(P,t)=\sum_{j\in J} \bs x_{{j}}(t)\basisfo(P), \qquad \forall P\in \Omega,  \ t>0,
\end{equation*}
where $ \bs x_{{j}}(t)\in \mathbb{R}^3$ and $\Sparam:=\text{span}\{\basisfo\}_{j\in J}$. Then, $\Gammaht=\Xh(\Omega,t)$ is the discrete surface, globally $C^{\ell}$ for all $t\in[0,T]$. Note that, for the boundary we have $\boundaryGhcero=\Xh(\partial \Omega,0)$ and $\boundaryGht=\boundaryGhcero$ for all $t$.

For each $t\in[0,T]$, we define a physical mesh $\mathcal{M}_h(t)$ on the discrete surface 
$\Gammaht$ as the image of $\mathcal{M}_h$ via the geometric mapping $\Xh\argu$, i.e., $\mathcal{M}_{h}(t):=\{K\subset \Gammaht: K=\Xh(Q,t), Q\in\mathcal{M}_h\}.$ 
Following~\cite{BBCHS2006}, we define the global mesh size for the time-dependent surface mesh $\mathcal{M}_h$ as 
$$
\hat h:=\underset{t\in[0,T]}{\max}\underset{K\in\mathcal{M}_h(t)}{\max}h_{K},
\qquad\text{with}\quad h_{K}=\norm{\nabla \Xh\argu}_{L^{\infty}(Q)}h_Q,
$$
where $h_Q=\text{diam}(Q)$, with $Q\in \mathcal{M}_h$ such that $K = \Xh(Q,t)$.
If $h_Q\simeq h_{K(t)}$, the global parametric mesh size $h:=\text{max}\{h_Q: Q\in\mathcal{M}_h\}$ and the global physical mesh size $\hat h$ satisfy $h\simeq \hat{h}$. In order to have this, we will always work under the assumption that the following assumption is always true.
\begin{assumption}[Regularity of $\Xh$]\label{ass: regularity Xh}
    There exists $h_0>0$ such that for all $h\leq h_0$ and $t\in[0,T]$, $\Xh\argu$ is a bi-Lipschitz homeomorphism. Moreover, $\Xh\argu|_{\bar{Q}}\in C^\infty(\bar{Q})$ for all $Q\in\mathcal{M}_{h}$, and $\invXh\argu|_{\bar{K}(t)}\in C^\infty(\bar{K}(t))$, for all $ K(t)\in\mathcal{M}_{h}(t)$.
\end{assumption}
Assumption~\ref{ass: regularity Xh} guarantees that there are no singularities and self-intersections in the parameterization $\Xh$ on $[0,T]$. We will show in Section~\ref{sec: relating discrete surfaces} that the above assumption holds true for $h$ small enough.

We define the discrete approximation spaces following the isogeometric approach, with the transport of the basis functions defined in the parametric domain. Thereby, for $\Gammaht$ we have
\begin{equation}
\begin{split}
   \Sh :=&{} \{f = F\circ \Xh^{-1}\argu,\;\;F\in \Sparam\} \\
   =&{} \Span \{\basisfg\argu:= \basisfo\circ \Xh^{-1}\argu,\;\;j \in J\},
\end{split}
\end{equation}
and also $\Shcero = \{f = F\circ \Xh^{-1}\argu,\;\;F\in \Sparamcero \}$.
Since we are working with evolving spline surfaces, we need a notion of discrete material derivative. We first let $\Ght=\underset{t\in [0,T]}{\bigcup} \Gammaht\times \{t\}$ and for a function $\eta_h:\Ght\to \R$ we define  
\begin{equation}
\dermat \eta_h (\Xh(P,t),t)=\partial_t \eta_h (\Xh(P,t),t), \label{def: discrete dermat}
\end{equation}
for all $\Xh(P,t)\in \Gammaht$ and $0\leq t\leq T$. 

Furthermore, since the basis functions on $\Gammaht$ are defined by the transport of the basis functions of the spline space in the parametric domain (see~\eqref{eq: transport property}), they satisfy the transport property
\begin{equation*}
    \dermat \basisfg =0.
\end{equation*}
Therefore, if $\eta_h\argu=\sum_{j \in J}{\eta}_{j}(t)\basisfg \in\Sh$ we have 
\begin{equation*}
    \dermat \eta_h(q,t)=\sum_{j \in J}\dot{\eta}_{j}(t)\basisfg(q), \qquad \forall q\in \Gammaht.
\end{equation*}
Analogously to~\eqref{def: vel}, the discrete velocity will be $\vh=\partial_t\Xh\circ\invXh$, and then
\begin{equation*}
 \vh(q,t)=\sum_{j \in J}\dot{\bs{x}}_{j}(t)\basisfg(q), \qquad  q\in \Gammaht.
\end{equation*}

\subsection{Auxiliary discrete surface $\Gammas$}

For the error analysis we need to introduce another discrete surface $\Gammas(t)$, that is obtained as the image of 
\begin{equation}\label{eq: quasi-int surf}
\Xs := \Q \X,
\qquad\text{i.e.,}\quad
\Gammas(t) = \Xs(\Omega,t),
\end{equation}
so that $\Xscero := \Xs(\cdot,0) = \Q\Xcero$ and by Lemma~\ref{lem: globalestimatesprojector} satisfies the following estimates.

\begin{corollary} 
If $\Xs\argu \in \Sparam$ is defined as in~\eqref{eq: quasi-int surf}, then, for $1\le q\le \infty$, and $t\in[0,T]$,
    \begin{align}
        \label{eq: error estimate X*}
        \norm{\X\argu-\Xs\argu}_{L^{q}(\Omega)}
        +h\norm{\X\argu-\Xs\argu}_{W^{1,q}(\Omega)}
        &\lesssim h^{p+1},\\
        \label{eq: error estimate X*0}
        \norm{\Xcero-\Xscero}_{L^q(\partial \Omega)}
        +h\norm{\Xcero-\Xscero}_{W ^{1,q}(\partial \Omega)}
        &\lesssim h^{p+1},
    \end{align}
    where the involved constants depend on the regularity of $X$.
\end{corollary}

We also define, similarly to~\eqref{def: vel},
$\vs(\Xs(P,t),t)=\partial_t \Xs(P,t)$, $\forall (P,t)\in \Omega\times[0,T]$.

In what follows, we assume that $\Xhcero=\Xscero$, so that, in particular, $\boundaryGscero=\boundaryGhcero$. In the same way that we defined $\Sh$ ($\Shcero$), we now define the discrete space $\Shs$ ($\Shscero$), which consists of discrete functions defined on this new surface (with vanishing trace). The basis functions $b_j^*$ are again defined as the transport to $\Gammast$ of the basis functions $\basisfo$ of the spline space $\Sparam$.  

\subsection{Lifts between different surfaces}

We will need to compare quantities and differential operators defined on different surfaces $\Gammat$, $\Gammast$ and $\Gammaht$. 
To do this, we proceed to define \emph{lifts} as follows: given a function defined on a surface, we first pass through its pullback to the parametric domain $\Omega$ and then define its push-forward to the new surface. Thus, for each $t\in [0,T]$, any $\funs{\phi}\argu$ defined on $\Gammast$ will be associated with the function
\begin{equation*}
    (\phi^*)^{\Gammat}\argu:=\phi^*\argu\circ \Xs\argu\circ \invX\argu,
\end{equation*}
defined in $\Gammat$. 

In the same way we can lift functions  on $\Gammaht$ to $\Gammast$, or between any two different surfaces, following the pullback and push-forward strategy. Note that $\Shs=\lifts{\Sh}$ and conversely,
$\Sh = \liftgamma{\Shs}$.

Using Lemma 14 in \cite{BDN20201}, we have the following relation between the tangential gradients of functions on $\Gammast$ and their lifts to $\Gammat$, 
\begin{align}\label{eq: chain rule gradients}
    \lifts{\gradg\liftgamma{\funs{\eta}}}=(\nabla\X G_\Gamma^{-1}(\nabla\Xs)^T)\circ \invXs\gradghs\funs{\eta},
\end{align}
where $ G_\Gamma = (\nabla \X)^T \nabla \X$ is the time-dependent metric tensor corresponding to the parametrization $\X\argu$. 

\paragraph{Lifts and norms.} 
Since $\Gammast$ is an approximation of order $p>0$ of $\Gammat$, by~\cite[Lemma 17]{BDN20201}, the Sobolev norms of a function and its lift are equivalent, independently of the mesh size $h$, if $h$ is sufficiently small. More precisely, given $k=0,1$ and $1\le s\le \infty$, there exists 
$h_0>0$ such that, for $0<h\leq h_0$, $0\le t \le T$, and any $\funss{\eta}\in W^{k,s}(\Gammat)$,
\begin{equation}\label{eq: equivalence Gs-G norms}
    \norm{\funss{\eta}}_{W^{k,s}(\Gammast)}\lesssim\norm{(\funss{\eta})^{\Gammat}}_{W^{k,s}(\Gammat)}
    \lesssim \norm{\funss{\eta}}_{W^{k,s}(\Gammast)},
\end{equation}
for all $t\in[0,T]$, where the involved constants and $h_0$ depend only on the regularity of $X$.

\subsection{Geometric perturbation errors}\label{sec: geometric error}

In this subsection we present estimates for the geometric consistency error when we approximate $\Gammat$ with its approximation $\Gammast$. To do this, we will study the error that results when the bilinear forms in Problem~\ref{weak problem MCF} are replaced by bilinear forms defined in $\Gammast$. This will produce errors between the corresponding area elements, as well as the differential operators that appear when comparing, for example, 
\begin{equation*}
    \intgammas \gradghs \funss{\eta}\cdot \gradghs \funss{\phi}
    \qquad\text{with}\qquad
    \intgamma \gradg \liftgamma{\funss{\eta}}\cdot \gradghs \liftgamma{\funss{\phi}}.
\end{equation*}
We denote the time-dependent area elements with $q_{\Gamma}$ and $q_{\Gammas}$, which are defined by $q_{\Gamma}\argu:=\sqrt{\text{det}(G_{\Gamma}\argu)}$ and $q_{\Gammas}\argu:=\sqrt{\text{det}(G_{\Gammas}\argu)}$, where $G_{\Gamma}\argu:=(\nabla \X\argu)^T\nabla \X\argu$ and $G_{\Gammas}\argu:=(\nabla \Xs\argu)^T\nabla \Xs\argu$ denote the first fundamental forms of $\Gammat$ and $\Gammast$ respectively. 

\begin{lemma}\label{Lem: Error estimates for G and q}
    There exists $h_0>0$, such that for $0<h\leq h_0$ we have the following geometric approximation estimates, for $0\le t \le T$,
    \begin{align}
        \max_{t \in [0,T]} \normLinfomega{1-q_{\Gammas}^{-1}\argu q_{\Gamma}\argu}
        &\lesssim  h^p,\label{eq: q bound}\\  
        \max_{t \in [0,T]} \normLinfomega{I-G_{\Gammas}\argu G_{\Gamma}^{-1}\argu}
        &\lesssim  h^p,\label{eq: G bound}\\
         \norm{1-q_{\boundaryGs}^{-1} q_{\boundaryG}}_{L^\infty(\partial \Omega)}
         &\lesssim h^p\label{eq: q partialG bound}.
    \end{align}
\end{lemma}

\begin{proof} 
The first two bounds are derived from the error estimates for $q$ and $G$ given for the stationary case in the perturbation theory for $C^{1,\alpha}$ surfaces presented in \cite{BDN20201}, as follows. Using \cite[Lemma 16]{BDN20201} for each $t\in[0,T]$ we have
\begin{equation*}
\begin{aligned}
    \normLinfomega{1-q_{\Gammas}^{-1}\argu q_{\Gamma}\argu}
    &\lesssim \normLinfomega{\nabla(\X\argu-\Xs\argu)},
    \\
    \normLinfomega{I-G_{\Gammas}^{-1}\argu G_{\Gamma}\argu}&\lesssim \normLinfomega{\nabla(\X\argu-\Xs\argu)},
\end{aligned}
\end{equation*}
where the hidden constant is a time-dependent stability constant $S_{\X}$, defined as in \cite{BDN20201}, which is independent of $h$ for all $h\leq h_0$ due to the properties of $\X$ and $\Xs$. Using~\eqref{eq: error estimate X*} and taking the maximum for $t$ we get the estimates~\eqref{eq: q bound} and~\eqref{eq: G bound}.

For the estimate~\eqref{eq: q partialG bound} we consider $\Xcero$ as a parametrization of the curve $\boundaryG$ and $\Xscero:=\Q \Xcero|_{\partial\Omega}$ as a parametrization of $\boundaryGs$, with $\Q$ as defined in~\eqref{def: projector}. Thus, the first fundamental forms are $|\partial_\tau\Xcero|^2$ and $|\partial_\tau\Xscero|^2$ respectively, and for the length elements we have
\begin{align*}
    \norm{1-q_{\partial\Gammas}^{-1}q_{\partial\Gamma}}_{L^\infty(\partial \Omega)}
    &= \norm{1-|\partial_\tau \Xscero|^{-1}|\partial_\tau \Xcero|}_{L^\infty(\partial \Omega)}
    \\
    &\lesssim \norm{\partial_\tau \X^0-\partial_\tau \Xscero}_{L^\infty(\partial \Omega)}
    \lesssim h^p,
\end{align*}
where we use the estimate~\eqref{eq: error estimate X*0} for $\Xscero$ on the boundary.
\end{proof}

In this paper, we use these estimates for time-dependent surfaces. Therefore, in addition to the estimates given above, we also need approximation estimates for $\partial_t G$ and for $\partial_t q$. 

\begin{lemma}
    \label{Lem: Error estimates dt G y dt q} There exists $h_0>0$ such that, if $0<h\le h_0$, we have
    \begin{align}
        \label{eq: dt G  bound}
       \max_{t \in [0,T]} \normLinfomega{\partial_t G_{\Gamma}(t)-\partial_t G_{\Gammas}(t)} &\lesssim  h^p,\\
          \label{eq: dt q  bound}
       \max_{t \in [0,T]} \normLinfomega{\partial_t q_{\Gamma}(t)-\partial_t q_{\Gammas}(t)} &\lesssim  h^p.
\end{align}
\end{lemma}

\begin{proof} 
    By the definition of $G_{\Gammas}$ and $G_{\Gamma}$.
    \begin{align*}
     \big\|\partial_t &G_{\Gamma}-\partial_t G_{\Gammas}\big\|_{\Linfomega}
     \\
     &\le 2\,\|\nabla\partial_t \X^T\nabla\X -  \nabla\partial_t\Xs^T \nabla\Xs\|_{\Linfomega}
     \\
        &\leq{}2\Big(\big\|\nabla\partial_t(\X-\Xs)^T\nabla\X\big\|_{\Linfomega}
        + \big\| \nabla\partial_t \Xs^T \nabla(\X-\Xs)\big\|_{\Linfomega} \Big)\\
         &\lesssim \left\| \nabla{\partial_t}( \X- \Xs)\right\|_{\Linfomega}+ \left\| \nabla(\X-\Xs)\right\|_{\Linfomega}
         \lesssim{} h^p,
         \end{align*}
         where we have used estimates of the interpolation error from Lemma~\ref{lem: globalestimatesprojector} and that 
         \[\max\{\normLinfomega{ \nabla\left(\partial_t \X\right)},\normLinfomega{ \nabla\left(\partial_t \Xs\right)}\}\] is uniformly bounded by interpolation properties.

         By Lemma~\ref{Lem: Error estimates for G and q} and the smoothness of $\X$, we have that, for sufficiently small $h$,
         \begin{equation}\label{eq: dt G inverse  bound} 
                     \max_{t\in[0,T]}\max\left\{\normLinfomega{G_{\Gamma}^{-1}},\normLinfomega{G_{\Gammas}^{-1}},
         \normLinfomega{\partial_t G_{\Gamma}},\normLinfomega{\partial_t G_{\Gammas}}\right\}\leq C.
        \end{equation}
         To obtain~\eqref{eq: dt q bound} we use that $
         \partial_t q_{\Gamma}=\frac{1}{2}q_{\Gamma}\,\mathrm{tr}\left( G_{\Gamma}^{-1}\partial_t G_{\Gamma}\right)$ 
 to rewrite
     \begin{align*}
         \partial_t q_{\Gamma}-\partial_t q_{\Gammas}={}&\frac{1}{2}q_{\Gamma}\,\mathrm{tr}\left( G_{\Gamma}^{-1}\partial_t G_{\Gamma}\right)-\frac{1}{2}q_{\Gammas}\,\mathrm{tr}\left( G_{\Gammas}^{-1}\partial_t G_{\Gammas}\right)\\
         ={}&\frac{1}{2}q_{\Gamma}\,\mathrm{tr}\left( G_{\Gamma}^{-1}\partial_t G_{\Gamma}\right)-\frac{1}{2}q_{\Gamma}\,\mathrm{tr}\left( G_{\Gammas}^{-1}\partial_t G_{\Gammas}\right)\\
         &+\frac{1}{2}q_{\Gamma}\,\mathrm{tr}\left( G_{\Gammas}^{-1}\partial_t G_{\Gammas}\right)-\frac{1}{2}q_{\Gammas}\,\mathrm{tr}\left( G_{\Gammas}^{-1}\partial_t G_{\Gammas}\right).
     \end{align*}
     Adding and subtracting the corresponding cross terms within the traces of the matrices, and using~\eqref{eq: dt G  bound}, \eqref{eq: dt G inverse  bound} as well as Lemma~\ref{Lem: Error estimates for G and q} one obtains the desired estimate~\eqref{eq: dt G  bound}.
 \end{proof}

\begin{lemma}\label{Lem: Error estimates matrices Eg and EBg}
There exists $h_0>0$ such that, if $0<h\le h_0$, we have
\begin{align}
    \max_{t \in [0,T]} \normLinfomega{E_{\Gammas}\argu} &\lesssim h^p,\label{eq: Egamma bound}\\   
    \max_{t \in [0,T]} \normLinfomega{E_{\Gammas}^{\mathcal{B}}\argu} &\lesssim h^p,\label{eq: EgammaB bound}
\end{align}
where $E_{\Gammas}$ and $E_{\Gammas}^{\mathcal{B}}$  are matrices defined in $\Omega$, for each $t\in[0,T]$, as 
\begin{align}
  \label{def: Egamma}
  E_{\Gammas}&= \Pi_{\Gammas}-\nabla \Xs G_{\Gamma}^{-1}(\nabla \Xs)^Tq_{\Gammas}^{-1}q_{\Gamma},\\
  \label{def: EgammaB}   E^{\mathcal{B}}_{\Gammas}&=\Pi_{\Gammas} B_{\Gammas}\Pi_{\Gammas}-\nabla \Xs G_{\Gamma}^{-1}\nabla \X^T {B_{\Gamma}}\nabla \X G_{\Gamma}^{-1}(\nabla\Xs)^T q_{\Gammas}^{-1}q_{\Gamma},
\end{align}
with $\Pi_{\Gammas}:=\nabla \Xs G_{\Gammas}^{-1}(\nabla \Xs)^T$ the projection matrix on the tangent plane to $\Gammas$,
and 
$B_\Gamma = \mathcal{B}_{\Gamma}\circ \X = \left[\mathcal{B}(\bs{v})=\textnormal{div}_{\Gammat}\bs{v}\,I_3\, -(\gradg \bs{v}+(\gradg \bs{v})^T)\right] \circ \X$,  
$B_\Gammas = \mathcal{B}_{\Gammas} = \mathcal{B}(\vs)\circ \Xs =\left[\textnormal{div}_{\Gammast}\vs\,I_3\, -(\nabla_{\Gammast} \vs+(\nabla_{\Gammast} \vs)^T)\right] \circ \Xs$,  
as defined immediately after~\eqref{eq: teo transp-grad-grad}.
\end{lemma}

\begin{proof}
    In order to prove~\eqref{eq: Egamma bound}, we rewrite $  E_{\Gammas}$ as follows
\begin{align*}
     E_{\Gammas}&=\nabla \Xs G_{\Gammas}^{-1}(\nabla \Xs)^T-\nabla \Xs G_{\Gamma}^{-1}(\nabla \Xs)^Tq_{\Gammas}^{-1}q_{\Gamma}\\
     &=\nabla \Xs (G_{\Gammas}^{-1}-G_{\Gamma}^{-1}q_{\Gammas}^{-1}q_{\Gamma})(\nabla \Xs)^T\\
     &=\nabla \Xs(G_{\Gammas}^{-1}-G_{\Gamma}^{-1}+G_{\Gamma}^{-1}-G_{\Gamma}^{-1}q_{\Gammas}^{-1}q_{\Gamma})(\nabla \Xs)^T\\
     &=\nabla \Xs(G_{\Gammas}^{-1}(I-G_{\Gammas}G_{\Gamma}^{-1})+G_{\Gamma}^{-1}(1-q_{\Gammas}^{-1}q_{\Gamma}))(\nabla \Xs)^T.
\end{align*}
To obtain the estimate, we resort to~\eqref{eq: q bound} and~\eqref{eq: G bound}, along with the fact that the following product is bounded using the same stability constant $S_{\X}$ defined in the proof Lemma~\ref{Lem: Error estimates for G and q}, which is independent of $h$:
\begin{equation}
\normLinfomega{G_{\Gamma}^{-1}}\normLinfomega{G_{\Gammas}^{-1}}\normLinfomega{\nabla \Xs}^2.
\end{equation}

In order to prove~\eqref{eq: EgammaB bound} we recall the definition of $\Pi_{\Gamma} = \nabla \X G_{\Gamma}^{-1}(\nabla \X)^T$
and using Lemmas 30 and 31 in~\cite{BGN2020275}, we have:
 \begin{align*}
    B_\Gammas
     &= \Divgs\vs\circ\Xs  I_3 - \left(\gradghs\vs + (\gradghs\vs)^T\right)\circ\Xs
     = \partial_t q_{\Gammas} I_3 - \frac{1}{2}\partial_t G_{\Gammas}, \\
     B_\Gamma
      &= \Divg\vel\circ\X  I_3 - \left(\gradg\vel + (\gradg\vel)^T\right)\circ\X
      = \partial_t q_{\Gamma} I_3 - \frac{1}{2}\partial_t G_{\Gamma}, 
    \end{align*}
This leads to the following decomposition:
 \begin{align*}
     E^{\mathcal{B}}_{\Gammas} 
     ={}& 
     \Pi_\Gammas\partial_t q_{\Gammas}\Pi_{\Gammas} 
     - \Pi_{\Gammas}\frac{1}{2}\partial_t G_{\Gammas}\Pi_{\Gammas} 
     \\
     &
     - \partial_t q_{\Gamma}\nabla \Xs G_{\Gamma}^{-1}(\nabla\Xs)^T q_{\Gammas}^{-1}q_{\Gamma} 
     \\
     & - \frac12\nabla \Xs G_{\Gamma}^{-1}\nabla \X^T 
     \partial_t G_{\Gamma}
     \nabla \X G_{\Gamma}^{-1}(\nabla\Xs)^T q_{\Gammas}^{-1}q_{\Gamma},
 \end{align*}
 after utilizing the identity $\nabla \X^T \nabla \X G_{\Gamma}^{-1} = \mathbf{I}$ and the projection matrix property of $\Pi_{\Gammas}$. By adding and subtracting the corresponding cross-terms, and using~\eqref{eq: Egamma bound} together with the error estimates from Lemmas~\ref{Lem: Error estimates for G and q} and~\ref{Lem: Error estimates dt G y dt q}, the desired estimate is established.
\end{proof}

\begin{lemma}[Geometric perturbation errors]\label{Lem: geometric perturbation} 
There exists a sufficiently small $h_0>0$, such that for $0<h\leq h_0$ we have the following bounds 
\begin{subequations}
    \begin{alignat}{1}
      \label{eq: bound m*-m}
        \left|\intgammas \funss{\eta}\funss{\phi} -\intgamma  \liftgamma{\funss{\eta}\funss{\phi}} \right|
        &\lesssim h^p\norm{\funss{\eta}}_{\Ldosgammas}\norm{\funss{\phi}}_{\Ldosgammas},
        \\
         \label{eq: bound m*-m boundary}
        \left|\intbgammas \funss{\eta}\funss{\phi}  -\intbgamma  \liftgamma{\funss{\phi}\funss{\eta}} \right|
        &\lesssim h^p\norm{\funss{\eta}}_{\Lp (\partial\Gammas)}\norm{\funss{\phi}}_{\Lp (\partial\Gammas)},
        \\
\label{eq: bound a*-a}
\bigg|\intgammas \gradghs \funss{\eta}\cdot \gradghs \funss{\phi}-\intgamma \gradg \liftgamma{\funss{\eta}}\cdot &\gradghs \liftgamma{\funss{\phi}}\bigg|
\\
\notag
&\lesssim h^p\norm{\gradghs\funss{\eta}}_{\Ldosgammas}\norm{\gradghs\funs{\phi}}_{\Ldosgammas},
\\
\label{eq: bound c*-c}
\bigg|\!\intgammas \!\!\Div _{\Gammas}\!\vs \funss{\eta}\funss{\phi}-\!\intgamma \!\Div _{\Gamma}\!\vel\, \liftgamma{\funss{\eta}\funss{\phi}}\bigg|
&\lesssim h^p\norm{\funss{\eta}}_{\Ldosgammas}\norm{\funss{\phi}}_{\Ldosgammas},
\\
\label{eq: bound b*-b}
\bigg|\intgammas\gradghs \funss{\eta} \cdot \mathcal{B}(\vs)\gradghs \funss{\phi}- 
 \intgamma\gradg &\liftgamma{\funss{\eta}}\cdot \mathcal{B}(\vel)\gradg \liftgamma{\funss{\phi}}\bigg|\\
\notag
&\lesssim h^p\norm{\gradghs\funss{\eta}}_{\Ldosgammas}\norm{\gradghs\funss{\phi}}_{\Ldosgammas}.
    \end{alignat}
\end{subequations}
for all sufficiently smooth functions $\funss{\eta}$ and $\funss{\phi}$ defined in $\Gammas$
and all $t\in[0,T]$.
The hidden constants and $h_0$ depend on the regularity of $X$.
Notice that we do not assume that $\funss{\phi}$ and $\funss{\eta}$ are necessarily in $\Shs$.
\end{lemma}
\begin{proof}
We will always consider surfaces and functions to be time-dependent, but we will omit the argument $t$ so as not to overload the notation.  
For~\eqref{eq: bound m*-m}, changing variables, using equivalence of norms between norms on $\Gammas$ and the parametric domain, and the geometric bound~\eqref{eq: q bound} we have
\begin{align*}
   \left|\intgammas \funss{\eta}\funss{\phi} -\intgamma  \liftgamma{\funss{\eta}} \liftgamma{\funss{\phi}}\right|
   &=\left|\intgammas \funss{\eta}\funss{\phi} -\intgammas \funss{\eta}\funss{\phi} \, q_{\Gamma}q_{\Gammas}^{-1}\circ \invXs\right|
   \\
   &=\left|\intgammas\funss{\eta}\funss{\phi}\left(1- q_{\Gamma}q_{\Gammas}^{-1}\right)\circ \invXs\right|\\
&\lesssim\normLdosgs{\funss{\eta}}\normLdosgs{\funss{\phi}}\normLinfomega{1- q_{\Gamma}q_{\Gammas}^{-1}}\\
   &\lesssim h^p\normLdosgs{\funss{\eta}}\normLdosgs{\funss{\phi}}.
\end{align*}
We proceed analogously for~\eqref{eq: bound m*-m boundary}, but using the estimate~\eqref{eq: q partialG bound} and Hölder's inequality with $L^\infty$-$L^1$ on the boundary.

To prove~\eqref{eq: bound a*-a} we proceed similarly, using a chain rule for the gradients as in~\eqref{eq: chain rule gradients} 
\begin{align*}
    \bigg|\intgammas \gradghs &\funss{\eta}\cdot \gradghs \funss{\phi}
    -\intgamma \gradg \liftgamma{\funss{\eta}}\cdot \gradghs \liftgamma{\funss{\phi}}\bigg|
    \\
    &=\bigg|\intgammas \gradghs \funss{\eta}\cdot \gradghs \funss{\phi}-\intgammas \lifts{\gradg \liftgamma{\funss{\eta}}}\cdot \lifts{\gradghs \liftgamma{\funss{\phi}}}q_{\Gamma}q_{\Gammas}^{-1}\bigg|\\
    &=\bigg|\intgammas \gradghs \funss{\eta}\cdot E_{\Gammas}\circ\invXs \gradghs \funss{\phi}\bigg|.
\end{align*}
The error matrix $E_{\Gammas}$, is defined as in~\eqref{def: Egamma} and is obtained using that 
\begin{equation*}
     \gradghs \funss{\eta}\cdot \gradghs \funss{\phi}=  \gradghs \funss{\eta}\cdot\Pi_{\Gammas}\gradghs \funss{\phi},
\end{equation*}
and
\begin{align*}
    &\lifts{\gradg\liftgamma{\funss{\eta}}}\cdot  \lifts{\gradg\liftgamma{\funss{\phi}}} \\
    &=(\nabla\X G_{\Gamma}^{-1}(\nabla\Xs)^T)\circ \invXs\gradghs\funs{\eta}\cdot(\nabla\X G_{\Gamma}^{-1}(\nabla\Xs)^T)\circ \invXs\gradghs\funss{\eta} 
    \\
    &=\gradghs\funss{\eta}\cdot(\nabla\Xs G_{\Gamma}^{-1}\nabla\X^T\nabla\X G_{\Gamma}^{-1}(\nabla\Xs)^T)\circ \invXs\gradghs\funss{\phi}
    \\
    &=\gradghs\funss{\eta}\cdot(\nabla\Xs G_{\Gamma}^{-1}(\nabla\Xs)^T)\circ \invXs\gradghs\funss{\phi}.
\end{align*}
We obtain the final estimate by using the equivalence of the norm of $E_{\Gammas}$ in $\Omega$ and its push-forward to $\Gammas$, together with the estimate~\eqref{eq: Egamma bound}.

For~\eqref{eq: bound c*-c}, we use a result about how the time-dependent area element evolves in time, which was proved in Lemma 31 in \cite{BGN2020275} for $C^2$-evolving surfaces. Then, we have $\partial_t q_{\Gamma}=\Div_{\Gamma}\vel\circ \X\,q_{\Gamma}$ in $\Omega$. Furthermore, by application of the same result to each element of $\Gammas$, we have $\partial_t q_{\Gammas}\circ \invXs=\Div _{\Gammas}\vs q_{\Gammas}\circ \invXs$ in each $K^*\subset \Gammas$. Then, we have
\begin{align*}
      \bigg|\intgammas \Div _{\Gammas}\vs\, &\funss{\eta}\;\funss{\phi}-\intgamma \Div _{\Gamma}\vel\, \liftgamma{\funss{\eta}} \liftgamma{\funss{\phi}}\bigg|
      \\
      &= \bigg|\intgammas \Div _{\Gammas}\vs\, \funss{\eta}\;\funss{\phi}-\intgammas \lifts{\Div _{\Gamma}\vel}\, \funss{\eta} \funss{\phi}(q_{\Gamma}q_{\Gammas}^{-1})\circ\invXs\bigg|
      \\
      & =\bigg|\intgammas \Big(\Div _{\Gammas}\vs\, - \lifts{\Div _{\Gamma}\vel}\Big)(q_{\Gamma} q_{\Gammas}^{-1})\circ\invXs\funss{\eta} \funss{\phi}\bigg|\\
    & =\bigg|\intgammas (\partial_t q_{\Gammas}- \partial_t q_\Gamma) q_{\Gammas}^{-1}\circ\invXs\funss{\eta} \funss{\phi}\bigg|,
\end{align*}

Thus, by using that $\norm{q_{\Gammas}^{-1}}$ is bounded independently of $h$,~\eqref{eq: dt q  bound} and the equivalence of norms between $\partial_t q_{\Gammas}- \partial_t q_\Gamma$ and its push-forward to $\Gammas$ we obtain the desired estimate~\eqref{eq: bound c*-c}. 

Finally, in order to show~\eqref{eq: bound b*-b}, we use again~\eqref{eq: chain rule gradients}. We will simplify the notation by referring to $\mathcal{B}(\vel)$ as $\mathcal{B}_{\Gamma}$ and $\mathcal{B}(\vs)$ as $\mathcal{B}_{\Gammas}$,
\begin{align*}
    \intgammas\gradghs \funss{\eta} &\cdot \mathcal{B}(\vs)\gradghs \funss{\phi}-\intgamma\gradg \liftgamma{\funss{\eta}}\cdot \mathcal{B}(\vel)\gradg \liftgamma{\funss{\phi}}\\
    &=\intgammas\gradghs \funss{\eta} \cdot \mathcal{B}_{\Gammas}\gradghs \funss{\phi}-\intgamma\lifts{\gradg \liftgamma{\funss{\eta}}}\cdot \lifts{\mathcal{B}_{\Gamma}}\lifts{\gradg\liftgamma{\funss{\phi}}}\\
    &=\intgammas\gradghs \funss{\eta} \cdot E^{\mathcal{B}}_{\Gammas}\circ\invXs\gradghs \funss{\phi},
\end{align*}
 where the error matrix $E^{\mathcal{B}}_{\Gammas}$ is defined as in~\eqref{def: EgammaB} and is obtained using that
 \begin{equation*}
     \gradghs \funss{\eta} \cdot \mathcal{B}_{\Gammas}\gradghs \funss{\phi}=\gradghs \funss{\eta} \cdot \Pi_{\Gammas}\mathcal{B}_{\Gammas}\Pi_{\Gammas}\gradghs \funss{\phi},
 \end{equation*}
 and~\eqref{eq: chain rule gradients} to rewrite $\lifts{\gradg \liftgamma{\funss{\eta}}}\cdot \lifts{\mathcal{B}_{\Gamma}}\lifts{\gradg\liftgamma{\funss{\phi}}}$ as
 \begin{align*}
    \gradghs \funss{\eta} \cdot(\nabla \Xs G_{\Gamma}^{-1}\nabla \X^T \lifts{\mathcal{B}_{\Gamma}}\nabla \X G_{\Gamma}^{-1}(\nabla\Xs)^T q_{\Gamma} q_{\Gammas}^{-1})\circ\invXs \gradghs \funss{\phi}.
  \end{align*}
We obtain the final estimate~\eqref{eq: bound b*-b} using again the equivalence of norms between $    E^{\mathcal{B}}_{\Gammas}$ and its push-forward to $\Gammas$ and~\eqref{eq: EgammaB bound}.
\end{proof}

\begin{remark}\label{remark: power of p}
     If we compare our geometric error estimates with those presented for ESFEM (\cite[Lemma 5.6]{K2018} or~\cite[Lemma 9.24]{ER2020}) we obtain an error of order $p$ instead of one of order $p+1$, even though we use splines of degree $p$. This is because in this work we use a generic lift $\mathcal{L}:=\X\circ\invXs$, from $\Gammas$ to $\Gamma$ instead of the lift given by the distance function. This behaviour is to be expected and is analogous to what happens, for example, with the Laplace-Beltrami operator when a generic lift is used (see~\cite[Section 4.3]{BDN20201}). However, this has no effect on the a priori $H^1$-error analysis as is the case for the Laplace-Beltrami operator. Using the distance function lift requires $\Gamma$ being $C^{p+2}$ for the optimal order of the geometric error (see~\cite[Remark 9.1]{ER2020}), i.e., at least $C^4$, if $p\geq 2$, whereas in this paper we see that using the generic lift with $\Gamma$ a $C^{p+1}$ surface gives the same order of approximation.
\end{remark}

\begin{remark}\label{remark: material velocity}
    Another difference worth noting compared to finite element methods is that, since we use the lift given by the global parameterizations of the surfaces $\mathcal{L}=\X\circ\invXs$, the velocity of material points is the same as the continuous velocity $\vel$. Therefore, we do not have two versions of the Transport Theorem on $\Gammat$, as is common in the literature; see, for example,~\cite[Section 5.3]{K2018} or~\cite[Section 8.6]{ER2020}. Thus, if we compare, for instance, the last two estimates in Lemma~\ref{Lem: geometric perturbation} with the last two estimates in Lemma 5.6 of~\cite{K2018}, we observe that here, instead of having the discrete velocity $\vel_h$ in the continuous bilinear forms (which is not the interpolation of $\vel$), we simply have $\vel$.
    
\end{remark}

\section{Approximation estimates}
\label{sec: spline approximation estimates}

\subsection{Spline approximation on the surface}
\begin{definition} \label{def: quasi-int on surface}
Recall the definition of $\Gammast$ from~\eqref{eq: quasi-int surf} with its corresponding spline space $\Shs$. For each $t>0$ we define a projector $\Qs[\Gammas]: L^2(\Gammat)\rightarrow \Shs$ as follows
\begin{equation*}\label{quasi-int on surface}
    \Qs[\Gammas](u):=\Q\big(\lifts{ u}\circ \Xs\argu\big)\circ(\Xs\argu)^{-1}.
\end{equation*}   
Where $\Q$ is the quasi-interpolant operator defined in Section~\ref{sec: splines}{sec: spline approximation estimates}. Moreover, $\Qs[\Gammat]$ is defined as the lift of $\Qs$, that is,  
$\Qs[\Gamma] u:=\liftgamma{ \Qs[\Gammast]\lifts{ u}}$.

\end{definition}

Under the assumption that $\X\argu$ and $\Xs\argu$ are sufficiently smooth for all $t \in [0,T]$, we can extend the approximation properties of the classical spline spaces $\Sparam$ to the discrete space $\Shs$. Throughout this chapter, we will assume that the spline spaces used are globally $C^\ell$ and of polynomial degree $p$.  

For the following lemma, we will require that $\X\argu$ and $\Xs\argu$ belong to $W^{k,\infty}(\Omega)$ with inverses in $W^{1,\infty}(\Omega)$, with uniform bounds.

\begin{lemma}\label{Lem: global error estimate quasi int on gamma}
    Let $0\leq |\bs \alpha|\leq k\leq p+1$ and $1\leq q\leq \infty$, and let us assume that
    \begin{multline}\label{assumption on X Xs}
        \max_{t\in[0,T]} \big( \| X\argu \|_{W^{k,\infty}(\Omega)} + \| \Xs\argu \|_{W^{k,\infty}(\Omega)}\\
        +  \| X^{-1}\argu \|_{W^{1,\infty}(\Gammat)} + \| (\Xs)^{-1}\argu \|_{W^{1,\infty}(\Gammast)} \big) \le C.
    \end{multline}
    If $u\in W^{k,q}(\Gammat)$ and $|\bs \alpha|\leq\ell$ then, for each $t\in[0, T]$,
    \begin{align}
        \label{eq: global approx on gamma}
    \|u-\Qs[\Gammat] u\|_{W^{|\bs \alpha|,q}(\Gammat)}
    &\lesssim h^{k-|\bs \alpha|}\|u \|_{W^{k,q}(\Gammat)}\\
    \label{uniform bound Q on surface}
    \text{and} \qquad
    \|\Qs[\Gammat] u\|_{W^{1,\infty}(\Gammat)}
    &\lesssim \|u \|_{W^{k,q}(\Gammat)}.
    \end{align}
    with the hidden constant depending only on $C$ from~\eqref{assumption on X Xs}. Also, as an immediate consequence,
    \end{lemma}
    \begin{proof}
        Using equivalence of norms, similar techniques to those used in Corollary 4.21 in~\cite{VBSB2014}, for each fixed $t\in[0,T]$, we have that 
        \begin{align*}
            \|u-\Qs[\Gamma] u\|_{W^{|\bs \alpha|,q}(\Gammat)}
            &\lesssim 
            \|\lifts{u}-\Qs u\|_{W^{|\bs \alpha|,q}(\Gammast)}
            \\
            &\lesssim 
            h^{k-|\bs \alpha|}\|\lifts{u}\|_{W^{k,q}(\Gammast)}
            \lesssim 
            h^{k-|\bs \alpha|}\|u\|_{W^{k,q}(\Gammat)}.
        \end{align*}
    The stability bound~\eqref{uniform bound Q on surface} is an immediate consequence of this.
    \end{proof}

    As in this article we will use spline spaces with boundary conditions, we consider the quasi-interpolant defined for functions belonging to spaces with boundary conditions.
    
    Just as we defined a quasi-interpolant for functions defined on surfaces, in this article we also work with geometric quantities defined only on the boundary of the surface. Therefore, it will be necessary to consider quasi-interpolants defined for functions that are only meaningful on the boundary. We will do this analogously to Definition~\ref{def: quasi-int on surface}.

    \begin{definition}
        Let $\boundaryGscero$ be given, with its corresponding spline space $\mathcal{S}_h^\partial$. For each $t>0$, we define a projector $\Qs[\boundaryGscero]: L^2(\boundaryG)\rightarrow \mathcal{S}_h^\partial$ as follows, for
        \begin{equation*}\label{quasi-int on boundary}
            \Qs[\boundaryGscero](u):=\Q_{\partial\Omega}\big({ u}^{\boundaryG}\circ \Xscero\argu\big)\circ(\Xscero\argu)^{-1}.
        \end{equation*}   
        where $\Q_{\partial\Omega}$ is the univariate quasi-interpolation operator defined analogously to the operator $\Q$ from~\eqref{def: projector}, but on $\partial\Omega$ instead of on $\Omega$. Additionally, $\Qs[\boundaryG]$ is defined as the lift of $\Qs[\boundaryGscero]$, that is, $\Qs[\boundaryG] u:={ \Qs[\boundaryGscero]u}^{\boundaryG}$.
    \end{definition}
    
    It is worth noting that we obtain analogous approximation estimates as for the bivariate quasi-interpolant when the function $u\in W^{k,s}(\partial\Gammat)$.

\subsection{Ritz map with vanishing traces}\label{sec: Linear Ritz projection}
We define a Ritz-like projection as in \cite{KLL2019}, but with boundary values and using splines spaces instead of finite elements.
\begin{definition}\label{def: ritz map zero trace linear Ru}
We define the modified Ritz projection $\Rs[\Gammas]^0:\Hk_0(\Gammat)\rightarrow \Shscero$ as follows: 
Given $u\in \Hk_0(\Gammat)$, let $\Rs[\Gammas]^0 u = w$, with $w$ the unique function satisfying $w \in \Shscero$ and 
	\begin{equation}\label{eq: ritz map zero trace linear Ru}
		\intgammas\gradghs w \cdot \gradghs  \funs{\eta} + w \; \funs{\eta}=\intgamma \gradg u\cdot \gradg\liftgamma{\funs{\eta}}+u\;\liftgamma{\funs{\eta}}, 
	\end{equation}
for all $\funs{\eta} \in \Shscero$. 
Moreover, we define $\Rscero[\Gamma] u$ as the lift of $\Rscero u$, that means $\Rscero[\Gamma] u:=\liftgamma{\Rscero u}$.
\end{definition}

This operator $\Rs[\Gammas]^0$ is well defined for $h$ sufficiently small, as is stated in the following lemma.

\begin{lemma}\label{Lem: stability - galerkin orthog linear Ru}
There exists $h_0 > 0$ such that, for all $0<h\le h_0$, all $t\in[0,T]$, and all $u\in \Hk_0(\Gammat)$, there exists a unique solution $\Rscero u$ of~\eqref{eq: ritz map zero trace linear Ru} satisfying
    \begin{align}
    \label{eq: stability linear Ru}
        \normHunogst{\Rscero u}&\lesssim\normHunogt{u},\\
        \label{eq: quasi-orth-galerkin linear Ru}
        \left| (\eR,  \liftg{\funs{\eta}})_{\Hunogammat}\right|&\lesssim h^p\norm{ u}_{\Hunogammat}\norm{\liftg{\funs{\eta}}}_{\Hunogammat}, 
    \end{align}
    for all $\funs{\eta} \in \Shscero$, with $\eR=u-\liftgamma{\Rscero u}=u-\Rscero[\Gamma] u$. 
\end{lemma}

The proof of this lemma is similar to those of Theorems 6.3 and 6.4 in~\cite{K2018}, where also an $L^2$-error estimate is obtained, which we do not need in this context.

\begin{proof}
The existence of a unique solution and the stability estimate~\eqref{eq: stability linear Ru} are direct consequences of Lax-Milgram theory and~\eqref{eq: equivalence Gs-G norms}. 
For the quasi-Galerkin orthogonality~\eqref{eq: quasi-orth-galerkin linear Ru}, using~\eqref{eq: ritz map zero trace linear Ru}, we have
    \begin{align*}
        (\eR,  \liftgamma{\funs{\eta}})_{\Hunogamma}
        ={}& ( u, \liftgamma{\funs{\eta}})_{\Hunogamma} - \left(\Rscero[\Gamma]  u ,  \liftgamma{\funs{\eta}}\right)_{\Hunogamma}\\
 = {}&\intgammas\gradghs\Rs[\Gammas]^0 u \cdot \gradghs  \funs{\eta} + \Rs[\Gammas]^0 u\; \funs{\eta}
 \\
 &- \intgamma \gradg\Rscero[\Gamma]  u \cdot \gradg  \liftgamma{\funs{\eta}} + \Rscero[\Gamma] u\; \liftgamma{\funs{\eta}}.
    \end{align*}
By regrouping terms and using ~\eqref{eq: bound a*-a} and~\eqref{eq: bound m*-m} together with~\eqref{eq: stability linear Ru} and~\eqref{eq: equivalence Gs-G norms} we deduce~\eqref{eq: quasi-orth-galerkin linear Ru}.
\end{proof}

In the next proposition we present the higher-order error estimates for the Ritz map from Definition~\ref{def: ritz map zero trace linear Ru}.

\begin{proposition}\label{prop: error estimate linear Ru}
Let $u$ be defined in $\Gt$ such that $u\in \Hk_0(\Gammat)\cap\Hk[p+1](\Gammat)$ for all $t\in[0,T]$. Then, there exists $h_0>0$ suth that the error in the Ritz map satisfies, for all $t\in[0,T]$ and $0<h\leq h_0$,
	\begin{equation}\label{eq: error estimate linear Ru}
		\norm{u-\Rscero[\Gamma] u}_{\Hunogammat}\lesssim h^p\norm{u}_{H^{p+1}(\Gammat)}.
	\end{equation}
\end{proposition}
\begin{proof} 
As before, we write $\eR=u-\Rscero[\Gamma] u$ and do not explicitely write the time $t$ arguments. 
From the equivalence of norms~\eqref{eq: equivalence Gs-G norms} and~\eqref{eq: quasi-orth-galerkin linear Ru}, there exists $h_0>0$ such that if $0<h<h_0$, we have
\begin{align*}
    \normcuad{\eR}_{\Hunogamma}&=(\eR, \eR)_{\Hunogamma}
    =(\eR, u-\Qs[\Gamma] u)_{\Hunogamma}+(\eR, \Qs[\Gamma] u-\Rscero[\Gamma] u)_{\Hunogamma}
    \\
    &\lesssim \norm{\eR}_{\Hunogamma}\norm{u-\Qs[\Gamma] u}_{\Hunogamma}+h^p \norm{u}_{\Hunogamma}\norm{\Qs[\Gamma] u-\Rscero[\Gamma] u}_{\Hunogamma}
    \\
    &\lesssim h^p \norm{\eR}_{\Hunogamma}\norm{u}_{H^{p+1}(\Gamma)}+h^p \norm{u}_{\Hunogamma}\norm{\Qs[\Gamma] u-\Rs[\Gamma] ^0u}_{\Hunogamma}.
\end{align*}
We have used that $\Qs[\Gamma] u-\Rscero[\Gamma] u=\liftgamma{\Qs^0 u-\Rscero u}$, and also the error estimate~\eqref{eq: global approx on gamma}.
Furthermore,
\begin{align*}
\norm{\Qs[\Gamma] u-\Rs[\Gamma] ^0u}_{\Hunogamma}
&= \norm{\Qs[\Gamma] u-u+u-\Rscero[\Gamma] u}\\
  &\leq  \norm{\Qs[\Gamma] u-u}_{\Hunogamma}+\norm{\eR}_{\Hunogamma}\\
  &\lesssim h^p\norm{u}_{H^{p+1}(\Gamma)}+\norm{\eR}_{\Hunogamma},
\end{align*}
whence
\begin{align*}
    \normcuad{\eR}_{\Hunogamma}
  \lesssim {}&h^p \norm{\eR}_{\Hunogamma}\norm{u}_{H^{p+1}(\Gamma)}+ h^{2p}\normcuad{u}_{H^{p+1}(\Gamma)}
  \\
  &+h^p\norm{u}_{\Hunogamma}\norm{\eR}_{\Hunogamma}.
\end{align*}
Finally, 
we arrive at the desired estimate using Young's inequality.
\end{proof}

\begin{lemma}\label{Lem: uniform bound linear Ru}
    If $u\in \Hk_0(\Gammat)\cap W^{p+1,\infty}(\Gammat)$ for all $t\in[0,T]$. Then,
    \begin{equation}\label{eq: uniform bound linear Ru}
        \norm{\Rs[\Gamma]^0u}_{\Wunoinf(\Gammat)}\lesssim \norm{{u}}_{{W^{p+1,\infty}(\Gammat)}}.
    \end{equation}
\end{lemma}

\begin{proof}
    As a consequence of the regularity of $u$, the inverse estimates~\cite{BBCHS2006} and the above error estimate $\Hk$, we bound $\norm{\Rs[\Gamma] ^0u}_{\Wunoinf(\Gamma)}$ as follows:
    \begin{align*}
        \norm{\Rs[\Gamma] ^0u}_{\Wunoinf(\Gamma)}\leq {}& 
        \norm{\Rs[\Gamma]^0u-\Qs[\Gamma]{u}}_{\Wunoinf(\Gamma)}+\norm{\Qs[\Gamma]{u}-{u}}_{\Wunoinf(\Gamma)}
        +\norm{{u}}_{\Wunoinf(\Gamma)}\\
        \lesssim{} & h^{-1}\norm{\Rs[\Gamma] ^0u-\Qs[\Gamma]{u}}_{\Hk(\Gamma)}+(1+h^{p+1})\norm{{u}}_{W^{p+1,\infty}(\Gamma)}\\
        \lesssim{} &h^{-1}\norm{\Rs[\Gamma] ^0u-{u}}_{\Hk(\Gamma)}+h^{-1}\norm{{u}-\Qs[\Gamma]{u}}_{\Hk(\Gamma)}+\norm{{u}}_{W^{p+1,\infty}(\Gamma)}\\
        \lesssim{} &(1+h^{p-1})\norm{{u}}_{W^{p+1,\infty}(\Gamma)} .
    \end{align*}
    Since $p\geq 1$, we have~\eqref{eq: uniform bound linear Ru}.
\end{proof}

In general, it is not true that $\dermat \Rscero u=\Rscero\dermat u$, but we have the following result, which is in the spirit of~\cite[Theorem 6.4, part (a)]{K2018} and~\cite[Lemmas 3.9, 3.10]{ER2020}, except that, as we pointed out in Remark~\ref{remark: material velocity}, in those works the discrete lift velocity is present, which is not necessary here.

\begin{lemma}\label{Lem: stability - galerkin orthog dermat linear Ru}
    Let $u$ be defined in $\Gt$ with $u\in \Hk_0(\Gammat)$ and $\dermat u\in \Hk_0(\Gammat)$ for all $t\in[0,T]$. We have that $\dermat \Rscero u\in \Shscero$ satisfies, for all $\testHhs\in \Shscero$,
    \begin{equation}\label{eq: def dermat linear Ru}
          (\dermat \Rscero u,\testHhs)_{\Hunogammas}=(\dermat u,\liftgamma{\testHhs})_{\Hunogamma}+ d_{\Gamma}(\vel;u,\liftgamma{\testHhs})- d_{\Gammas}(\vs;\Rscero u,\testHhs),
    \end{equation}
     where
\begin{align*}
    d_{\Gamma}(\vel;u,\liftgamma{\testHhs})&:=\intgamma \Div _{\Gamma}\vel\, u\;\liftgamma{\testHhs}+\intgamma\gradg  u \cdot \mathcal{B}(\vel)\gradg \liftgamma{\testHhs},\\
    d_{\Gammas}(\vs;\Rscero u,\testHhs)&:=\intgammas \Div _{\Gammas}\vs\, \Rscero u\;\testHhs+\intgammas\gradghs \Rscero u \cdot \mathcal{B}(\vs)\gradghs \testHhs.
\end{align*}

Also, 
for all $h$ sufficiently small, and $t\in [0,T]$ we have
\begin{equation}\label{eq: stability dermat ritz kappa}
     \normHunogst{\dermat \Rscero u} \lesssim \normHunogt{u}+\normHunogt{\dermat u},
\end{equation}
and if $\eR=u-\Rscero[\Gamma] u$, for all $\testHhs\in \Shscero$,
\begin{equation}\label{eq: quasi-orth-galerkin dermat ritz kappa}
        \left|(\dermat \eR,\liftg{\testHhs})_{\Hunogammat}\right|
        \lesssim h^p\left(\normHunogt{u}+\normHunogt{\dermat u}\right)\normHunogt{\liftg{\testHhs}}.
\end{equation}
\end{lemma}

\begin{proof}
Applying the Transport formulae~\eqref{eq: teo transp}--\eqref{eq: teo transp-grad-grad} on both sides of~\eqref{def: ritz map zero trace linear Ru}, we obtain 
\begin{multline*}
(\Rscero u,\dermat\testHhs)_{\Hunogammas}+(\dermat \Rscero u,\testHhs)_{\Hunogammas}+ d_{\Gammas}(\vs;\Rscero u,\testHhs)\\
=(u,\dermat\liftgamma{\testHhs})_{\Hunogamma}+(\dermat u,\liftgamma{\testHhs})_{\Hunogamma}+ d_{\Gamma}(\vel;u,\liftgamma{\testHhs}).
\end{multline*}

The definition of the Ritz map~\eqref{eq: ritz map zero trace linear Ru} implies that $(\Rs u,\dermat\testHhs)_{\Hunogammas} = (u,\dermat\liftgamma{\testHhs})_{\Hunogamma}$, for all $\testHhs\in \Shscero$, i.e., the first terms on both sides of the previous equality coincide and cancel each other, leading directly to~\eqref{eq: def dermat linear Ru}.

In order to prove the stability bound, we recall that from the definition of $\vs$ as the quasi-interpolant of the exact velocity $\vel$, for sufficiently small $h>0$ and $t\in[0,T]$, we have
\begin{equation}\label{eq: bound for d_Gammas}
\begin{split}
       \left| d_{\Gammas}(\vs;\Rscero u,\testHhs)\right|&\leq \normWunost{\vs}\normHunogst{ \Rscero u}\normHunogst{\testHhs}\\
       &\lesssim \normHunogt{ u}\normHunogst{\testHhs},
       \end{split}
       \end{equation}
       and
       \begin{align*}
    \left| d_{\Gamma}(\vel;u,\liftgamma{\testHhs})\right|&\leq \normWunogt{\vel}\normHunogt{\, u}\;\normHunogt{\liftgamma{\testHhs}}\\
    &\lesssim \normHunogt{u}\;\normHunogst{\testHhs},
\end{align*}
where we have also used the bound~\eqref{eq: stability linear Ru} of the Ritz map. 
Testing with $\testHhs=\dermat \Rscero u$ in~\eqref{eq: def dermat linear Ru} we obtain the desired stability bound~\eqref{eq: stability dermat ritz kappa}.

Finally, to prove~\eqref{eq: quasi-orth-galerkin dermat ritz kappa} we use~\eqref{eq: def dermat linear Ru} to write
\begin{align*}
      (\dermat \eR,\liftgamma{\testHhs})_{\Hunogamma}={}&
      (\dermat u,\liftgamma{\testHhs})_{\Hunogamma}-
      (\dermat \Rscero[\Gamma] u ,\liftgamma{\testHhs})_{\Hunogamma}
      \\
      ={}&(\dermat \Rscero u,\testHhs)_{\Hunogammas}- (\dermat \Rscero[\Gamma] u ,\liftgamma{\testHhs})_{\Hunogamma}
      \\
      &+ d_{\Gammas}(\vs;\Rscero u,\testHhs)- d_{\Gamma}(\vel;u,\liftgamma{\testHhs}).
\end{align*}
We now recall that $\liftgamma{\dermat \Rscero u}=\dermat \Rscero[\Gamma] u $, 
and use~\eqref{eq: bound a*-a},~\eqref{eq: bound m*-m} and~\eqref{eq: stability dermat ritz kappa} to obtain
\begin{align*}
        \Big|(\dermat \eR,\liftgamma{\testHhs})_{\Hunogamma}\Big|
        \lesssim {}& h^p\normHunogs{\dermat \Rscero u}\normHunogs{\testHhs}
        \\
        &+| d_{\Gammas}(\vs;\Rscero u,\testHhs)- d_{\Gamma}(\vel;u,\liftgamma{\testHhs})|\\
        \lesssim{} & h^p\left(\normHunogt{u}+\normHunogt{\dermat u}\right)\normHunog{\liftgamma{\testHhs}}\\
        &+| d_{\Gammas}(\vs;\Rscero u,\testHhs)- d_{\Gamma}(\vel;u,\liftgamma{\testHhs})|.
\end{align*}
In order to bound the last term, we recall the bound~\eqref{eq: bound for d_Gammas} of $d_{\Gammas}$, and the error estimate for $\eR$ from Proposition~\ref{prop: error estimate linear Ru}:
\begin{align*}
    \big|d_{\Gamma}(\vs;&\Rscero u,\testHhs)-d_{\Gammas}(\vel;u,\liftgamma{\testHhs})\big|\\
    &=  | d_{\Gammas}(\vs;\Rscero u- \lifts{u} +\lifts{u},\testHhs)- d_{\Gamma}(\vel;u,\liftgamma{\testHhs})|\\
    &\leq  | d_{\Gammas}(\vs;\lifts{\eR},\testHhs)|+| d_{\Gammas}(\vs;\lifts{u},\testHhs)- d_{\Gamma}(\vel;u,\liftgamma{\testHhs})|\\
    & \lesssim \normHunog{\eR}\normHunog{\liftgamma{\testHhs}}+| d_{\Gammas}(\vs;\lifts{u},\testHhs)- d_{\Gamma}(\vel;u,\liftgamma{\testHhs})|\\
    & \lesssim h^p\norm{u}_{\Hk[p+1](\Gamma)}\normHunog{\liftgamma{\testHhs}}+h^p\normHunog{u}\normHunog{\liftgamma{\testHhs}},
\end{align*}
where in the last step we have used together~\eqref{eq: bound b*-b} and~\eqref{eq: bound c*-c}. 
\end{proof}

Now, we present an $\Hk$-error estimate for the material derivative of the Ritz projection.
\begin{proposition}\label{prop: error estimate dermat linear Ru}
 Let $u$ be defined in $\Gt$ with $u\in  \Hk_0(\Gammat)\cap \Hk[p+1](\Gammat)$ and $\dermat u\in H^{p+1}(\Gammat)$ for all $t\in[0,T]$. Then, there exists $h_0>0$, such that the error in the Ritz map satisfies
		\begin{equation}\label{eq: error estimate dermat linear Ru}
		\norm{\dermat u-\dermat \Rscero[\Gamma]u}_{\Hunogammat}\lesssim h^p\left(\norm{u}_{H^{p+1}(\Gammat)}+\norm{\dermat u}_{H^{p+1}(\Gammat)}\right),
		\end{equation}
   for all $t\in[0,T]$ and $0<h\leq h_0$.
\end{proposition}
\begin{proof}
We start by adding and subtracting $ \Qs[\Gamma]\dermat u$ in the definition of $\Hunogammat$-norm of $\dermat\eR=\dermat u-\dermat \Rscero[\Gamma]u$
\begin{align*}
     \normcuad{\dermat\eR}_{\Hunogamma}
     ={}&(\dermat\eR,\dermat u-\Qs[\Gamma]\dermat u)_{\Hunogamma}+(\dermat\eR,\Qs[\Gamma]\dermat u-\dermat\Rscero[\Gamma]u)_{\Hunogamma}
     \\
     \leq{}&  \normHunog{\dermat\eR} \normHunog{\dermat u-\Qs[\Gamma]\dermat u}
     \\
     &+ (\dermat\eR,\Qs[\Gamma]\dermat u-\dermat\Rscero[\Gamma]u)_{\Hunogamma}\\
     \lesssim{}&  \normHunog{\dermat\eR} h^p\norm{\dermat u}_{\Hk[p+1](\Gamma)}
     \\
     &+(\dermat\eR,\Qs[\Gamma]\dermat u-\dermat\Rscero[\Gamma]u)_{\Hunogamma},
\end{align*}
where we have used the error estimate~\eqref{eq: global approx on gamma} for $ \Qs[\Gamma]\dermat u$. 
We recall that $\Qs[\Gamma]\dermat u-\dermat\Rscero[\Gamma]u=\liftgamma{\Qs \dermat u-\dermat \Rscero u}$, 
and take $\testHhs=\Qs \dermat u-\dermat \Rscero u$ in~\eqref{eq: quasi-orth-galerkin dermat ritz kappa}, 
to arrive at
\begin{multline*}
        (\dermat \eR,\Qs[\Gamma]\dermat u-\dermat\Rscero[\Gamma]u)_{\Hunogamma}
        \\
        \lesssim h^p\left(\normHunogt{u}+\normHunogt{\dermat u}\right)\normHunog{\Qs[\Gamma]\dermat u-\dermat\Rscero[\Gamma]u}.
\end{multline*} 
We now bound 
\begin{align*}
\normHunog{\Qs[\Gamma]\dermat u-\dermat\Rscero[\Gamma]u} 
&\le \normHunog{\Qs[\Gamma]\dermat u-\dermat u} + \normHunog{\dermat u-\dermat\Rscero[\Gamma]u} 
\\
&\lesssim h^p \normHunogt{\dermat u} + \norm{\dermat\eR}_{\Hunogamma}.
\end{align*}
Summing up and using Young's inequality, 
we arrive at the desired estimate.
\end{proof}

We conclude this section with the following uniform bound for $\dermat\Rscero[\Gamma]u$, 
which can be proved analogously to Lemma~\ref{Lem: uniform bound linear Ru}.

\begin{lemma}\label{Lem: uniform bound linear dermat Ru}
    If
    $\dermat {u}\in \Hk_0(\Gammat)\cap W^{p+1,\infty}(\Gammat)^3$
    for all $t\in[0,T]$. Then, for all $\ell\geq 1$ we have
    \begin{equation}\label{eq: uniform bound linear dermat Ru}
        \norm{\dermat\Rscero[\Gamma]u}_{\Wunoinf(\Gammat)}
        \lesssim \norm{\dermat{u}}_{{W^{p+1,\infty}(\Gammat)}}.
    \end{equation}
\end{lemma}

\subsection{Nonlinear Ritz map with traces orthogonal to $\btau$}\label{sec: Nonlinear Ritz projection}

Recall the definition of $\Otau$ from Section~\ref{sec: weak formulation} which appeared in the weak formulation:
 \begin{equation}\label{eq: Otau}
 \begin{split}
      \Otau:=&\left\{ \bs \phi \in H^1(\Gammat)^3: \bs \phi \cdot \bs \tau =0,\ \text{on} \ \boundaryGt \right\}\\
      =&\left\{ \bs \phi \in H^1(\Gammat)^3: \Ptau\bs \phi=\bs \phi,\ \text{on} \ \boundaryGt \right\},
 \end{split}
    \end{equation}
where $\bs\tau$ is the tangent vector of $\boundaryGt$ and $\Ptau=\bs I-\btau\tensor\btau$. 

We need to define a suitable discrete space to approximate this space $\Otau$. 
Due to the Dirichlet boundary condition that we assume, we can use for all $t>0$ the same approximation $\partial\Gamma_h^0$ to $\partial \Gamma^0$ and thus, to its tangent vector. Let $\btauh\argu:=\btau_{h,0}$ be an approximation of the initial tangent vector $\btau$, and let $\btauhunit:=\dfrac{\btauh}{|\btauh|}$. We propose the following discretely orthogonal space
\begin{equation}\label{def:Otauh}
    \Otauh:= \left\{ \bs\phi_h \in \Sh^3: \int_{\partial\Gammah^0}(\bs\phi_h\cdot \btauhunit)w_h=0, \,\text{for all}\ w_h\in \mathcal{S}_{h}^\partial  \right\}
\end{equation}
where we enforce, weakly, the orthogonality constraint on the spline space, as is usually done for certain saddle-point problems. 
Details about the implementation will be presented later in Section~\ref{sec: num exp}.
Analogously, we define
\begin{equation}\label{eq:otauhs}
    \Otauhs:= \left\{ \funs{\bs\phi} \in \Shs^3: \int_{\boundaryGhcero}( \funs{\bs\phi}\cdot \btauhunit) w_h=0, \,\text{for all}\,  w_h\in \mathcal{S}_{h}^\partial \right\}.
\end{equation}
We have used here, in this definition, that $\boundaryGhcero=\boundaryGscero$, which also implies that $\{ v|_{\boundaryGhcero} : v \in \Sh \}=\{v|_{\boundaryGscero} : v\in\Shs\}$. 

We also define the space of traces of functions in $\Otauh$, $\Otauhs$, as follows:
\begin{equation}\label{eq: Otauhb}
\begin{split}
    \Otauhb &= \{ \bs\phi_h\big|_{\boundaryGhcero} : \bs\phi_h \in \Otauh \},\\
    \Otauhsb &= \{ \funs{\bs\phi}\big|_{\boundaryGhcero} : \funs{\bs\phi} \in \Otauhs \},
\end{split}
\end{equation}
and note that $\Otauhb=\Otauhsb$.

The discrete analog of $\Ptau$ is defined as an $\Lp(\boundaryGhcero)$-orthogonal projection onto the discrete space $\Otauhsb$, as follows: $\Ps(\btauhunit): \Hk(\boundaryGhcero)^3\rightarrow \Otauhsb$, is defined, for $\bs u\in \Hk(\boundaryGhcero)$ as $\Ps(\btauhunit)\bs{u}=\bs w$, with $\bs w$ the unique solution to
\begin{equation}\label{eq: def L2 boundry orthog proj}
\bs w\in\Otauhsb:\quad 
    \int_{\boundaryGhcero}\bs w\cdot \funs{\bs{z}}=\int_{\boundaryGhcero} \bs{u}\cdot \funs{\bs{z}},
\quad \forall \funs{\bs{z}}\in \Otauhsb.
\end{equation}
Clearly, we have the following stability bounds for $\Ps$:
\[
\norm{\Ps(\btauhunit)\bs{u}}_{\Lp (\boundaryGhcero)} 
\lesssim 
\norm{\bs{u}}_{\Lp (\boundaryGhcero)},
\quad 
\norm{\Ps(\btauhunit)\bs{u}}_{\Hk (\boundaryGhcero)} 
\lesssim 
\norm{\bs{u}}_{\Hk (\boundaryGhcero)},
\]
where for the second one an inverse inequality must be used.

This operator will be key to obtaining error estimates both for the nonlinear Ritz projection, presented in Definition~\ref{def: non linear Ru trace in Otauh}, and for its material derivative. This is because, in both Proposition~\ref{prop: error estimate non-linear Ru} and Proposition~\ref{prop: error estimate dermat Ritz proj on boundary}, obtaining error estimates requires introducing another discrete function that is ``close'' to the function we aim to approximate while simultaneously belonging to the discrete space $\Otauh$.  

In Section~\ref{sec: Linear Ritz projection}, we were able to use the quasi-interpolant because we worked with discrete function spaces that vanish on the boundary. However, this is not the case in this section, where we cannot proceed in the same manner since we do not have a guarantee that $\Qs \bs{u}$ belongs to the space $\Otauh$, even if $\bs{u} \in \Otau$.  

In the proofs of Proposition~\ref{prop: error estimate non-linear Ru} and Proposition~\ref{prop: error estimate dermat Ritz proj on boundary} below, we will follow the following strategy: given $\bs{u} \in \Otau$, we first project its trace onto a suitable discrete space. In our case, we choose the space $\Otauhsb$, and then obtain the discrete projection of the trace of $\bs{u}$ using the operator $\Ps(\btauhunit)$. In the next step, we define the harmonic extension of $\Ps(\btauhunit)\lifts{\bs{u}}$ in $\Hk(\Gamma)^3$ and denote it as $\tilde{\bs{u}}$. Finally, we consider the quasi-interpolant of this extension as the theoretical tool that allows us to obtain a discrete function with good approximation properties that also belongs to $\Otauh$. 

In addition to its definition, the following approximation estimates will be very important, the proof of which follows by using techniques similar to those presented in~\cite[Lemma 5.1 and Lemma 5.2]{AFKL2021}, and we thus omit here.

\begin{lemma}\label{Lem: L2 orth proj on boundary}
    If $\btau\in W^{p+1,\infty} (\boundaryGs)^3$, for all $\bs u\in \Hk[p+1](\boundaryG)^3\cap W^{1,\infty}(\boundaryG)^3$,
        \begin{multline*}
            \norm{\Ps(\btauhunit)\lifts{\bs{u}}-\mathcal{P}(\btauhunit)\lifts{\bs{u}}}_{\Lp (\boundaryGs)}
            +h\norm{\Ps(\btauhunit)\lifts{\bs{u}}-\mathcal{P}(\btauhunit)\lifts{\bs{u}}}_{\Hk(\boundaryGs)}
            \\
            \lesssim h^{p+1}(\norm{\bs{u}}_{\Hk[p+1](\boundaryG)}+\norm{\bs{u}}_{W^{1,\infty}(\boundaryG)}),
        \end{multline*}
     where the involved constant depends on $\norm{\btau}_{W^{p+1,\infty}(\boundaryGs)}$.
\end{lemma}

We now define a \emph{nonlinear} Ritz projection, necessary to bound the \emph{defect} of the normal $\normLdosgs{\dn}$, which is defined later in~\eqref{eq: weakh*-normal}. If we use just a \emph{linear} Ritz projection like in Section~\ref{sec: Linear Ritz projection} we can only obtain an estimate for $\norm{\dn}_{H_h^{-1}(\Gammas)}$, due to the boundary terms in the weak formulation. The latter is not enough to get an estimate for $\norm{\dermat\dn}_{L^2(\Gammas)}$. And this last bound is crucial to obtain the stability of the semi-discrete scheme.

\begin{definition}\label{def: non linear Ru trace in Otauh}
  For each $\bs u\in \Otau$, and for $\lambda>0$ sufficiently large (see Lemma~\ref{Lem: stability non linear Ru} below) we define $\Rs[\Gammas]\bs u = \bs w$ as the unique element in $\bs w \in \Otauhs$ that satisfies, for all $\testnuhs\in \Otauhs$,
  \begin{equation}\label{eq: non linear Ru trace in Otauh}
    \begin{split}
    \intgammas&\gradghs\bs w : \gradghs  \testnuhs 
    + \lambda\, \bs w\cdot \testnuhs 
    - \int_{\boundaryGhcero}\bs w\cdot \curvbdryh\,(\bs w\times \btauh)\cdot \testnuhs
    \\
      ={}&\intgamma \gradg \bs u: \gradg\liftgamma{\testnuhs}
      +\lambda\,\bs u\cdot\liftgamma{\testnuhs} 
      -\int_{\boundaryG}\bs u\cdot \curvbdry\,( \bs u\times \btau)\cdot \liftgamma{\testnuhs} .
	\end{split}
\end{equation}
\end{definition}

Nonlinear Ritz maps for parabolic equations with nonlinear boundary conditions in stationary domains were investigated by Douglas and Dupont in \cite{DD1973}; following some ideas from this paper, we extend this approach to the case of parabolic equations with nonlinear boundary conditions on evolving surfaces.

In order to prove that this nonlinear projection is well defined, we first need the following Lemma,
which follows from the properties of the quasi-interpolant and the regularity of the functions $\btau$ and $\curvbdry$.

\begin{lemma}\label{Lem: def and estimate tauh normalbh}
    Let $\btau$ and $\curvbdry$ in $W^{p+1,\infty}(\boundaryG)^3$, $\btauh:=\Qs[\boundaryGhcero]\btau$ and $\curvbdryh:=\Qs[\boundaryGhcero] \curvbdry$, then there exists $h_0>0$ such that
    \begin{align}
        \label{eq: tangent normalb quasi-int error}
        \norm{(\btau)^{\boundaryGhcero}-\btauh}_{\Lp(\boundaryGhcero)}&\lesssim h^{p+1},
        &\norm{\lifts{\curvbdry}-\curvbdryh}_{\Lp(\boundaryGhcero)}&\lesssim h^{p+1},
    \end{align}
    for all $0<h\le h_0$. Furthermore, from these estimates we get, for $0 < h \le h_0$,
    \begin{align}
        \label{eq: uniform bound tangent normalb}
        \norm{\btauh}_{\Lp[\infty](\boundaryGhcero)}&\lesssim C,
        &\norm{\curvbdryh}_{\Lp[\infty](\boundaryGhcero)}&\lesssim C.   
    \end{align}
\end{lemma}

We are now in a position to state and prove that the nonlinear projection $\Rs\bs u$ is well defined for a sufficiently large $\lambda>0$.

\begin{lemma}\label{Lem: stability non linear Ru}
    There exist $\lambda>0$, $h_0>0$ and $C>0$ such that, for each $\bs u\in \Otau$, there exists a unique solution $\Rs\bs u$ of~\eqref{eq: non linear Ru trace in Otauh} which satisfies, for all $0<h\le h_0$ and all $t\in[0,T]$, the following stability bound:
    \begin{align}
    \label{eq: stability ritz nu}
        \normHunogstvec{\Rs \bs u}
        &\le C \normHunogtvec{\bs u}. 
    \end{align}
\end{lemma}
\begin{proof}
    Let $h_0>0$ be as in Lemma~\ref{Lem: def and estimate tauh normalbh} and let $0<h<h_0$.
We will first prove that there exists $\lambda>0$ independent of $h$ such that the nonlinear projection exists and is unique.
Given $\bs u\in \Otau$, we define the operator $T: \Otauhs \rightarrow \Otauhs$ as follows: for $\bs v\in \Otauhs$, let $T(\bs v)$ be the unique function in $\Otauhs$ which satisfies, for all $\testnuhs\in \Otauhs$,
\begin{equation*}
    \begin{split}
        \intgammas\gradghs&T (\bs v) : \gradghs  \testnuhs + \lambda\, T (\bs v)\cdot \testnuhs \\
          ={}&\intgamma \gradg \bs u: \gradg\liftgamma{\testnuhs}+\lambda\,\bs u\cdot\liftgamma{\testnuhs} \\
          &-\int_{\boundaryG}\bs u\cdot \curvbdry( \bs u\times \btau)\cdot \liftgamma{\testnuhs}
          +\int_{\boundaryGhcero}\bs v\cdot \curvbdryh(\bs v\times \btauh)\cdot \testnuhs .
    \end{split}
        \end{equation*}

We now observe that due to~\eqref{eq: uniform bound tangent normalb}, if
$T(\bs v_1)-T(\bs v_2)$ for 
$\bs v_1, \bs v_2 \in \Otauhs$, we have
\begin{equation*}
    \begin{split}
        \intgammas\gradghs&(T(\bs v_1)-T(\bs v_2)): \gradghs  \testnuhs + \lambda\, (T(\bs v_1)-T(\bs v_2))\cdot \testnuhs \\
          ={}&\int_{\boundaryGhcero}\bs v_1\cdot \curvbdryh(\bs v_1\times \btauh)\cdot \testnuhs 
          - \int_{\boundaryGhcero}\bs v_2\cdot \curvbdryh(\bs v_2\times \btauh)\cdot \testnuhs\\
          ={}&\int_{\boundaryGhcero}\bs v_1\cdot \curvbdryh(\bs v_1\times \btauh)\cdot \testnuhs
          -\int_{\boundaryGhcero}\bs v_1\cdot \curvbdryh(\bs v_2\times \btauh)\cdot \testnuhs \\
          &+\int_{\boundaryGhcero}\bs v_1\cdot \curvbdryh(\bs v_2\times \btauh)\cdot \testnuhs
          - \int_{\boundaryGhcero}\bs v_2\cdot \curvbdryh(\bs v_2\times \btauh)\cdot \testnuhs\\
          &\leq{}c{}\norm{\bs v_1 -\bs v_2}_{\Lp(\boundaryG)}\norm{\testnuhs}_{\Lp(\boundaryGhcero)}\\
          &\leq{}c{}\frac{1}{2}\normcuad{\bs v_1 -\bs v_2}_{\Lp(\boundaryGhcero)}+\frac{1}{2}\normcuad{\testnuhs}_{\Lp(\boundaryGhcero)},
    \end{split}
        \end{equation*}
with $c$ solely depending on the $\Lp[\infty]$-norm of $\curvbdry$ and $\btau$.

We will now use the following trace inequality:
\begin{equation}\label{eq: trace ineq}
    \norm{\bs u}_{\Lp(\boundaryG)}\leq c_T\left(\epsilon\norm{\gradg\bs u}_{\Lp(\Gamma)}+\epsilon^{-1}\norm{\bs u}_{\Lp(\Gamma)}\right),
\end{equation}
which holds for all $0< \epsilon\leq 1$ and $\bs u\in\Hk(\Gamma)^3$, with a constant $c_T$ independent of $\epsilon$ and $u$.

Together with the equivalence of norms~\eqref{eq: equivalence Gs-G norms}, this trace inequality yields:
\begin{equation*}
    \begin{split}
        \intgammas\gradghs(T(\bs v_1)-T(\bs v_2)): &\gradghs  \testnuhs + \lambda\, (T(\bs v_1)-T(\bs v_2))\cdot \testnuhs \\
          \leq{}&\frac{c\epsilon}{2}\normcuad{\gradg(\bs v_1 -\bs v_2)}_{\Lp(\Gammas)}+\frac{c}{2\epsilon}\normcuad{\bs v_1 -\bs v_2}_{\Lp(\Gammas)}\\
          &+\frac{1}{4}\normcuad{\gradg\testnuhs}_{\Lp(\Gammas)}+\normcuad{\testnuhs}_{\Lp(\Gammas)}.
    \end{split}
        \end{equation*}
Taking $\testnuhs=T(\bs v_1)-T(\bs v_2)$ and $\epsilon=\frac{1}{2c}$ we get
    \begin{multline*}
        \frac{3}{4}\normcuad{\gradghs(T(\bs v_1)-T(\bs v_2))}_{\Lp(\Gammas)}
        +(\lambda-1) \normcuad{T(\bs v_1)-T(\bs v_2)}_{\Lp(\Gammas)}\\
          \leq{}\frac{1}{4}\normcuad{\gradg(\bs v_1 -\bs v_2)}_{\Lp(\Gammas)}+c^2\normcuad{\bs v_1 -\bs v_2}_{\Lp(\Gammas)}.
    \end{multline*}
We now choose $\lambda > 3c^2+1$ so that
$c^2\leq \frac{1}{3}(\lambda-1)$, and thus
\begin{multline*}
        \frac{3}{4}\normcuad{\gradghs(T(\bs v_1)-T(\bs v_2))}+(\lambda-1) \normcuad{T(\bs v_1)-T(\bs v_2)}\\
          \leq\frac{1}{3}\left(\frac{3}{4}\normcuad{\gradg(\bs v_1 -\bs v_2)}_{\Lp(\Gammas)}+(\lambda-1)\normcuad{\bs v_1 -\bs v_2}_{\Lp(\Gammas)}\right).
        \end{multline*}
Therefore, considering the norm $\tnorm \bs v\tnorm ^2=\frac34\normcuad{\gradghs\bs v}_{\Lp(\Gammas)}+(\lambda-1)\normcuad{\bs v}_{\Lp(\Gammas)}$ in $H^1(\Gammas)$, we have that $T$ is a contraction, i.e, 
\begin{equation*}
    \tnorm T(\bs v_1)-T(\bs v_2)\tnorm \leq \sqrt{\frac{3}{4}} \tnorm \bs v_1-\bs v_2\tnorm .
\end{equation*}
Therefore, for any $\lambda > 3c^2+1$, there exists a unique fixed point of $T$ in $\Otauhs$, from which we conclude the existence and uniqueness of $\Rs \bs u$.

Stability follows from with similar arguments, maybe for a larger $\lambda$, testing with $\funs{\bs\psi}=\Rs[\Gammas] \bs u$ in the definition~\eqref{eq: non linear Ru trace in Otauh} of $\Rs$ and using again the trace inequality~\eqref{eq: trace ineq}.
\end{proof}

We now prove higher-order error estimates for $\Rs \bs u$.

\paragraph{Notation.} In what follows, we will use $\Hk[p+1]_{\partial}(\Gammat)$ to denote the set of functions $u \in \Hk[p+1](\Gammat)$ such that $u_{|\partial\Gammat} \in \Hk[p+1](\partial\Gammat)$, and also $W_{\partial}^{p+1,s}(\Gammat)$ will denote $u \in W^{p+1,s}(\Gammat)$ such that $u_{|\partial\Gammat} \in W^{p+1,s}(\partial\Gammat)$, for $1\le s \le \infty$.

\begin{proposition}\label{prop: error estimate non-linear Ru} 
Let $\bs u$ be defined in $\Gt$ such that $\bs u\in \Otau\cap\Hk[p+1]_{\partial}(\Gammat)^3$ for all $t\in[0,T]$. Then, there exists $h_0>0$, such that 
\begin{align}
\label{eq: error estimate non-linear Ru}
\norm{\bs u-\Rs[\Gamma] \bs u}_{\Hk(\Gammat)}
&\lesssim h^{p}\Big(\norm{\bs u}_{H^{p+1}(\Gammat)}+\norm{\bs u}_{H^{p+1}(\boundaryG)}\Big),
\end{align}
for all $t\in[0,T]$ and $0<h\leq h_0$, where $\Rs[\Gamma]\bs u$ is the lift of $\Rs\bs u$ to $\Gammat$.
\end{proposition} 

\begin{proof}
To prove~\eqref{eq: error estimate non-linear Ru}, we denote $\eR = \bs u-\Rs[\Gamma] \bs u$ and 
\begin{equation*}
    \normcuad{\eR}_{\Hunogammalambda}:=\normcuad{\gradg\eR}_{\Ldosgamma}+\lambda\normcuad{\eR}_{\Ldosgamma}.
\end{equation*}
We now define $\tilde{\bs{u}}\in \Hk(\Gamma)^3$ as the only function which satisfies 
$\tilde{\bs{u}}\big|_{\boundaryG}=\big(\Ps(\btauhunit)\lifts{\bs{u}}\big)^{\boundaryG}$
and $( \tilde{\bs{u}},\bs \eta)_{\Hk(\Gamma)}=(\bs u,\bs \eta)_{\Hk(\Gamma)}$ for all $\bs \eta\in \Hk_0(\Gamma)^3$,
whence 
    \begin{equation}\label{eq: H1/2 norm u-tilde u}
        \normHunog{\bs u-\tilde{\bs{u}}}
        \cong \norm{ \bs u-(\Ps(\btauhunit)\lifts{\bs{u}})^{\boundaryG}}_{\Hk[\frac{1}{2}](\boundaryG)}.
    \end{equation}
Then
\begin{align*}
\normcuad{\eR}_{\Hunogammalambda}
={}&\left(\intgamma\gradg\eR:\gradg( \bs{u}-\Qs[\Gamma]\tilde{\bs{u}})+\lambda\intgamma\eR\cdot( \bs{u}-\Qs[\Gamma]\tilde{\bs{u}})\right)\\
&+\left(\intgamma\gradg\eR:\gradg(\Qs[\Gamma]\tilde{\bs{u}}-\Rs[\Gamma]\bs{u})+\lambda\intgamma\eR\cdot( \Qs[\Gamma]\tilde{\bs{u}}-\Rs[\Gamma]\bs{u})\right)\\
={}&(I)+(II).
\end{align*} 
As we explained earlier, the idea to estimate $\eR$ is to introduce an intermediate discrete function with good approximation properties. 
We have already defined this intermediate discrete function as $\Qs[\Gamma]\tilde{\bs{u}}$; now, let us verify that it possesses good approximation properties. To this end, it is necessary to first study how close $\tilde{\bs{u}}$ is to $\bs{u}$.  

Since $\bs u\in \Otau$, $\bs u\big|_{\boundaryG} = \Ptau\bs u\big|_{\boundaryG}$, so that~\eqref{eq: H1/2 norm u-tilde u}, Lemma~\ref{Lem: L2 orth proj on boundary} and the fact that $\btauh:=\Qs[\boundaryGhcero]\btau$, imply
    \begin{align*}
        \normHunog{\bs u-\tilde{\bs{u}}}
        \leq{}& \norm{\Ptau\bs u-\big(\Ps(\btauhunit)\lifts{\bs u}\big)^{\boundaryG}}_{\Hk(\boundaryG)}
        \\
        \lesssim{}& \norm{ (\Ptau\bs u
)^{\boundaryGhcero}-\Ps(\btauhunit)\lifts{\bs u}}_{\Hk(\boundaryGhcero)}
        \\
    \le{}& \norm{ (\Ptau\bs u
    )^{\boundaryGhcero}- \mathcal{P}(\btauhunit)\lifts{\bs u}}_{\Hk(\boundaryGhcero)}
    \\
    &+\norm{ \mathcal{P}(\btauhunit)\lifts{\bs u}-\Ps(\btauhunit)\lifts{\bs{u}}}_{\Hk(\boundaryGhcero)}
    \\
    \lesssim{}& \norm{ (\btau\tensor\btau)^{\boundaryGhcero}- \btauhunit\tensor \btauhunit}_{\Wunoinf(\boundaryGhcero)}\norm{\bs u}_{\Hk(\boundaryG)}
    \\
    &+h^p\norm{\bs{u}}_{{\Hk[p+1](\boundaryG)}}\\
    \lesssim{}& h^{p}\norm{ \btau}_{{W^{p+1,\infty}(\boundaryG)}}\norm{\bs u}_{\Hk(\boundaryG)}+Ch^p\norm{\bs{u}}_{\Hk[p+1](\boundaryG)},
    \end{align*}
    Therefore, 
    \begin{equation}\label{eq: bound u minus tilde u Huno}
        \normHunog{\bs u-\tilde{\bs{u}}} \lesssim{} h^p \norm{\bs{u}}_{{\Hk[p+1](\boundaryG)}},
    \end{equation}   
    Furthermore, using similar arguments and Lemma~\ref{Lem: L2 orth proj on boundary} we deduce that
\begin{align}
    &\label{eq: bound u minus tilde u Ldosb}\norm{\bs u- \tilde{\bs{u}}}_{{\Lp(\boundaryG)}}\lesssim h^{p+1} \norm{\bs{u}}_{\Hk[p+1](\boundaryG)}.
\end{align}
From those estimates, by using the interpolation error estimate~\eqref{eq: global approx on gamma} and the stability of $\Qs[\Gamma]$, we have 
\begin{align}
    \label{eq: bound u minus Q tilde u Huno}\normHunog{\bs u- \Qs[\Gamma]\tilde{\bs{u}}}&\lesssim h^p \left(\norm{\bs{u}}_{{\Hk[p+1](\Gamma)}}+\norm{\bs{u}}_{{\Hk[p+1](\boundaryG)}}\right),\\
    \label{eq: bound u minus Q tilde u Ldosb}\norm{\bs u- \Qs[\Gamma]\tilde{\bs{u}}}_{{\Lp(\boundaryG)}}&\lesssim h^{p+1} \norm{\bs{u}}_{{\Hk[p+1](\boundaryG)}}.
\end{align}
Then, using~\eqref{eq: bound u minus Q tilde u Huno} and Young's inequality, with $\delta_1>0$ conveniently chosen,
\begin{align*}
   |(I)| ={}&\left|\intgamma\gradg\eR:\gradg( \bs{u}-\Qs[\Gamma]\tilde{\bs{u}})+\lambda\intgamma\eR\cdot( \bs{u}-\Qs[\Gamma]\tilde{\bs{u}})\right|\\
   \lesssim{}&  \norm{\eR}_{\Hunogammalambda} \norm{\bs{u}-\Qs[\Gamma]\tilde{\bs{u}}}_{\Hunogammalambda}\\
 \lesssim{} & \norm{\eR}_{\Hunogammalambda} h^p\left(\norm{\bs{u}}_{\Hk[p+1](\Gamma)}+\norm{\bs{u}}_{\Hk[p+1](\boundaryG)}\right)\\
 \leq{} & \delta_1\normcuad{\eR}_{\Hunogammalambda}+ c_{\lambda}h^{2p}\left(\normcuad{\bs{u}}_{\Hk[p+1](\Gamma)}+\normcuad{\bs{u}}_{\Hk[p+1](\boundaryG)}\right).
\end{align*}

To bound $(II)$, we first use Definition \ref{def: non linear Ru trace in Otauh} to deduce that
\begin{align*}
    (II) ={}&\intgamma\gradg\eR:\gradg(\Qs[\Gamma]\tilde{\bs{u}}-\Rs[\Gamma]\bs{u})+\lambda\intgamma\eR\cdot( \Qs[\Gamma]\tilde{\bs{u}}-\Rs[\Gamma]\bs{u})\\
    ={}&\intgamma\gradg\bs u:\gradg(\Qs[\Gamma]\tilde{\bs{u}}-\Rs[\Gamma]\bs{u})+\lambda\intgamma\bs u\cdot( \Qs[\Gamma]\tilde{\bs{u}}-\Rs[\Gamma]\bs{u})\\ 
&-\intgamma\gradg\Rs[\Gamma]\bs u:\gradg(\Qs[\Gamma]\tilde{\bs{u}}-\Rs[\Gamma]\bs{u})+\lambda\intgamma\Rs[\Gamma]\bs u\cdot( \Qs[\Gamma]\tilde{\bs{u}}-\Rs[\Gamma]\bs{u})\\ 
    ={}&\intgammas\gradghs\Rs\bs u:\gradghs(\Qs\tilde{\bs{u}}-\Rs\bs{u})+\lambda\intgammas\Rs\bs u\cdot( \Qs\tilde{\bs{u}}-\Rs\bs{u})\\ 
    &-(III)-\intgamma\gradg\Rs[\Gamma]\bs u:\gradg(\Qs[\Gamma]\tilde{\bs{u}}-\Rs[\Gamma]\bs{u})+\lambda\intgamma\Rs[\Gamma]\bs u\cdot( \Qs[\Gamma]\tilde{\bs{u}}-\Rs[\Gamma]\bs{u}),
\end{align*}
where 
\begin{align*}
(III)={}&\int_{\boundaryG}\bs u\cdot \curvbdry( \bs u\times \btau)\cdot (\Qs[\Gamma]\tilde{\bs{u}}-\Rs[\Gamma]\bs{u})\\
    &- \int_{\boundaryGhcero}\Rs[\Gammas] \bs u\cdot \curvbdryh(\Rs[\Gammas] \bs u\times \btauh)\cdot (\Qs\tilde{\bs{u}}-\Rs\bs{u}).
\end{align*}
Thus, by Lemma~\ref{Lem: geometric perturbation},~\eqref{eq: stability ritz nu} and Young’s inequality with $\delta_2>0$ conveniently chosen and~\eqref{eq: bound u minus Q tilde u Huno}
\begin{align*}
    |(II)| \lesssim{}& ch^p\norm{\Rs[\Gamma]\bs u}_{\Hunogamma}\norm{\Qs[\Gamma]\tilde{\bs{u}}-\Rs[\Gamma]\bs{u}}_{\Hunogammalambda} + |(III)|\\
    \leq{}& c h^p\normHunog{\bs u}\left(\norm{\Qs[\Gamma]\tilde{\bs{u}}-\bs u}_{\Hunogammalambda} +\norm{\bs u-\Rs[\Gamma]\bs{u}}_{\Hunogammalambda} \right) + |(III)|\\
    \leq{}& c_{\lambda}h^{2p}\left(\normcuad{\bs{u}}_{{\Hk[p+1](\Gamma)}}+\normcuad{\bs{u}}_{\Hk[p+1](\boundaryG)}\right)+\delta_2\normcuad{\eR}_{\Hunogammalambda}+ |(III)|.
\end{align*}
Now, we need to bound $|(III)|$. First, adding and subtracting $\Rs[\Gamma]\bs u$ in the first integral, and $\lifts{\bs u}$ in the second integral, we obtain 
\begin{align*}
    (III)={}&\int_{\boundaryG} \eR \cdot \curvbdry( \bs u\times \btau)\cdot (\Qs[\Gamma]\tilde{\bs{u}}-\Rs[\Gamma]\bs{u})\\
    &+\int_{\boundaryG}\Rs[\Gamma] \bs u\cdot \curvbdry(\bs u\times \btau)\cdot (\Qs[\Gamma]\tilde{\bs{u}}-\Rs[\Gamma]\bs{u})\\
    & + \int_{\boundaryGhcero}\Rs \bs u\cdot \curvbdryh(\lifts{\eR}\times \btauh)\cdot (\Qs\tilde{\bs{u}}-\Rs\bs{u})\\
    &- \int_{\boundaryGhcero}\Rs \bs u\cdot \curvbdryh(\lifts{\bs u}\times \btauh)\cdot (\Qs\tilde{\bs{u}}-\Rs\bs{u})\\
    ={}& (J_1)+(J_2)+(J_3)+(J_4).
\end{align*}
Where, using that $\Qs[\Gamma]\tilde{\bs{u}}-\Rs[\Gamma]\bs{u}=\Qs[\Gamma]\tilde{\bs{u}}-\bs{u}+\eR$, we have
\begin{align*}
    (J_1)+(J_3)
    ={}&\int_{\boundaryG} {\eR}\cdot \curvbdry( \bs u\times \btau)\cdot (\Qs[\Gamma]\tilde{\bs{u}}-\bs{u})+\int_{\boundaryG}{\eR}\cdot \curvbdry( \bs u\times \btau)\cdot \eR\\
    &+ \int_{\boundaryGhcero}\Rs[\Gammas] \bs u\cdot \curvbdryh({\lifts{\eR}}\times \btauh)\cdot (\Qs\tilde{\bs{u}}-\bs{u}).
\end{align*}
Thus, using~\eqref{eq: bound u minus tilde u Ldosb}, the regularity of $\bs{u}$, $\btau$ and $\curvbdry$ at the boundary, and that $\norm{\Rs[\Gammas] \bs u}_{\Lp[\infty](\boundaryGhcero)}\lesssim{} h^{-1}\norm{\Rs[\Gammas] \bs u}_{\Ldosgammas}\lesssim{} h^{-1}\norm{ \bs u}_{\Hunogamma}\lesssim{} h^{-1}$, Young's inequality leads to
\begin{align*}
    |(J_1)+(J_3)|\leq{}&\norm{\eR}_{\Lp(\boundaryG)}\norm{\curvbdry( \bs u\times \btau)}_{\Lp[\infty](\boundaryG)}\norm{\Qs[\Gamma]\tilde{\bs{u}}-\bs{u}}_{\Lp(\boundaryG)}\\
    &+\norm{\eR}_{\Lp(\boundaryG)} {\norm{\curvbdry( \bs u\times \btau)}_{\Lp[\infty](\boundaryG)}}\norm{\eR}_{\Lp(\boundaryG)}\\
    &+ {\norm{\Rs[\Gammas] \bs u}_{\Lp[\infty](\boundaryGhcero)}} \norm{\curvbdryh}_{\Lp[\infty](\boundaryGhcero)}\norm{\lifts{\eR}}_{\Lp(\boundaryGhcero)} 
    \\
    &\quad\times \norm{\btauh}_{\Lp[\infty](\boundaryGhcero)} \norm{\Qs\tilde{\bs{u}}-\bs{u}
    }_{\Lp(\boundaryGhcero)}\\
    \leq{}&c\norm{\eR}_{\Hunogamma}h^{p+1}\norm{\bs u}_{\Hk[p+1](\boundaryG)}+{c_1}\normcuad{\eR}_{\Lp(\boundaryG)}\\
    &+c\,{h^{-1}}\norm{\eR}_{\Hunogamma}{h^{p+1}}\norm{\bs u}_{\Hk[p+1](\boundaryG)}\\
    \leq{}&\delta_3\normcuad{\eR}_{\Hunogammalambda}+ch^{2p}\normcuad{\bs u}_{\Hk[p+1](\boundaryG)}+c_1\normcuad{\eR}_{\Lp(\boundaryG)}.
\end{align*}
Finally, adding and subtracting $\btauh^{\boundaryG}$ in the integral $(J_2)$ and $\curvbdry^{\boundaryGscero}$ in $(J_4)$, we have
\begin{align*}
    (J_2)+(J_4)
    ={}&\int_{\boundaryG}\Rs[\Gamma] \bs u\cdot \curvbdry(\bs u\times (\btau-\btauh^{\boundaryG}))\cdot (\Qs[\Gamma]\tilde{\bs{u}}-\Rs[\Gamma]\bs{u})\\
    &+\int_{\boundaryG}\Rs[\Gamma] \bs u\cdot \curvbdry(\bs u\times \btauh^{\boundaryG})\cdot (\Qs[\Gamma]\tilde{\bs{u}}-\Rs[\Gamma]\bs{u})\\
    & + \int_{\boundaryGhcero}\Rs \bs u\cdot (\curvbdry^{\boundaryGhcero}-\curvbdryh)(\lifts{\bs u}\times \btauh)\cdot (\Qs\tilde{\bs{u}}-\Rs\bs{u})\\
    & - \int_{\boundaryGhcero}\Rs \bs u\cdot (\curvbdry^{\boundaryGhcero})(\lifts{\bs u}\times \btauh)\cdot (\Qs\tilde{\bs{u}}-\Rs\bs{u}).
\end{align*}
By~\eqref{eq: bound u minus Q tilde u Ldosb} we deduce that
\begin{align*}
    \norm{\Qs[\Gamma]\tilde{\bs{u}}-\Rs[\Gamma]\bs{u}}_{\Lp(\boundaryG)}\leq h^{p+1} \norm{\bs{u}}_{{\Hk[p+1](\boundaryG)}}+\norm{\eR}_{\boundaryG}.
\end{align*}
Using arguments similar to those we have used for the bounds on $|(J_1)+(J_3)|$ and interpolation error estimates for $\btauh$ and $\curvbdryh$, 
we deduce that
\begin{align*}
    |(J_2)+(J_4)|
    &\leq{}ch^p\normHunog{\bs u}\left(h^{p+1} \norm{\bs{u}}_{{\Hk[p+1](\boundaryG)}}+\normHunog{\eR}\right) \\
    &\leq{}ch^{2p}\normHunog{\bs u}^2 + ch^{2p}\normcuad{\bs u}_{\Hk[p+1](\boundaryG)}+\delta_4\normcuad{\eR}_{\Hunogammalambda},
\end{align*}
Finally, using the trace inequality~\eqref{eq: trace ineq},
\begin{align*}
 |(III)|\leq{} & ch^{2p}\normHunog{\bs u}^2 + ch^{2p}\normcuad{\bs u}_{\Hk[p+1](\boundaryG)}+(\delta_3+\delta_4)\norm{\eR}_{\Hunogammalambda}\\
  &+c_1\epsilon\normcuad{\gradg\eR}_{{\Lp(\Gamma)}}+\dfrac{c_1}{\epsilon}\normcuad{\eR}_{{\Lp(\Gamma)}},
\end{align*}
whence, choosing $\epsilon=\frac{1}{4c_1}$,
\begin{align*}
|(II)|\leq{} & ch^{2p}\left(\normcuad{\bs{u}}_{{\Hk[p+1](\Gamma)}}+\normcuad{\bs{u}}_{\Hk[p+1](\boundaryG)}\right)\\
&+(\delta_2+\delta_3+\delta_4)\normcuad{\eR}_{\Hunogammalambda}+\frac{1}{4}\normcuad{\gradg\eR}_{{\Lp(\Gamma)}}+4c_1^2\normcuad{\eR}_{{\Lp(\Gamma)}}.
\end{align*}
Sumarizing, we have
\begin{align*}
    \normcuad{\eR}_{\Hunogammalambda}\leq{}&c_{\lambda}h^{2p}\left(\normcuad{\bs{u}}_{{\Hk[p+1](\Gamma)}}+\normcuad{\bs{u}}_{\Hk[p+1](\boundaryG)}\right)\\
    &+(\delta_1+\delta_2+\delta_3+\delta_4)\normcuad{\eR}_{\Hunogammalambda}\\
    &+\frac{1}{4}\normcuad{\gradg\eR}_{{\Lp(\Gamma)}}+4c_1^2\normcuad{\eR}_{{\Lp(\Gamma)}}.
\end{align*}
Then, choosing $\delta_1= \delta_2=\delta_3=\delta_4=\frac{1}{8}$ and $\lambda>0$ sufficiently large, we ensure that we get the desired estimate~\eqref{eq: error estimate non-linear Ru}.
\end{proof}

As a consequence of the regularity of $\bs{u}$, inverse estimates \cite{BBCHS2006} and the previous $\Hk$ error estimate, we have the following uniform bound for the non-linear projection, the proof is straightforward and thus omitted.
\begin{lemma}\label{Lem: uniform bound non-linear Ru}
    If $\bs{u}\in {W^{p+1,\infty}_\partial(\Gammat)^3}$ for all $t\in[0,T]$. Then,
    \begin{equation}\label{eq: uniform bound nonlinear Ru}
        \norm{\Rs \bs{u}}_{\Wunoinf(\Gammat)}\lesssim\left(\norm{\bs{u}}_{{W^{p+1,\infty}(\Gammat)}}+\norm{\bs{u}}_{{W^{p+1,\infty}(\boundaryG)}}\right).
    \end{equation}
\end{lemma}

We now present stability estimates for the material derivative of $\Rs[\Gammas] \bs u$, which are necessary to obtain estimates for the error $(\dermat)^{(\ell)} (\bs u- \Rs[\Gamma]\bs u)$.

\begin{lemma}\label{Lem: def stability non linear dermat Ru} 
    Let $\bs u$ be defined in $\Gt$ with $\bs u\in \Otau$ and $\dermat \bs u\in \Otau$ for all $t\in[0,T]$. Then, $\dermat \Rs \bs u\in \Otauhs$ satisfies 
    \begin{equation}\label{eq: def dermat non linear Ru}
        \begin{split}
            \intgammas&\gradghs(\dermat \Rs[\Gammas] \bs u) : \gradghs  \testnuhs + \lambda \intgammas\dermat \Rs[\Gammas] \bs u\cdot \testnuhs\\
              ={}&\intgamma \gradg \dermat\bs u: \gradg\liftgamma{\testnuhs}+\lambda \intgamma\dermat \bs u\cdot\liftgamma{\testnuhs}\\
             &+\int_{\boundaryG}\left(\dermat \bs u\cdot \curvbdry( \bs u\times \btau)+\bs u\cdot \curvbdry( \dermat \bs u\times \btau)\right)\cdot \liftgamma{\testnuhs} \\
             &-\int_{\boundaryGhcero}\left(\dermat\Rs[\Gammas] \bs u\cdot \curvbdryh(\Rs[\Gammas] \bs u\times \btauh)+\Rs[\Gammas] \bs u\cdot \curvbdryh(\dermat \Rs[\Gammas] \bs u\times \btauh)\right)\cdot \testnuhs\\
             &+ d^{\lambda}_{\Gamma}(\vel;\bs u,\liftgamma{\testnuhs})-d^{\lambda}_{\Gammas}(\vs;\Rs \bs u,\testnuhs),
        \end{split}
            \end{equation}
     where
            \begin{align*}
                &d^{\lambda}_{\Gamma}(\vel;\bs u,\liftgamma{\testnuhs})
                :=\lambda\intgamma \Div _{\Gamma}\vel\, \bs u\cdot\liftgamma{\testnuhs}+\intgamma\gradg  \bs u : \mathcal{B}\gradg \liftgamma{\testnuhs},\\
                &d^{\lambda}_{\Gammas}(\vs;\Rs \bs u,\testnuhs)
                :=\lambda\intgammas \Div _{\Gammas}\vs\, \Rs \bs u\cdot\testnuhs+\intgammas\gradghs \Rs \bs u : \mathcal{B}^*\gradghs \testnuhs,
            \end{align*}
    for all $ \testnuhs\in \Otauhs$.
Moreover, there exists $\lambda>0$ sufficiently large and $h_0 > 0$ such that, if $0<h\leq h_0$, and $t\in [0,T]$ we have
\begin{equation}\label{eq: stability dermat non linear Ru}
\begin{split}
     \normHunogst{\dermat \Rs \bs u}
     \lesssim{}&\normHunogtvec{\bs u}+\normHunogtvec{\dermat \bs u}
     \\
     &+\norm{\bs u}_{\Hk(\boundaryG)}+\norm{\dermat \bs u}_{\Hk(\boundaryG)}.
\end{split}
\end{equation}

\end{lemma}
\begin{proof}
    The proof of the identity~\eqref{eq: def dermat non linear Ru} is analogous to the proof of~\eqref{eq: def dermat linear Ru}. 
    The stability bound~\eqref{eq: stability dermat non linear Ru} follows using~\eqref{eq: def dermat non linear Ru} with $\testnuhs=\dermat\Rs \bs u$,~\eqref{eq: uniform bound nonlinear Ru} and arguments similar to those used in Lemma~\ref{Lem: stability non linear Ru}, for some $\lambda>0$ 
    sufficiently large. Thus, we only focus on the proof of the fact that $\dermat \Rs \bs u$ belongs to $\Otauhs$. 
    Given that $\Rs \bs u \in \Otauhs$, we know that for all $w_h\in \Otauhb$
    \begin{equation*}
        \int_{\boundaryGhcero} (\Rs \bs u\cdot \btauh)w_h = 0.
    \end{equation*}
    Then, deriving with respect to $t$ in~\eqref{eq: def L2 boundry orthog proj} and considering that $\boundaryGhcero$ is fixed and independent of $t$, we have, for all $w_h\in \Otauhb$,
    \begin{equation*}
    \begin{split}
        \partdt\int_{\boundaryGhcero} (\Rs \bs u\cdot \btauh)w_h = 0.
    \end{split}
    \end{equation*}
Now, since $\Rs \bs u\in \Otauhs$, and $\dermat w_h=0$, the previous equality reads
     \begin{equation*}
     \begin{split}
        0=\int_{\boundaryGhcero} (\dermat(\Rs \bs u)\cdot \btauh)w_h+(\Rs \bs u\cdot \btauh)\dermat w_h&= \int_{\boundaryGhcero} (\dermat(\Rs \bs u)\cdot \btauh)w_h.
     \end{split}
 \end{equation*}
 Therefore, $\dermat \Rs \bs u $ belongs to $\Otauhs$.

\end{proof}

\begin{proposition}\label{prop: error estimate dermat Ritz proj on boundary}
 Let $\bs{u}$ be a smooth function defined in $\Gt$, such that $\bs{u}\in \Otau$ with $\bs u\in {H^{p+1}_{\partial}(\Gammat)}^3$, and $(\dermat)^{(\ell)} \bs u\in {H^{p+1}_{\partial}(\Gammat)}^3$ for all $t\in[0,T]$. Then, there exists $h_0>0$ such that the material derivative of the Ritz map satisfies, for all $t\in[0,T]$ and $0<h\leq h_0$,
\begin{multline}\label{eq: error dermat Ru}
      \norm{(\dermat)^{(\ell)} (\bs u- \Rs[\Gamma]\bs u)}_{\Hunogammat}\\
      \lesssim h^p\Big(\sum_{j=0}^{\ell}\norm{(\dermat)^{(\ell)}\bs u}_{{H^{p+1}(\Gammat)}}+\norm{(\dermat)^{(\ell)} \bs u}_{H^{p+1}(\boundaryG)}\Big).
\end{multline}
\end{proposition}
\begin{proof} 
    We prove~\eqref{eq: error dermat Ru} for $\ell=1$, the case when $\ell>1$ is analogous.
    
    We consider an extension $\tilde{\bs{w}}\in\Hunogamma^3$
    of ${(\Ps(\btauhunit)\lifts{\dermat\bs{u}})}^{\boundaryG}\in H^{\frac{1}{2}}(\boundaryG)^3$,
    such that $\tilde{\bs{w}}|_{\boundaryG}=(\Ps(\btauhunit)\lifts{\dermat\bs{u}})^{\boundaryG}$
    satisfies ${(\tilde{\bs{w}},\bs \eta)}_{\Hk(\Gamma)}={(\dermat\bs u,\bs \eta)}_{\Hk(\Gamma)}$ for all $\bs \eta\in \Hk_0(\Gamma)^3$, whence 
\begin{equation}\label{eq: H1/2 norm dermat u-tilde w}
    \normHunog{\dermat\bs u-\tilde{\bs{w}}}=\norm{ \dermat\bs u-(\Ps(\btauhunit)\dermat\bs{u})^{\boundaryG}}_{\Hk[\frac{1}{2}](\boundaryG)}.
\end{equation}

We start by denoting $\dermat\eR = \dermat \bs u-\dermat \Rs[\Gamma]\bs u$ and adding and subtracting $ \Qs[\Gamma ]\tilde{\bs{w}}$ in the definition of the $\Hunogammalambda$-norm of $\dermat\eR$
\begin{align*}
     \normcuad{\dermat\eR}_{\Hunogammalambda}
    ={}&(\dermat\eR,\dermat u-\Qs[\Gamma ]\tilde{\bs{w}})_{\Hunogammalambda}+(\dermat\eR,\Qs[\Gamma ]\tilde{\bs{w}}-\dermat\Rs[\Gamma]u)_{\Hunogammalambda} 
     \\
    \leq{}&  \norm{\dermat\eR}_{\Hunogammalambda} \norm{\dermat u-\Qs[\Gamma ]\tilde{\bs{w}}}_{\Hunogammalambda} 
    \\
    &+(\dermat\eR,\Qs[\Gamma ]\tilde{\bs{w}}-\dermat\Rs[\Gamma]u)_{\Hunogammalambda}
    \end{align*}
Since $\dermat\bs{u}\in \Otau$, using~\eqref{eq: H1/2 norm dermat u-tilde w} and proceeding as in Proposition~\ref{prop: error estimate non-linear Ru} we can bound
\begin{align*}
    \normHunog{\dermat \bs u-\tilde{\bs{w}}}&\lesssim h^p \norm{\dermat \bs u}_{{\Hk[p+1](\boundaryG)}},\\
    \norm{\dermat \bs u-\tilde{\bs{w}}}_{{\Lp(\boundaryG)}}&\lesssim h^{p+1} \norm{\dermat \bs u}_{{\Hk[p+1](\boundaryG)}}.
\end{align*}
These estimates allow us to conclude
\begin{align}\label{eq: bound dermat u minus Q tilde w Huno}
    \normHunog{\dermat \bs u-\Qs[\Gamma ]\tilde{\bs{w}}}&\lesssim h^p \left(\norm{\dermat \bs u}_{{\Hk[p+1](\Gamma)}}+\norm{\dermat \bs u}_{{\Hk[p+1](\boundaryG)}}\right)\\\label{eq: bound dermat u minus Q tilde w Ldosb}
    \norm{\dermat \bs u-\Qs[\Gamma ]\tilde{\bs{w}}}_{\Lp(\boundaryG)}&\lesssim h^{p+1} \norm{\dermat \bs u}_{{\Hk[p+1](\boundaryG)}}.
\end{align}
Then, 
\begin{align*}
     \normcuad{\dermat\eR}_{\Hunogammalambda}   \leq  &\norm{\dermat\eR}_{\Hunogammalambda} c_\lambda h^p\left(\norm{\dermat \bs u}_{{\Hk[p+1](\Gamma)}}+\norm{\dermat \bs u}_{\Hk[p+1](\boundaryG)}\right)\\
     &+{(\dermat\eR,\Qs[\Gamma]\tilde{\bs{w}}-\dermat\Rs[\Gamma]u)}_{\Hunogammalambda}.
\end{align*}
To bound the last term we first analize ${(\dermat\eR,\liftg{\testnuhs})}_{\Hunogammalambda}$ for any $\testnuhs\in\Otauhs$, using the Definition~\ref{eq: def dermat non linear Ru} and similar arguments to those we have used in Lemma~\ref{Lem: stability - galerkin orthog dermat linear Ru} we have
\begin{align*}
    {(\dermat\eR,\testnuhs)}_{\Hunogammalambda}&\lesssim h^p\left(\normHunog{\bs u}+\normHunog{\dermat \bs u}\right)\norm{\liftg{\testnuhs}}_{\Hunogammalambda}+|(I)|\\
  &\leq ch^{2p}\left(\normcuad{\bs u}_{\Hunogamma}+\normcuad{\dermat \bs u}_{\Hunogamma}\right)+\frac{1}{2}\normcuad{\liftg{\testnuhs}}_{\Hunogammalambda}+|(I)|,
\end{align*}
where 
\begin{align*}
    (I)=&\int_{\boundaryGhcero}\left(\dermat\Rs[\Gammas] \bs u\cdot \curvbdryh(\Rs[\Gammas] \bs u\times \btauh)+\Rs[\Gammas] \bs u\cdot \curvbdryh(\dermat \Rs[\Gammas] \bs u\times \btauh)\right)\cdot \testnuhs\\
    &-\int_{\boundaryG}\left(\dermat \bs u\cdot \curvbdry( \bs u\times \btau)+\bs u\cdot \curvbdry( \dermat \bs u\times \btau)\right)\cdot \liftgamma{\testnuhs}.
\end{align*}
This last group of terms can be bounded applying similar strategies as we have used to bound the group of terms $(III)$ in Lemma~\ref{prop: error estimate non-linear Ru}. Thus, using the regularity of $\bs u$ and $\Rs \bs u$, given by Lemma~\ref{Lem: uniform bound non-linear Ru}, we deduce that 
\begin{align*}
|(I)|\lesssim & \normcuad{\dermat\eR}_{\Lp(\boundaryG)}+c h^{2p}\left(\normcuad{\bs u}_{{\Hk[p+1](\Gamma)}^3}+\normcuad{\dermat \bs u}_{\Hk[p+1](\boundaryG)^3}\right)\\
&+\normcuad{\eR}_{\Lp(\boundaryG)}+\normcuad{\liftg{\testnuhs}}_{\Lp(\boundaryG)}.
\end{align*}
Whence, if $\testnuhs=\Qs[\Gamma]\tilde{\bs{w}}-\dermat\Rs[\Gamma]u$ we have
\begin{align*}
    \normcuad{\dermat\eR}_{\Hunogammalambda}   \leq{}  &\frac{1}{4}\normcuad{\dermat\eR}_{\Hunogammalambda}+c\norm{\dermat\eR}_{\Lp(\boundaryG)}+\frac{1}{4}\normcuad{\liftg{\testnuhs}}_{\Hunogammalambda}\\
    &+c_\lambda h^{2p}\left(\normcuad{\bs u}_{{\Hk[p+1](\Gamma)}^3}+c\normcuad{\dermat \bs u}_{{\Hk[p+1](\boundaryG)}^3}\right)\\
    &+c\normcuad{\eR}_{\Lp(\boundaryG)}+c\normcuad{\liftg{\testnuhs}}_{\Lp(\boundaryG)}.
\end{align*}
We now take $\testnuhs=\Qs\lifts{\tilde{\bs{w}}}-\dermat\Rs \bs u\in\Otauhs$ in this last estimate and obtain
\begin{align*}
    \normcuad{\dermat\eR}_{\Hunogammalambda} 
    \leq{}& \frac{1}{4}\normcuad{\dermat\eR}_{\Hunogammalambda}+c \norm{\dermat\eR}_{\Lp(\boundaryG)}+\frac{1}{4}\normcuad{\Qs[\Gamma]\tilde{\bs{w}}-\dermat\Rs[\Gamma]u}_{\Hunogammalambda}
    \\
    &+c_\lambda h^{2p}\left(\normcuad{\bs u}_{{\Hk[p+1](\Gamma)}}+\normcuad{\dermat \bs u}_{{\Hk[p+1](\boundaryG)}}\right)\\
    &+c\normcuad{\eR}_{\Lp(\boundaryG)}+c\normcuad{\Qs[\Gamma]\tilde{\bs{w}}-\dermat\Rs[\Gamma]u}_{\Lp(\boundaryG)}
    \\
    \leq{}& \frac{1}{4}\normcuad{\dermat\eR}_{\Hunogammalambda}+c \norm{\dermat\eR}_{\Lp(\boundaryG)}
    \\
    &+\frac{1}{4}\normcuad{\Qs[\Gamma]\tilde{\bs{w}}-\dermat \bs u+\dermat \bs u-\dermat\Rs[\Gamma]u}_{\Hunogammalambda}\\
    &+c_\lambda h^{2p}\left(\normcuad{\bs u}_{{\Hk[p+1](\Gamma)}}+\normcuad{\dermat \bs u}_{{\Hk[p+1](\boundaryG)}}\right)\\
    &+c\normcuad{\eR}_{\Lp(\boundaryG)}+c\normcuad{\Qs[\Gamma]\tilde{\bs{w}}-\dermat \bs u+\dermat \bs u-\dermat\Rs[\Gamma]u}_{\Lp(\boundaryG)}
    \\
    \leq{}& \frac{1}{2}\normcuad{\dermat\eR}_{\Hunogammalambda}+c \norm{\dermat\eR}_{\Lp(\boundaryG)} 
    \\
    &+c_\lambda h^{2p}\left(\normcuad{\bs u}_{{\Hk[p+1](\Gamma)}}+\normcuad{\dermat \bs u}_{{\Hk[p+1](\boundaryG)}}\right)\\
    &+c\normcuad{\eR}_{\Lp(\boundaryG)},
\end{align*}
where in the last inequality we have used~\eqref{eq: bound dermat u minus Q tilde w Huno} and~\eqref{eq: bound dermat u minus Q tilde w Ldosb}. 

Finally, by using trace inequality~\eqref{eq: trace ineq} with an appropiate $\epsilon>0$, we deduce the $\Hk$-error estimate~\eqref{eq: error dermat Ru} for $\lambda>0$ sufficiently large.
\end{proof}

Following the steps of the proof of Lemma~\ref{Lem: uniform bound linear Ru}, and using previous proposition we obtain the following regularity result for the nonlinear Ritz projection.

\begin{lemma}\label{Lem: uniform bound non-linear dermat Ru}
    If 
    $(\dermat)^{(\ell)}\bs{u}\in W^{p+1,\infty}(\Gammat)^3$
    for all $t\in[0,T]$ and some $\ell \ge 1$, then there exists $c>0$ such that
    \begin{equation}\label{eq: uniform bound nonlinear dermat Ru}
        \norm{(\dermat)^{(\ell)}\Rs \bs{u}}_{\Wunoinf(\Gammat)}
        \le c \norm{(\dermat)^{(\ell)}\bs{u}}_{{W^{p+1,\infty}(\Gammat)}}.
    \end{equation}
\end{lemma}

\section{Convergence of semi-discrete scheme}\label{sec: semidisc espacial}

Making use of the preliminary results developed in the previous sections, we now prove
a priori error estimates for a Galerkin approximation of Problem~\ref{weak problem MCF}, under $h$-refinements and using B-spline spaces of degree $p\ge 2$ and smoothness $C^\ell$, $0\le \ell \le p-1$ for all the variables.

\subsection{Spatial semi-discretization}\label{sec: spatial semi disc}

We start defining the semi-discrete problem as follows.

\begin{problem}\label{weak discrete problem MCF}
    Find $\Xh\argu\in \Sparam^3$ ($\Gammah(t)=$ image of $\Omega$ under $\Xh(t)$), $\vh\argu\in \Shcero^3$, $\Hh\argu\in \Shcero$ and $\nh \argu\in\Otauh$ such that:
		\begin{align}
   \bs v_h = &\ \Qs[\Gammah](-\Hh\nh),\label{eq: weakh-velocity}\\
     \int_{\Gammah} \dermat \Hh \testHh+ \int_{\Gammah} \gradgh \Hh \cdot \gradgh \testHh = &\ \int_{\Gammah} |A_h|^2 \Hh \testHh ,\label{eq: weakh-curvature}  \\
			\int_{\Gammah} \dermat   \nh \cdot \testnuh + \int_{\Gammah}\gradgh\nh \cdot \gradgh \testnuh =& \ \int_{\Gammah} |A_h|^2 \nh \cdot \testnuh+\int_{\boundaryGhcero} \alpha_{\partial,h}\conorh\cdot \testnuh,\label{eq: weakh-normal} \\
				\partial_t \Xh = &\ \vh\circ \Xh,\label{eq: discrt vel}
		\end{align}
 for all $t>0$ and for all $\testHh\in \Shcero$ and $  \testnuh\in\Otauh$.
 The discrete initial values are defined as $\X_{h,0}:=\Q\X_0$, $\kappa_{h,0}:=\Qs[\Gamma_{h,0}](\kappa_0)$, $\n_{h,0}:=\Rs[\Gamma_{h,0}](\n_0)$ and $\vel_{h,0}:=\Qs[\Gamma_{h,0}](-\kappa_{h,0}\n_{h,0})$. 

 We denote with $A_h=\gradgh\nh$ the Weingarten map and with $|A_h|$ its Frobenius norm. 
 Also, $\btauh$ and $\curvbdryh$ are defined as in Lemma~\ref{Lem: def and estimate tauh normalbh}, and we define $\alpha_{\partial,h}=(\curvbdryh \cdot \nh)$ and $\conorh=\nh\times\btauh$. Note that $\alpha_{\partial,h}$ and $\conorh$ are neither the normal curvature nor the conormal vector to $\boundaryGhcero$ because $\nh$ is not necessarily the normal vector to the discrete surface.
\end{problem}

The goal of this section is to estimate the error between each discrete function $u_h\in\{\idh,\vh,\Hh,\nh\}$ and its continuous counterpart $u\in\{\idg,\vel,\Hm,\n\}$, from the weak formulation in Problem~\ref{weak problem MCF}.
To do this, we will consider a theoretical discrete approximation $u^*$ defined on the quasi-interpolated surface $\Gammas$ for which we have an error estimate of optimal order, so that
\begin{align*}
    \norm{u - \liftgamma{u_h}}_{\Hunogamma}&\leq \norm{u-\liftgamma{u^*}}_{\Hunogamma}
    +\norm{\liftgamma{u^*}-\liftgamma{u_h}}_{\Hunogamma}\\
    &= O(h^p)+ \norm{\liftgamma{u^*}-\liftgamma{u_h}}_{\Hunogamma}
    = O(h^p) +\normHunog{e_{u}}.
\end{align*}
Therefore, for each unknown variable $u\in\{\idg,\vel,\Hm,\n\}$ from the weak formulation we now define a discrete approximation $u^*$ and prove bounds for $e_u := u^* - u_h$. This the so-called consistency error.

Recall the definition of the quasi-interpolated surface $\Gammast$ from~\eqref{eq: quasi-int surf}, this surface moves with velocity $\vs$, such that
\begin{equation}\label{eq: quasi-int vel}
    \partial_t \Xs = \ \vs\circ \Xs.
\end{equation}
That means $\vs:=\Qs[\Gammas](\vel)$. For $\Hm$ and $\n$ we will use the Ritz projections presented in Sections~\ref{sec: Linear Ritz projection} and~\ref{sec: Nonlinear Ritz projection} and define
\begin{equation}
\Hs:=\Rscero\Hm
\quad\text{and}\quad
\ns:=\Rs \n.
\end{equation}


We also need to define suitable approximations for the quantities defined only at the boundary. Since $\boundaryGhcero=\boundaryGscero$, we consider $\btaus:=\btauh$ and $\curvbdrys:=\curvbdryh$, which are thus independent of $t$ and of the discrete solution. For $\conor=\n\times \btau$ and $\alpha=\curvbdry\cdot \n$, we consider suitable approximations defined in the following Lemma.
\begin{lemma} 
    Let $\btaus=\Qs[\boundaryGhcero]\btau$, $\curvbdrys=\Qs[\boundaryGhcero]\curvbdry$, and let $\conors:=\ns\times \btaus$, $\alpha^*=\curvbdrys \cdot \ns$.
    Then, there exists $h_0>0$ such that for all $0< h\leq h_0$,
    \begin{align}
         \label{eq: conor alpha error}
       \norm{\liftb{\conor}-\conors}_{\Lp(\boundaryGhcero)}&\lesssim h^{p},
       &\norm{\liftb{\alpha}-\alpha^*}_{\Lp(\boundaryGhcero)}&\lesssim h^{p}.
    \end{align}
    Furthermore, from these estimates we get for all $0 < h\leq h_0$ 
    \begin{align}
        \label{eq: uniform bound conor alpha}
        \norm{\conors}_{\Lp[\infty](\boundaryGhcero)}\lesssim C,
        \qquad
        \norm{\alpha^*}_{\Lp[\infty](\boundaryGhcero)}\lesssim C.   
    \end{align}
\end{lemma}
\begin{proof}
The error estimates~\eqref{eq: conor alpha error} can be derived from Lemma~\ref{Lem: def and estimate tauh normalbh}, the error estimates~\eqref{eq: error estimate non-linear Ru} for $\Rs\n$ and Lemma~\ref{Lem: uniform bound non-linear Ru}. While~\eqref{eq: uniform bound conor alpha} follows from inverse inequalities, interpolation error estimates and regularity of $\conor$ and $\alpha$ (similar to the proof of Lemma~\ref{Lem: uniform bound non-linear Ru}).
\end{proof}

Note that $\conors$ and $\alpha^*$ are only defined on $\boundaryGhcero$.
Also, $\alpha^*$ is not necessarily the normal curvature, and $\conors$ is not the vector conormal to $\boundaryGhcero$.


\begin{lemma}\label{Lem: uniform bounds xs vs hs nus}
    If $\X(t)$, $\vel(t)$, $\Hm(t)$ and $\n(t)$ are sufficiently smooth, 
    there exists $h_0>0$ such that,  
      \begin{equation}\label{eq: uniform bound Xs}
      \begin{split}
             \norm{\nabla\Xs\argu}_{\Lp[\infty](\Omega)}&\lesssim  \norm{\nabla\X\argu}_{\Lp[\infty](\Omega)},
      \end{split}
      \end{equation}
      the quasi-interpolants $\ids$ and $\vs$ satisfy
     \begin{equation}\label{eq: uniform bound xs vs}
          \begin{gathered}
           \normWunost{\ids}+\normWunost{\vs}\leq C,\\
      \end{gathered}
     \end{equation}
     and the Ritz-type projections $\Hs$ and $\ns$ satisfy
     \begin{equation}\label{eq: uniform bound hs ns}
          \normWunost{\Hs}+\normWunost{\ns}\leq C,\\
     \end{equation}
     for all $t\in[0,T]$ and $0<h\leq h_0$.
     \end{lemma}
     \begin{proof}
  The uniform estimate~\eqref{eq: uniform bound Xs} follows directly from the definition of $\Xs$, the regularity of $\X$, and interpolation error given in~\eqref{eq: error estimate X*}. Moreover, using norm equivalence and the fact that $\ids=\Xs\circ\invXs$, that is, $\ids$ is the push-forward of $\Xs$, we obtain the uniform bound for $\ids$. The estimate for $\vs$ follows directly from its definition and the bound~\eqref{uniform bound Q on surface}. The last inequality in~\eqref{eq: uniform bound hs ns} is obtained from Lemmas~\ref{Lem: uniform bound linear Ru} and~\ref{Lem: uniform bound non-linear Ru}.
     \end{proof}

\begin{definition}\label{def: defects}
We define the \emph{defects} $\dv$, $\dH$ and $\dn$ as the quantities that appear on the right-hand side of the equations~\eqref{eq: weakh-velocity}--\eqref{eq: weakh-normal} when replacing $\Xh$, $\vh$, $\Hh$, and $\nh$ by the corresponding discrete quantities $\Xs$, $\vs$, $\Hs$ and $\ns$.
That is, for each $t>0$, $\dH\in \Shscero$, $\dv \in \Shscero^3$, $\dn\in \Otauhs$ satisfy
\begin{align}
    \vs &= \Qs(-\Hs\ns) + \dv \label{eq: weakh*-velocity},
    \\
    \intgammas \dermat \Hs \testHhs + \intgammas \gradghs \Hs \cdot \gradghs \testHhs &=  \intgammas |A_h^*|^2 \Hs \testHhs + \intgammas \dH  \testHhs \label{eq: weakh*-curvature},
\end{align}
\begin{equation}
    \begin{split}
        \intgammas\dermat\ns \cdot \testnuhs + \intgammas \gradghs\ns:&\gradghs\testnuhs =  \intgammas|A_h^*|^2\ns \cdot \testnuhs \\
        &+ \int_{\boundaryGscero} \alpha_{\partial}^*\conors\cdot \testnuhs + \intgammas \dn\cdot \testnuhs \label{eq: weakh*-normal},
    \end{split}
\end{equation}
for all $\testHhs\in \Shscero$ and all $\testnuhs \in \Otauhs$,
where $A_h^*=\gradghs\ns$.
\end{definition}

\subsection{Consistency}\label{sec: consistency}

\begin{proposition}\label{prop: consistency estimates}
Let $\dv$, $\dH$ and $\dn$ be the functions defined in Definition~\ref{def: defects}, and assume that the exact solution is sufficiently smooth.
Then, there exists $h_0>0$ such that for all $0<h\leq h_0$ and all $t\in[0,T]$,
\begin{gather*}
           \normHunogs{\dv\argu}\lesssim{}h^p,\quad
          \normLdosgs{\dH\argu}\lesssim{} h^p\\
             \normLdosgs{ \dn\argu}\lesssim{} h^p,\quad\text{and}\quad 
             \normLdosgs{ \dermat\dn\argu}\lesssim{} h^p.
\end{gather*}
\end{proposition}

The following regularity conditions are be sufficient for this proposition:
\begin{align*}
    \X\argu\in W_{\partial}^{p+1,\infty}(\Omega)^3,&\quad  \vel\argu\in H^{p+1}(\Gammat)^3,& \Hm\argu\in W^{p+1,\infty}(\Gammat),\\
   \dermat\Hm\argu\in W^{p+1,\infty}(\Gammat),&\quad \n\argu\in W^{p+1,\infty}_{\partial}(\Gammat)^3,&  \dermat\n\argu\in W^{p+1,\infty}(\Gammat)^3,\\
  \partial_{\conor}\n\in \Hk[p+1](\boundaryG)^3,& \quad\dermat(\partial_{\conor}\n)\in \Hk[p+1](\boundaryG)^3,& \dermat\dermat\n\argu\in \Hk[p+1]_{\partial}(\Gammat)^3.
\end{align*}
\begin{proof}
    Subtracting~\eqref{eq: weak velocity}--\eqref{eq: weak curvature} from~\eqref{eq: weakh*-velocity}--\eqref{eq: weakh*-curvature} we have the following equations for the defects $\dv$ and $\dH$:
 \begin{align*}
	\dv ={}& (\vs -\lifts{\vel}) +(\lifts{-\Hm\n}- \Qs (-\Hs\ns)),\\
    \intgammas \dH  \testHhs={}&\intgammas \dermat \Hs \testHhs +\intgammas \gradghs \Hs\cdot \gradghs\testHhs-  \intgammas |A_h^*|^2 \Hs \testHhs\\
		&-\intgamma \dermat \Hm \liftgamma{\testHhs} -\intgamma \gradg \Hm \cdot \gradg \liftgamma{\testHhs}+ \intgamma |A|^2 \Hm\liftgamma{\testHhs}.
	\end{align*}
We start by obtaining the estimate for $\dv$. Recalling that $\vs=\Qs \vel$ and using the interpolation estimate for the velocity, we have
\begin{align*}
	\normHunogs{\dv} ={}& \normHunogs{\Qs \vel -\lifts{\vel}} 
+\normHunogs{\lifts{\Hm\n}- \Qs (\Hs\ns)}
 \\
 \lesssim{}& h^p\normHunog{\vel}
 +\normHunogs{\lifts{\Hm\n}-\Qs (\Hm\n)}
 \\
 &+\normHunogs{\Qs (\Hm\n)- \Qs (\Hs\ns)}\\
 \lesssim{}& h^p\normHunog{\vel}+ h^p\norm{\Hm\n}_{\Hk[p+1](\Gamma)}
+\normHunogs{\Qs (\Hm\n-\Hs\ns)},
\end{align*}
where we have used the linearity of $\Qs$.
To bound the last term, we take advantage of the linearity and stability of the operator $\Qs$, the error estimates for the Ritz projections of $\n$ and $\Hm$ and their uniform bounds~\eqref{eq: uniform bound linear Ru}, \eqref{eq: uniform bound nonlinear Ru}: 
\begin{align*}
  \normHunogs{\Qs (\Hm\n-\Hs\ns)}  
  \lesssim&\normHunogs{\lifts{\Hm\n} -(\Hs\ns)}\\
  \lesssim&\normHunogs{(\Hm-\Hs)\lifts{\n}} 
+ \normHunogs{\Hs(\lifts{\n} -\ns)}\\
  \lesssim{}& h^p.
\end{align*}
We thus have $\normHunogst{\dv} \lesssim h^p$, where the involved constant depends on the norms of $\vel$, $\Hm$ and $\n$ but is independent of $h$.

Regrouping the terms in the equation for $\dH$, recalling that $\Hs:=\Rscero\Hm$, and using the definition of $\Rscero$ from~\eqref{def: ritz map zero trace linear Ru}, we have
\begin{align*}
      \intgammas \dH \testHhs
      ={}&\left(\intgammas \dermat \Hs \testHhs -\intgamma \dermat \Hm \liftgamma{\testHhs}\right)
      \\
      &+\left(\intgammas \gradghs \Hs\cdot \gradghs\testHhs -\intgamma \gradg \Hm \cdot \gradg \liftgamma{\testHhs}\right)
      \\
      &- \left( \intgammas |A_h^*|^2 \Hs \testHhs- \intgamma |A|^2 \Hm\liftgamma{\testHhs}\right)\\
     ={}&\left(\intgammas \dermat \Hs \testHhs -\intgamma \dermat \Hm \liftgamma{\testHhs}\right)
      +\left(\intgammas  \Hs\testHhs -\intgamma \Hm  \liftgamma{\testHhs}\right)\\
      &- \left( \intgammas |A_h^*|^2 \Hs \testHhs- \intgamma |A|^2 \Hm\liftgamma{\testHhs}\right) \\
      ={}&(I) + (II) + (III).
	\end{align*}
This is a key point in the proof, where we have been able to get rid of the gradients, and thus obtain only $L^2$-type errors, thanks to the Ritz projection.

We now proceed to bound each of the three differences, using the geometric perturbation error  and the interpolation estimates. 
To bound $(I)$, we use~\eqref{eq: bound m*-m}, the uniform bound of $\dermat\Hs$, the regularity of $\dermat\Hm$ and the equivalence of norms
\begin{align*}
     |(I)|&=\left|\intgammas \dermat \Hs \testHhs -\intgamma (\dermat \Hm -\liftgamma{\dermat \Hs}+\liftgamma{\dermat \Hs})\liftgamma{\testHhs}\right|\\
     &\leq\left|\intgammas \dermat \Hs \testHhs-\intgamma\liftgamma{\dermat \Hs}\liftgamma{\testHhs}\right|+ \left|\intgamma (\dermat \Hm -\liftgamma{\dermat \Hs})\liftgamma{\testHhs}\right|\\
     &\lesssim h^p\normLdosgs{\dermat \Hs}\normLdosgs{\testHhs}+ h^p\normLdosgs{\lifts{\dermat \Hm}}\normLdosgs{\testHhs}
     \\
     &\lesssim h^p\normLdosgs{\testHhs}.
	\end{align*}
We can proceed similarly for $(II)$. Finally, we bound $(III)$ as follows:
\begin{align*}
      |(III)|={}&\left| \intgammas |A_h^*|^2 \Hs \testHhs- \intgamma |A|^2 \Hm\liftgamma{\testHhs}\right|\\
      \leq{}&\left| \intgammas |A_h^*|^2 \Hs \testHhs-\intgamma \liftgamma{|A_h^*|^2 \Hs}\liftgamma{\testHhs}\right|
      \\
      &+\left| \intgamma (|A|^2 \Hm-\liftgamma{|A_h^*|^2 \Hs})\liftgamma{\testHhs}\right|
      \\
      \lesssim{}& h^p\normLdosgs{\testHhs}+ \normLdosg{|A|^2 \Hm
      -\liftgamma{|A_h^*|^2 \Hs}}\normLdosg{\liftgamma{\testHhs}}
      \\
      \lesssim{}& h^p\normLdosgs{\testHhs}.
	\end{align*}
In the last bound we have used that
\[
\normLdosg{|A|^2 \Hm-\liftgamma{|A_h^*|^2 \Hs}}\lesssim  \normLdosg{A -\liftgamma{ A_h^*}}+ \normLdosg{\Hm -\liftgamma{ \Hs}},
\]
together with the estimates~\eqref{eq: error estimate non-linear Ru} and~\eqref{eq: error estimate linear Ru} and the equivalence of norms on $\Gamma$ and $\Gammas$. Altogether, we obtain $\intgammas \dH \testHhs\lesssim h^p \normLdosgs{\testHhs}$, 
and thereby the desired bound for $\normLdosgs{\dH}$.

The proof for $\dn$ is a bit more technical because $\liftg{\Otauh}\not\subset \Otau$,
which entails an additional non-conformity of the method.

We recall that, on $\partial \Gamma$, due to~\eqref{eq: A mu mu},
\[
    \partial_{\conor}\n
= (\partial_{\conor}\n \cdot \btau) \btau + (\partial_{\conor}\n\cdot \conor) \conor
= (\partial_{\conor}\n \cdot \btau) \btau + (\underbrace{\n \cdot \curvbdry}_\alpha) \conor
.
\]
Using this identity and~\eqref{eq: green nu} we obtain that,
for any $\testnuhs\in \Otauhs$,
    \begin{align*}
        \intgamma \dermat\n \cdot \liftgamma{\testnuhs} 
     =&-\intgamma \gradg \n:\gradg \liftgamma{\testnuhs}
     + \intgamma |A|^2\n \cdot \liftgamma{\testnuhs}+\int_{\partial \Gamma} \alpha\conor\cdot\liftgamma{\testnuhs}\\
     &+\int_{\partial \Gamma} (\partial_{\conor}\n\cdot\btau)\btau\cdot \liftgamma{\testnuhs}
    \end{align*}
using only that $\liftg{\testnuhs}\in \Hk(\Gamma)^3$.
    Adding this to~\eqref{eq: weakh*-normal} we deduce
    \begin{align*}
        \intgammas \dn \cdot \testnuhs
        ={}&\intgammas\dermat\ns \cdot \testnuhs
        + \intgammas \gradghs\ns:\gradghs\testnuhs
        - \intgammas|A_h^*|^2\ns \cdot \testnuhs
        \\
        &-\int_{\partial \Gammas} \alpha_h^*\conors\cdot \testnuhs
     +\int_{\partial \Gamma} \alpha\conor\cdot\liftgamma{\testnuhs}\\
           &-	\intgamma \dermat\n \cdot \liftgamma{\testnuhs} 
     -\intgamma \gradg \n:\gradg \liftgamma{\testnuhs}
     + \intgamma |A|^2\n \cdot \liftgamma{\testnuhs}\\
     &+\int_{\partial \Gamma} (\partial_{\conor}\n\cdot\btau)\btau\cdot \liftgamma{\testnuhs}.
    \end{align*}
We recall that $\ns:=\Rs \n$, with $\Rs$ the nonlinear Ritz projection from Definition~\ref{def: non linear Ru trace in Otauh}, and also that $\conor:=\n\times \btau$, $\alpha=\curvbdry \cdot \n$, $\conors:=\ns\times \btaus$, $\alpha^*=\curvbdrys \cdot \ns$, so that the first two boundary integrals cancel out, leading to
\begin{equation}\label{eq: defect for normal}
    \begin{split}
        \intgammas \dn \cdot \testnuhs
          ={}&\left(\intgammas\dermat\ns \cdot \testnuhs-\intgamma \dermat\n \cdot \liftgamma{\testnuhs}\right)\\
          &+\lambda \left(\intgammas\ns \cdot \testnuhs-\intgamma \n \cdot \liftgamma{\testnuhs}\right)
          \\
          &- \left( \intgammas|A_h^*|^2\ns \cdot \testnuhs- \intgamma |A|^2\n \cdot \liftgamma{\testnuhs}\right)\\
          &+\int_{\partial \Gamma} (\partial_{\conor}\n\cdot\btau)\btau\cdot \liftgamma{\testnuhs}.
    \end{split}
	\end{equation}
Here, again, we were able to get rid of the gradients, by resorting to the definition of the Ritz projection.
The first three terms can be bounded as we did for $\dH$, by $ch^p\normLdosgs{\testnuhs}$. 
For the last term, which involves a boundary integral, we get
\begin{equation}\label{eq: boundary integral Dn mu on defect}
    \begin{split}
       \int_{\partial \Gamma} (\partial_{\conor}\n\cdot\btau)\btau\cdot \liftgamma{\testnuhs}
        =&\int_{\partial \Gamma} (\partial_{\conor}\n\cdot\btau)(\btau-\btauhunit)\cdot \liftgamma{\testnuhs}\\
        &+\int_{\partial \Gamma} (\partial_{\conor}\n\cdot\btau)\btauhunit\cdot \liftgamma{\testnuhs}\\
        &-\int_{\boundaryGhcero} \Qs[\boundaryGhcero](\partial_{\conor}\n\cdot\btau)\btauhunit\cdot \liftgamma{\testnuhs},
    \end{split}
\end{equation}
due to the fact that the last term vanishes because $\testnuhs \in \Otauhs$ (recall the definition of $\Otauhs$ from~\eqref{eq:otauhs}).
Then, using a trace and an inverse inequality,
\begin{equation}\label{eq: estimate boundary integral Dn mu on defect}
    \begin{split}
        \Big|\int_{\partial \Gamma} (\partial_{\conor}\n\cdot\btau)\btau&\cdot \liftgamma{\testnuhs} \Big|\lesssim \norm{\btau-\btauhunit}_{\Lp(\boundaryGhcero)}\norm{\testnuhs}_{\Lp(\boundaryGhcero)}\\
        +&\norm{(\partial_{\conor}\n\cdot\btau)-\Qs[\boundaryGhcero](\partial_{\conor}\n\cdot\btau)}_{\Lp(\boundaryGhcero)}\norm{\testnuhs}_{\Lp(\boundaryGhcero)}\\
        \lesssim &\ h^{p+1}h^{-1}\normLdosgs{\testnuhs}+h^{p+1}h^{-1}\normLdosgs{\testnuhs}.
    \end{split}
\end{equation}
With this last estimate we obtain the desired bound for $\dn$.

To bound $\normLdosgs{ \dermat\dn\argu}$, we observe that, if 
$\dermat\testnuhs=0$,
\begin{equation*}
   \intgammas\dermat \dn\cdot \testnuhs= \dt \intgammas \dn\cdot\testnuhs-\intgammas\dn\cdot \testnuhs\Divgs\vs,
\end{equation*}
whence
\begin{equation}\label{eq: main dt dn}
    \left|\intgammas\dermat \dn\cdot \testnuhs\right| \leq \left|\dt \intgammas \dn\cdot\testnuhs\right|+\left|\intgammas\dn\cdot \testnuhs\Divgs\vs\right|.
 \end{equation}

 In order to bound the first term on the right-hand side, we differentiate~\eqref{eq: defect for normal} in time:
    \begin{equation}\label{eq: dt dn}
        \begin{split}
        \dt \intgammas \dn\cdot&\testnuhs=\left(\intgammas\dermat\dermat\ns\cdot\testnuhs-\intgamma\dermat\dermat\n\cdot\lifts{\testnuhs}\right)\\
        &+\left(\intgammas\dermat\ns\cdot\testnuhs\Divgs\vs-\intgamma\dermat\n\cdot\lifts{\testnuhs}\Divg\vel\right)\\
        &+\left(\intgammas\dermat\ns\cdot\testnuhs-\intgamma\dermat\n\cdot\lifts{\testnuhs}\right)\\
        &+\left(\intgammas\ns\cdot\testnuhs\Divgs\vs-\intgamma\n\cdot\lifts{\testnuhs}\Divg\vel\right)\\
        &-\left(\intgammas\dermat(|A^*|^2)\ns\cdot\testnuhs-\intgamma\dermat(|A|^2)\n\cdot\lifts{\testnuhs}\right)\\
        &-\left(\intgammas|A^*|^2\dermat\ns\cdot\testnuhs-\intgamma|A|^2\dermat\n\cdot\lifts{\testnuhs}\right)\\
        &-\left(\intgammas|A^*|^2\ns\cdot\testnuhs\Divgs\vs-\intgamma|A|^2\n\cdot\lifts{\testnuhs}\Divg\vel\right)\\
        &+\int_{\partial \Gamma} (\dermat(\partial_{\conor}\n)\cdot\btau)\btau\cdot \liftgamma{\testnuhs}.
    \end{split}
\end{equation}
    In addition to the transport formula we have that
\begin{align*}
    \dermat (|A|^2)&=2(\gradg (\dermat\n):A-AD:A),\\
    \dermat (|\funss{A}|^2)&=2(\gradghs (\dermat\ns):\funss{A}-\funss{A}\funss{D}:\funss{A}),
\end{align*}
with $D=((\gradg \vel)^T-\n\tensor \n\gradg \vel)$ y $\funss{D}=((\gradghs \vs)^T-\n_{\Gammas}\tensor \n_{\Gammas}\gradghs \vs)$, and $\n_{\Gammas}$ the normal vector of $\Gammas$.
Taking into account these identities, all the differences of integrals over $\Gammas$ and $\Gamma$ appearing on the right-hand side of~\eqref{eq: dt dn} can be bounded with Lemma~\ref{Lem: geometric perturbation} using geometric error estimates together with similar arguments to those used above when we bound $\dn$, the regularity of $\dermat\dermat\n$, and the estimate~\eqref{eq: uniform bound nonlinear dermat Ru}.

To bound the last term of~\eqref{eq: dt dn}, corresponding to a boundary integral, we can proceed as in~\eqref{eq: boundary integral Dn mu on defect}:
\begin{equation}\label{eq: boundary integral dermat Dn mu on defect}
    \begin{split}
       \int_{\partial \Gamma} (\dermat(\partial_{\conor}\n)\cdot\btau)\btau\cdot \liftgamma{\testnuhs}
        &=\int_{\partial \Gamma} (\dermat(\partial_{\conor}\n)\cdot\btau)(\btau-\btauhunit)\cdot \liftgamma{\testnuhs}\\
        &+\int_{\partial \Gamma} (\dermat(\partial_{\conor}\n)\cdot\btau)\btauhunit\cdot \liftgamma{\testnuhs}\\
        &-\int_{\boundaryGhcero} \Qs[\boundaryGhcero]((\dermat(\partial_{\conor}\n)\cdot\btau))\btauhunit\cdot \liftgamma{\testnuhs},
    \end{split}
\end{equation}
and then we proceed again as in~\eqref{eq: estimate boundary integral Dn mu on defect} using also the smoothness of $\dermat(\partial_{\conor}\n)\cdot\btau$ and interpolation estimates. Finally, the first term on the right-hand side of~\eqref{eq: main dt dn} is bounded as
\begin{equation*}
    \left|\dt \intgammas \dn\cdot\testnuhs\right|\lesssim h^p\normLdosgs{\testnuhs}.
\end{equation*}
For the second term on the right-hand side of~\eqref{eq: main dt dn} we use the already obtained bound for $\normLdosgs{ \dn\argu}$, and get
\begin{equation*}
    \left|\intgammas\dn\cdot \testnuhs\Divgs\vs\right|\lesssim h^p\normLdosgs{\testnuhs},
\end{equation*}
which leads to the last bound of the assertion of this proposition.
\end{proof}

\subsection{Error equations}\label{sec: error equations}

We define the following error functions which belong to $\Shs$, $\Shscero$, $\Shscero^3$ and $\Otauhs$, respectively, for all $t>0$:
\begin{align*}
    \ex&=\lifts{\idh}-\ids, &\eH&=\lifts{\Hh}-\Hs, \\
    \ev&=\lifts{\vh}-\vs,& \en&=\lifts{\nh}-\ns,
\end{align*}
where we use the identity functions $\idh=\Xh\circ(\Xh)^{-1}$ and $\ids=\Xs\circ(\Xs) ^{-1}$. 
We thus obtain the \emph{error equations} by subtracting~\eqref{eq: weakh*-velocity},~\eqref{eq: weakh*-curvature} and~\eqref{eq: weakh*-normal} from~\eqref{eq: weakh-velocity},~\eqref{eq: weakh-curvature} and~\eqref{eq: weakh-normal}, respectively
\begin{align}\label{eq: error-pos}
    \ederx&=\ev \\
    \label{eq: error-vel}
\ev &= \lifts{\Qs[\Gammah] (-\Hh\nh)} - \Qs[\Gammas] (-\Hs\ns)- \dv 
\end{align}
\begin{equation}\label{eq: error-curvature}
\begin{split}
    \int_{\Gammah} \dermat \lifth{\eH} \testHh &+ \int_{\Gammah} \gradgh \lifth{\eH}\cdot \gradgh \testHh\\
    =&-\left(\int_{\Gammah}\gradgh (\Hs)^{\Gammah}\cdot \gradgh \testHh-\intgammas \gradghs \Hs\cdot \gradghs \testHhs\right)\\
&-\left(\int_{\Gammah} \dermat ((\Hs)^{\Gammah}) \testHh-\intgammas \dermat \Hs \testHhs\right)\\
&+\int_{\Gammah} |A_h|^2 \Hh \testHh-\intgammas |A_h^*|^2 \Hs \testHhs+\intgammas \dH \testHhs.
\end{split}
\end{equation}
\begin{equation}\label{eq: error-normal}
\begin{split}
    \intgammah \dermat \lifth{\en} &\cdot\testnuh + \intgammah \gradgh \lifth{\en}: \gradgh \testnuh
    \\
    =&-\left(\int_{\Gammah}\gradgh (\ns)^{\Gammah}: \gradgh \testnuh-\intgammas \gradghs \ns: \gradghs \testnuhs\right)\\
&-\left(\intgammah \dermat ((\ns)^{\Gammah})\cdot \testnuh-\intgammas \dermat \ns \cdot\testHhs\right)\\
&+\int_{\Gammah} |A_h|^2 \nh \cdot\testnuh-\intgammas |A_h^*|^2 \ns \cdot\testnuhs\\
&+ \int_{\boundaryGhcero} \alpha_{\partial,h}\conorh\cdot \testnuh-\int_{\boundaryGhcero} \alpha_{\partial}^*\conors\cdot \testnuhs+\intgammas \dn \cdot\testnuhs.
\end{split}
\end{equation}

\subsection{Relating discrete surfaces}\label{sec: relating discrete surfaces}

We start observing that due to the continuity of the two discrete approximations of $X$ and the fact that they coincide at $t=0$, they remain close for some time.

\begin{lemma}\label{lem: bound for splines parameterizations}
    Let $p\geq2$, and for $h>0$ let $\Xh$ be the spline parameterization of $\Gammah$ given by the numerical solution of Problem~\ref{weak discrete problem MCF}, and let $\Xs$ be as in~\eqref{eq: quasi-int surf}. 
    Then, there exist $t_0\in[0,T]$ and $h_0>0$ such that 
    \begin{equation}\label{eq: bound EX}
        \normWunomega{\Xh\argu-\Xs\argu}\le h^{(p-1)/2},\quad\forall t\in[0,t_0],
        \ \forall h\in(0,h_0].
    \end{equation}
    Also, Assumption~\ref{ass: regularity Xh} holds, i.e.,
    $\Xh$ and its inverse are bounded in $\Wunoinf(\Omega)$, uniformly in $h$ for all $0<h\le h_0$, and all $0\le t \le t_0$.
    \end{lemma}

    \begin{proof}
        Let $\EX$ be the pull-back map $\EX=\ex\circ\Xs=\Xh-\Xs$. 
        Due to the fact that $\EX:[0,T]\to W^{1,\infty}(\Omega)$ is continuous in time
        and $\EX\argu[0]=0$, for each $h>0$ there exists $t_0>0$ (which can be chosen independent of $h$, for $h<h_0$, and some $h_0>0$) such that~\eqref{eq: bound EX} holds.
        From~\eqref{eq: bound EX} and using~\eqref{eq: uniform bound Xs} we deduce that, taking $h_0$ as in Lemma~\ref{Lem: uniform bounds xs vs hs nus}, Assumption~\ref{ass: regularity Xh} holds.
    \end{proof}

Although we do not use the matrix-vector formulation used in \cite{KLL2019} and \cite{KLLP2017}, with arguments analogous to those in~\cite[Lemmas 4.1, 4.2 and 4.3]{KLLP2017} for relating discrete surfaces, we obtain the following lemma. 
In order to prove it, we relate the different discrete surfaces using a new (artificial) motion $\Xb(\cdot,\theta): \Omega \rightarrow \Gammabar(\theta)$, with $\theta\in[0,1]$, such that  $\Gammabar(0)=\Gammast$, $\Gammabar(1)=\Gammat$ for some fixed $t\in[0,t_0]$ where the previous Lemma is fulfilled.

\begin{lemma}\label{Lem: equivalence of discrete norms} For $\theta\in [0,1]$, let $\Gammabar(\theta)$ be the surface parameterized by $\Xb(\cdot,\theta)=\Xs+\theta\EX$, with 
$\EX$ the pull-back map $\EX=\ex\circ\Xs=\Xh-\Xs$. If 
\begin{equation}\label{eq: uniform bound for Ex}
    \normLinfomega{\nabla\EX}< c_{\X},
\end{equation}
with $c_{\X}>0$ possibly depending on $\X$.
Then, there exists $h_0>0$ such that, for all $0<h\le h_0$ we have:
\begin{enumerate}[(i)]
    \item For $1\leq q\leq \infty$, and $\funss{\eta} \in W^{1,p}(\Gammas)$,
    \begin{align}\label{eq: equivalence discrete norms}
        \norm{\liftgbarra{\funss{\eta}}}_{W^{1,q}(\Gammabart)}\lesssim\norm{\funss{\eta}}_{W^{1,q}(\Gammas)}.
    \end{align}
    \item \label{item: equivalence norms} Also, the $\Lp(\Gammabart)$ and $\Hk(\Gammabart)$ norms are $h$-uniformly equivalent for $\theta\in[0,1]$.
\end{enumerate}
\end{lemma}
\begin{proof}
   The main difference with \cite{KLLP2017} is that, in order to prove~\eqref{eq: equivalence discrete norms}, we use the parameterization defined in the reference domain.
   Then, since $\nabla \Xb\argu[\theta]=\nabla(\Xs+\theta\EX)$ and~\eqref{eq: uniform bound Xs} holds, we have that $ \normLinfomega{\nabla\Xb\argu[\theta]}$ is bounded uniformly in $h$ and $\theta$. Furthermore, $\nabla \Xb\argu[\theta]$ will be a full rank matrix if 
   \[\normLinfomega{(G_\Gammas)^{-1}(\nabla\Xs)^T\nabla\EX}<1. \]
   Then, by Lemma~\ref{lem: bound for splines parameterizations}, we can choose $h_0>0$ such that 
   \begin{equation*}
       \normLinfomega{\nabla\EX} \le h^{\frac{p-1}{2}}<\dfrac{1}{C_{\X}}:=c_{\X},
       \quad\forall h\in(0,h_0],
   \end{equation*}
  where $C_{\X}$, depending of $\X$, satisfies
  $\normLinfomega{(G_\Gammas)^{-1}(\nabla\Xs)^T}\leq  C_{\X}$. 

  Thus, we obtain that $G_{\Gammabar}\argu[\theta]$ is non-singular with $\normLinfomega{G_{\Gammabar}\argu[\theta]^{-1}}$ uniformly bounded. Thus, the following is well defined
\begin{align*}
        (\gradgbarra \liftgbarra{\funss{\eta}})\circ \Xb\argu[\theta]&=\nabla \Xb\argu[\theta] G_{\Gammabar}\argu[\theta]^{-1}\nabla (\liftgbarra{\funss{\eta}}\circ \Xb\argu[\theta])\\
        &=\nabla \Xb\argu[\theta] G_{\Gammabar}\argu[\theta]^{-1}\nabla (\funss{\eta}\circ \Xs)\\
         &=\nabla \Xb\argu[\theta] G_{\Gammabar}\argu[\theta]^{-1}(\nabla \Xs)^T (\gradghs\funss{\eta})\circ \Xs,
   \end{align*}
and we can deduce~\eqref{eq: equivalence discrete norms}. 

Finally, by swapping the roles of $\Xs$ and $\Xh$ in the definition of $\Xb$, we deduce that the $\Lp(\Gammabart)$ and $\Hk(\Gammabart)$ norms are $h$-uniformly equivalent for all $\theta\in[0,1]$.

\end{proof}
\begin{remark}
    Note that the condition $\normLinfomega{(G_\Gammas)^{-1}(\nabla\Xs)^T\nabla\EX}<1$ implies that $\normLinfs{\gradghs\ex}<1$ 
    for the equivalence of discrete norms, using that $(\gradghs\ex)\circ \Xs=\nabla\EX(G_\Gammas)^{-1}(\nabla\Xs)^T$ and the equivalence of norms between pullback and pushforward. 
    This condition arises from the fact that    
   \[(G_\Gammas)^{-1}(\nabla\Xs)^T\nabla \Xb\argu[\theta]=I_2+\theta(\funss{G})^{-1}(\nabla\Xs)^T\nabla\EX.\]
\end{remark}

Before we prove the stability we need the following auxiliary results.
\begin{lemma}\label{Lem: estimates discrete bilinear forms} Let $\funss{\phi}$ and $\funss{\psi}$ be functions defined on $\Gammas$ and  let $\eta $ be defined on $\Gammah$, such that $ \norm{\funss{\phi}}_{\Lp[\infty](\Gammas)}\leq C$, the following quantities exist and \eqref{eq: uniform bound for Ex} holds, then
\begin{equation}\label{eq: bound funh - funs}
\begin{split}
\bigg|\intgammah {\eta}\lifth{\funss{\psi}} -\intgammas  \funss{\phi}\funss{\psi}\bigg|
\lesssim{}& \normLdosgs{\lifts{{ \eta}}-\funss{\phi}} \normLdosgs{\funss{\psi}}\\
&+\normLdosgs{\gradghs\ex} \normLdosgs{\funss{\psi}}.
\end{split}
\end{equation}
   If $ \norm{\funss{\phi}}_{\Wunoinf(\Gammas)}\leq C$ then,
   \begin{equation}\label{eq: bound Bhgradh - Bsgradhs}
       \begin{split}
   \Big|       \intgammah  \gradgh&\lifth{\funss{\phi}} \cdot \mathcal{B}(\vh) \gradgh\lifth{\funss{\psi}}
           -\intgammas\gradghs\funss{\phi}\cdot \mathcal{B}(\vs)\gradghs\funss{ \phi}\Big|
           \\
           &\lesssim \left(\norm{\gradghs \ex}_{\Lp(\Gammas)}+\normLdosgs{\gradghs \ev}\right)\normLdosgs{\gradghs\funss{\psi}},
       \end{split}
   \end{equation}
  where $\mathcal{B}(\vh)$ and $\mathcal{B}(\vs)$ are defined as in~\ref{eq: teo transp-grad-grad} with $\vh$ and $\vs$ instead of $\vel$.
\end{lemma}
\begin{proof}
    We consider that $t$ is fixed in $[0,T]$. Therefore, we omit the $t$ argument from the proof. Let $\Gammabart$ be as in Lemma \ref{Lem: equivalence of discrete norms}. Note that $\Xh$ and $\Xs$ are independent of $\theta$ and $\Gammabar(0)=\Gammat$ and $\Gammabar(1)=\Gammast$. This intermediate surface evolves with velocity 
    \begin{equation*}
      \vel_{\Gammabart}
      =\partial_{\theta}\Xb\circ\Xb^{-1} 
      = (\EX)\circ\Xb^{-1}=\liftgbarra{\ex}.
    \end{equation*}
Thus, using Lemma \ref{Lem: equivalence of discrete norms}, we have
    \begin{equation}\label{eq: bound velgbar}
      \norm{\vel_{\Gammabart}}_{\Wunoinf(\Gammabart)}\lesssim \norm{\ex}_{\Wunoinf(\Gammas)}\leq c,
    \end{equation}
for all $0\leq\theta\leq1$. Since $\liftgbarra{\funss{\phi}}$ is the transport of a $\theta$-independent function $\funss{\phi}\circ \Xs$ to $\Gammabart$, we can use~\eqref{eq: transport property} to get $\dermat \liftgbarra{\funss{\phi}}=0$. Then defining $ \bar{\bs\varepsilon}\argu[\theta]:=\liftgbarra{\funss{\phi}}+\theta\liftgbarra{\funss{\eta}-\funss{\bs\phi}}$ and using \eqref{eq: teo transp} 
\begin{equation*}
\begin{split}
(I)&=\intgammah {\eta}\lifth{\funss{\psi}} -\intgammas  \funss{\phi} \funss{\psi}=\int^0_1\frac{d}{d\theta}\int_{\Gammabar} \bar{\bs\varepsilon}\argu[\theta]\liftlambda{\funss{\psi}}\\
&=\int^0_1\int_{\Gammabar} \dermat \bar{\bs\varepsilon}\argu[\theta] \liftlambda{\funss{\psi}}+\bar{\bs\varepsilon}\argu[\theta] \liftlambda{\funss{\psi}}\Divgbar\vel_{\Gammabar}.
\end{split}
\end{equation*}
Considering that $\dermat \bar{\bs\varepsilon}\argu[\theta]=\liftgbarra{\funss{\eta}-\funss{\phi}}$, the regularity assumption in each case, H\"older's inequality and using Lemma \ref{Lem: equivalence of discrete norms}, the first estimate is obtained.



For the second estimate we apply piecewise \eqref{eq: teo transp},
\begin{align*}
(II)
={}& \intgammah  \gradgh\lifth{\funss{\phi}} \cdot \mathcal{B}(\vh) \gradgh\lifth{\funss{\psi}}
-\intgammas\gradghs\funss{\phi}\cdot \mathcal{B}(\vs)\gradghs\funss{ \phi}
\\
={}&\int_{0}^{1}\frac{d}{d \theta}\intgammabar\nabla_{\Gammabar} \liftlambda{\funss{\phi}}\cdot \mathcal{B}(\bar{\bs v}) \gradgbarra \liftlambda{\funss{\psi}}\\
={}&\int_{0}^{1}\intgammabar \dermat_\theta (\gradgbarra \liftlambda{\funss{\phi}}\cdot  \mathcal{B}(\bar{\bs v}) \gradgbarra \liftlambda{\funss{\psi}})\\
&+\int_{0}^{1}\intgammabar\Divgbar\vel_{\Gammabar}\gradgbarra \liftlambda{\funss{\phi}}\cdot \mathcal{B}(\bar{\bs v}) \gradgbarra  \liftlambda{\funss{\psi}},
\end{align*}
where we have defined $\bar{\bs v}\argu[\theta]:=\liftgbarra{\vs}+\theta \liftgbarra{\ev}$.
Then, for the first term we use the following formula  \cite[Lemma 2.6]{DKM2013},
	\begin{equation}\label{eq: dermat grad}
        \begin{split}
            \dermat (\gradgbarra g)=&\gradgbarra \dermat  g -(\gradgbarra \vel_{\Gammabar}^T -\n_{\Gammabar}  \tensor\n_{\Gammabar}  \gradgbarra \vel_{\Gammabar} )\gradgbarra g\\
            =&\gradgbarra \dermat  g - D(\vel_{\Gammabar},\n_{\Gammabar}) \gradgbarra g,
        \end{split}
	\end{equation} 
 where $\n_{\Gammabar}$ is the normal vector of $\Gammabar$, for which we have $\norm{\n_{\Gammabart}}_{\Lp[\infty](\Gammabart)}\leq C$ for all $\theta$ and for $h\leq h_0$ such that~\eqref{eq: uniform bound for Ex} is fulfilled. Furthermore, by~\eqref{eq: bound velgbar}, we have $\norm{D(\vel_{\Gammabar},\n_{\Gammabar})}_{\Lp[\infty](\Gammabart)}\leq C$ and $\norm{D(\vel_{\Gammabar},\n_{\Gammabar})}_{\Lp(\Gammabart)}\leq \normLdosgs{\gradghs\ex}$ for all $\theta\in [0,1]$.

 Thus, 
 \begin{equation*}
    \begin{split}
    \dermat (\gradgbarra &\liftlambda{\funss{\phi}} \cdot \mathcal{B}(\bar{\vel}\gradgbarra \liftlambda{\funss{\psi}}))\\
={}& \gradgbarra\dermat \liftlambda{\funss{\phi}}\cdot\mathcal{B}(\bar{\vel})\gradgbarra \liftlambda{\funss{\psi}} -\gradgbarra\liftlambda{\funss{\phi}}\cdot\mathcal{B}(\bar{\vel})D(\vel_{\Gammabar},\n_{\Gammabar}) \gradgbarra \liftlambda{\funss{\psi}}\\
&+ \gradgbarra \liftlambda{\funss{\phi}}\cdot\mathcal{B}(\bar{\vel})\gradgbarra \dermat \liftlambda{\funss{\psi}}-\gradgbarra \liftlambda{\funss{\phi}}\cdot\mathcal{B}(\bar{\vel})D(\vel_{\Gammabar},\n_{\Gammabar}) \gradgbarra \liftlambda{\funss{\psi}}\\
&+ \gradgbarra \liftlambda{\funss{\phi}}\cdot \dermat \mathcal{B}(\bar{\vel})\gradgbarra \liftlambda{\funss{\psi}}\\
={}&-2\gradgbarra \liftlambda{\funss{\phi}}\cdot\mathcal{B}(\bar{\vel})D(\vel_{\Gammabar},\n_{\Gammabar}) \gradgbarra \liftlambda{\funss{\psi}}+ \gradgbarra \liftlambda{\funss{\phi}}\cdot \dermat \mathcal{B}(\bar{\vel})\gradgbarra \liftlambda{\funss{\psi}},
    \end{split}
\end{equation*}   
where we have used that $\dermat_\theta \liftgbarra{\funss{\phi}}=\dermat_\theta \liftgbarra{\funss{\psi}}=0$.

Now, we use the fact that both the quantities related to the normal vector to $\Gammabar$ and its velocity are uniformly bounded. That is, \( \norm{\n_{\Gammabart}}_{\Lp[\infty](\Gammabart)}\leq C \) for all \( \theta \) and for \( h\leq h_0 \) such that~\eqref{eq: uniform bound for Ex} holds. Then,
$\norm{D(\vel_{\Gammabar},\n_{\Gammabar})}_{\Lp[\infty](\Gammabart)}\leq C$ for all $\theta\in [0,1]$, which is guaranteed by~\eqref{eq: bound velgbar}.  

Finally, we will use the fact that $\norm{D(\vel_{\Gammabar},\n_{\Gammabar})}_{\Lp(\Gammabart)}\leq \normLdosgs{\gradghs\ex}$ for all $\theta\in [0,1]$, thanks to~\eqref{eq: bound velgbar}. Together with  
\begin{align*}
    \norm{\mathcal{B}(\bar{\vel})}_{\Lp[\infty](\Gammabart)}\leq 3 \norm{\gradgbarra\bar{\vel}}_{\Lp[\infty](\Gammabart)}\leq C\normLinfs{\gradghs\vs}\leq C,
\end{align*}
~\eqref{eq: equivalence discrete norms} and the assumptions about $\testnuhs$ we deduce
\begin{align*}
|II|\lesssim &\norm{D(\vel_{\Gammabar},\n_{\Gammabar})}_{\Lp(\Gammabart)}\normLdosgs{\gradghs\funss{\psi}}+\norm{\dermat_\theta\mathcal{B}(\bar{\vel})}_{\Lp(\Gammabart)}\normLdosgs{\gradghs\funss{\psi}}\\
&+ \norm{\gradgbarra\vel_{\Gammabar}}_{\Lp(\Gammabart)}\normLdosgs{\gradghs\funss{\psi}}\\
\lesssim &\left(\norm{\gradghs \ex}_{\Lp(\Gammas)}+\normLdosgs{\gradghs \ev}\right)\normLdosgs{\gradghs\funss{\psi}},
\end{align*}
which is the second assertion of this lemma.
For last inequality we have used that 
\begin{align*}
    \dermat_\theta\mathcal{B}(\bar{\vel})&=\dermat_\theta(\Divgbar\bar{\vel}-(\gradgbarra\bar{\vel}+(\gradgbarra\bar{\vel}^T)))\\
    &=\mathcal{B}(\dermat_\theta \bar{\vel})-\text{tr}(F(\theta))+F(\theta)+F(\theta)^T\\
    &=\mathcal{B}(\liftgbarra{\ev})-\text{tr}(F(\theta))+F(\theta)+F(\theta)^T,
\end{align*}
where $F(\theta):=\gradgbarra\bar{\vel}((\gradgbarra\vel_{\Gammabar})^T-\n_{\Gammabar}\tensor\n_{\Gammabar}\gradgbarra\vel_{\Gammabar})$. And thus, 
\begin{align*}
    \norm{\dermat_\theta\mathcal{B}(\bar{\vel})}_{\Lp(\Gammabart)}&\leq 3\norm{\gradgbarra\liftgbarra{\ev}}_{\Lp(\Gammabart)}+3\norm{F(\theta)}_{\Lp(\Gammabart)}\\
    &\lesssim \normLdosgs{\gradghs \ev}+\norm{\gradghs \ex}_{\Lp(\Gammas)}.
\end{align*}
\end{proof}

If we consider $\mathcal{B}(\bs w)$ to be the identity matrix for all $\bs w$ in~\eqref{eq: bound Bhgradh - Bsgradhs}, then $\mathcal{B}(\bs w)$ is independent of $\theta$, whence $\dermat_\theta \mathcal{B}(\bs w) = 0$ for all $\bs w$. This leads to the first result of the following Corollary.
And, if we consider \( \mathcal{B}(\bs w) \) equal to the identity matrix for all \( \bs w \) in the inequality~\eqref{eq: bound Bhgradh - Bsgradhs}, then \( \mathcal{B}(\bs w) \) is independent of \( \theta \), whence \( \dermat_\theta \mathcal{B}(\bs w)=0 \) for all \( \bs w \), and we will have the second inequality of the following Corollary.  
\begin{corollary}\label{coro: bound gradh  gradhs}
    Let $\funss{\phi}$ and $\funss{\psi}$ be as in Lemma~\ref{Lem: estimates discrete bilinear forms}. If $ \norm{\funss{\phi}}_{\Wunoinf(\Gammas)}\leq C$ then, for $0<h\leq h_0$ such that~\eqref{eq: uniform bound for Ex} fulfilled,
    \begin{equation}\label{eq: bound funslifth - funs}
        \begin{split}
      \left|  \intgammah \lifth{\funss{\phi}}\lifth{\funss{\psi}} -\intgammas  \funss{\phi}\funss{\psi}\right|
        \lesssim{}
        &\normLdosgs{\gradghs\ex} \normLdosgs{\funss{\psi}}.
        \end{split}
        \end{equation}
  and
    \begin{equation}\label{eq: bound gradh - gradhs}
        \begin{split}
       \Big|     \intgammah 
             \gradgh\lifth{\funss{\phi}} \cdot \gradgh\lifth{\funss{\psi}}
            -&\intgammas\gradghs\funss{\psi}\cdot \gradghs\funss{ \phi}\Big|\\
            &\lesssim \norm{\gradghs \ex}_{\Lp(\Gammas)}\normLdosgs{\gradghs\funss{\psi}}.
        \end{split}
    \end{equation} 
\end{corollary}
Note that in Lemma~\ref{Lem: estimates discrete bilinear forms} and Corollary~\ref{coro: bound gradh gradhs} we are considering discrete surfaces but not discrete functions.

In summary, to conclude this section, it is worth noting the following:

\begin{quote}\it
If the parametrizations of two surfaces parametrized over the same reference domain are sufficiently close, then we will have norm equivalence between functions and their corresponding lifts. This allows us to estimate differences between bilinear forms as in Corollary~\ref{coro: bound gradh gradhs}.
\end{quote}

This statement is made more precise in the following Corollary.
\begin{corollary}\label{coro: equivalent norms and bounds bilinear forms}
    Let $\Lambda$ and $\Sigma$ be two sufficiently smooth surfaces, parameterized by $\X_{\Lambda}:\Omega\to\Rn[3]$ and $\X_{\Sigma}:\Omega\to\Rn[3]$. Let $f$ and $g$ be two functions defined on $\Sigma$. If  
    \begin{equation}\label{eq: Error parameterizations}
        \normWunomega{\X_{\Lambda}-\X_{\Sigma}}\leq C,
    \end{equation}
    then the norms $\Lp(\Lambda)$ and $\Lp(\Sigma)$, as well as the $\Hk$ norms, are equivalent. Moreover, if $\norm{f}_{\Wunoinf(\Sigma)}\leq C$, then  
    \begin{equation}
        \begin{split}
       \left| \int_{\Lambda} (f)^{\Lambda}\,(g)^{\Lambda} -\int_{\Sigma} f\, g\right|
        \lesssim{}
        &\norm{\nabla_{\Sigma}e_{\X}}_{\Lp(\Sigma)} \norm{g}_{\Lp(\Sigma)} .
        \end{split}
    \end{equation}
    and  
    \begin{equation}
        \begin{split}
         \Big|   \int_{\Lambda}
            \nabla_{\Lambda} (f)^{\Lambda}\, \cdot \nabla_{\Lambda}(g)^{\Lambda} 
            -&\int_{\Sigma} \nabla_{\Sigma}f\cdot \nabla_{\Sigma}g\Big|\lesssim \norm{\nabla_{\Sigma}e_{\X}}_{\Lp(\Sigma)} \norm{\nabla_{\Sigma}g}_{\Lp(\Sigma)} .
        \end{split}
    \end{equation} 
    where the lifts are defined via the parameterizations, i.e., $(f)^{\Lambda}=f\circ\X_{\Sigma}\circ(\X_{\Lambda})^{-1}$.
\end{corollary}
\begin{proof}
    The proof consists of following the same steps as in Lemma~\ref{Lem: equivalence of discrete norms}, and then applying the same steps as in Corollary~\ref{coro: bound gradh gradhs}.
\end{proof}

\subsection{Stability}\label{sec: stability}
In this section, we present the stability analysis of the semi-discrete scheme. The main result is Proposition~\ref{prop: stability}. 

We start with three lemmas that contain key energy estimates for the stability proof. However, before that, we present an auxiliary Lemma which will play a key role in the estimates of this section.

\begin{lemma}\label{lem: bounds for numerical approx}
If $p\geq2$, there exists $h_0>0$ and $t_0\in[0,T]$ such that 
        \begin{equation}\label{eq: bounds for numerical approx}
        \begin{split}
            \normWunost{\ex\argu}&\leq h^{(p-1)/2},\\
            \normWunost{\ev\argu}&\leq h^{(p-1)/2},\\
            \normWunost{\eH\argu}&\leq h^{(p-1)/2},\\
            \normWunost{\en\argu}&\leq h^{(p-1)/2},\\
        \end{split}
    \end{equation}
    for all $0<h\leq h_0$ and all $t\in [0,t_0]$. 
    Moreover, there exists $C$ independent of $h$ and $t$ such that
    \begin{equation}\label{eq: uniform bound numerical approx}
         \begin{gathered}
          \normWunost{\idh}+\normWunost{\vh}+\normWunost{\Hh}+\normWunost{\nh}\leq C,
     \end{gathered}
    \end{equation}
    for all $0<h\leq h_0$ and all $t\in [0,t_0]$. 
\end{lemma}

The idea of this proof is analogous to those used in~\cite{KLL2019}, but we include it here for completeness.

\begin{proof}
    Notice that $\ex\argu[0]=0$, $\ev\argu[0]=0$ and also $\en(0)=0$ because $\n_{h,0} = \Rs[\Gammas] \n_0$. Also, recalling that $p>1$, $\normWunost{\eH\argu[0]}\lesssim h^{p-1} \le h^{(p-1)/2}$ for $h$ sufficiently small.
    By continuity, for a given $h_0>0$ there exists $t_0>0$ such that $t_0$ is the maximum time at which all the inequalities in~\eqref{eq: bounds for numerical approx} are satisfied for all $t\in[0,t_0]$, and $0<h\le h_0$. 

Finally, to see the uniform bounds of the numerical solutions, we proceed as follows for $(u_h,u^*)\in\{(\idh,\ids),(\vh,\vs),(\Hh,\Hs),(\nh,\ns)\}$:
\[
    \normWunost{u_h}\leq \normWunost{u_h-u^*}+\normWunost{u^*}.
\]
We now use~\eqref{eq: bounds for numerical approx} to uniformly bound the first term for $h\leq h_0$. For the second term, the uniform bounds of the different operators defined in Section~\ref{sec: spline approximation estimates} are used, that is, Lemma~\ref{Lem: uniform bounds xs vs hs nus}, along with the equivalence of norms.
\end{proof}

\begin{remark}\label{rem: implicancias de uniform bound EX}
    From Lemma~\ref{lem: bound for splines parameterizations}, we know there exists \( h_0>0 \) such that~\eqref{eq: bound EX} is satisfied for all \( t\in[0,t_0] \). Then, the condition~\eqref{eq: uniform bound for Ex} in Lemma~\ref{Lem: equivalence of discrete norms} is satisfied for \( 0<h\leq h_0 \) and \( 0\leq t \leq t_0 \). Therefore, the statements of Lemma~\ref{Lem: estimates discrete bilinear forms} and Corollary~\ref{coro: bound gradh gradhs} hold for \( 0<h\leq h_0 \) and \( 0\leq t \leq t_0 \).
\end{remark}

\begin{lemma}\label{Lem: energy estimate ex ev}
    Let $t_0>0$ and $h_0>0$ be such that Lemma~\ref{lem: bound for splines parameterizations} is satisfied. Then, for all $t\in[0,t_0]$ and all $0<h<h_0$ we have
    \begin{equation}\label{eq: energy estimate ex}
        \dt\normcuad{\ex}_{\Hunogammas}\lesssim\normcuad{ \ev}_{\Hunogammas} + \normcuad{\ex}_{\Hunogammas},
    \end{equation}
    and 
    \begin{align}\label{eq: energy estimate ev}
        \normHunogs{\ev}^2 &\lesssim \normHunogs{\eH}^2+ \normHunogs{\en}^2+ \normHunogs{\dv}^2. 
    \end{align}
\end{lemma}
\begin{proof}
    Let $t\in[0,t_0]$, and $0<h\le h_0$. 
    Then Lemma~\ref{Lem: equivalence of discrete norms}~(\ref{item: equivalence norms}) is fulfilled and we have equivalence of norms.
    Using the transport formula~\eqref{eq: teo transp} the fact that $\dermat \ex =\ev$, we have
    \begin{align*}
        \dt\normcuad{\ex}_{\Hunogammas} ={}& 
        2\intgammas \ex\cdot\ev + \ex\cdot\ex\, \Divgs\vs \\
        & + 2\intgammas \gradghs\ex : \gradghs\ev 
        + \gradghs\ex : \mathcal{B}(\vs)\gradghs\ex,
    \end{align*}
    Young's inequality and the uniform bound for $\vs$ from Lemma~\ref{Lem: uniform bounds xs vs hs nus} lead to~\eqref{eq: energy estimate ex}. 
    Considering that $\lifts{\Qs[\Gammah] (-\Hh\nh)}=\Qs[\Gammas]\lifts{-\Hh\nh}$, we get the following estimate for $\ev$
    \begin{align*}
        \normHunogs{\ev} &= \normHunogs{\Qs[\Gammas] (\lifts{-\Hh\nh}-\Hs\ns)}+  \normHunogs{\dv}\\
        &\lesssim \normHunogs{\lifts{-\Hh\nh}-\Hs\ns}+  \normHunogs{\dv}. 
        \end{align*}
    The uniform bound for the Ritz projections 
    from Sections~\ref{sec: Linear Ritz projection} and~\ref{sec: Nonlinear Ritz projection}, yields the remaining assertion~\eqref{eq: energy estimate ev}.
   
\end{proof}
\begin{lemma}\label{Lem: energy estimate eH}
    Let $t_0>0$ and $h_0>0$ be such that Lemma~\ref{lem: bound for splines parameterizations} is satisfied. Then, for all $t\in[0,t_0]$ and all $0<h<h_0$ we have
    \begin{equation}\label{eq: energy estimate eH}
        \begin{split}
            \dt\normcuad{  \eH}_{\Hunogammas}
            \lesssim{} &
        {-\frac{d}{d t}\left(\int_{\Gammah} \gradgh \lifth{\Hs}\cdot \gradgh  \lifth{\eH}-\intgammas \gradghs \Hs\cdot \gradgh  \eH\right)}\\
        &+\normcuad{\en}_{\Hunogammas}+\normcuad{\ex}_{\Hunogammas}+\normcuad{\ev}_{\Hunogammas}\\
        & +\normcuad{\eH}_{\Hunogammas}+\normcuad{\dH}_{\Ldosgammas}.
        \end{split}
        \end{equation}
\end{lemma}
\begin{proof}
    Since $\dermat \lifth{\eH}=\lifth{\dermat \eH}$, substituting $\testHh= \lifth{\dermat \eH}$ and $\testHhs =\dermat \eH$ in equation~\eqref{eq: error-curvature} we obtain
\begin{equation*}
\begin{split}
\| \lifth{\dermat \eH}\|^2_{\Ldosgammah} &+ \intgammah \gradgh \lifth{\eH}\cdot \gradgh \lifth{\dermat \eH}\\
={}&-\left({\intgammah \gradgh \lifth{\Hs}\cdot \gradgh \lifth{\dermat \eH}-\intgammas \gradghs \Hs\cdot \gradghs \dermat \eH}\right)\\
&-\left({\intgammah \dermat (\Hs)^{\Gammah} \lifth{\dermat \eH} -\intgammas \dermat \Hs \dermat \eH}\right)\\
&+{\int_{\Gammah} |A_h|^2 \Hh \lifth{\dermat \eH}-\intgammas |A_h^*|^2 \Hs \dermat \eH}+\intgammas \dH \dermat \eH.
\end{split}
\end{equation*}
Applying Lemma~\ref{Lem: estimates discrete bilinear forms} to the second difference and Corollary~\ref{coro: bound gradh gradhs} to the third one, we get
\begin{equation*}
    \begin{split}
    \| \lifth{\dermat \eH}&\|^2_{\Ldosgammah} + \intgammah \gradgh \lifth{\eH}\cdot \gradgh \lifth{\dermat \eH}\\
    \lesssim{}& \underbrace{-\left(\intgammah\gradgh \lifth{\Hs}\cdot \gradgh \lifth{\dermat \eH}-\intgammas \gradghs \Hs\cdot \gradghs \dermat \eH\right)}_{(A)}
 \\
 &+\left(\normLdosgs{\gradghs\ex}+\normLdosgs{\lifts{|A_h|^2\Hh}-|\funss{A}|^2\Hs}\right)\normLdosgs{\dermat \eH}\\
 &+\normLdosgs{\dH}\normLdosgs{\dermat \eH},
\end{split}
\end{equation*}
To bound~$(A)$ we use~\eqref{eq: teo transp-grad-grad} to rewrite it as follows
\begin{align*}
(A)
={}&-\frac{d}{d t}\left(\int_{\Gammah} \gradgh \lifth{\Hs}\cdot \gradgh  \lifth{\eH}-\intgammas \gradghs \Hs\cdot \gradgh  \eH\right)\\
&+\left(\int_{\Gammah}\gradgh \dermat \lifth{\Hs}\cdot \gradgh \lifth{\eH}-\intgammas\gradghs \dermat \Hs\cdot \gradgh \eH\right)\\
&+\left(\intgammah\gradgh \lifth{\Hs}\cdot \mathcal{B}(\vh) \gradgh  \lifth{\eH}-\intgammas\gradghs \Hs\cdot \mathcal{B}(\vs) \gradgh  \eH\right)\\
\lesssim{}& -\frac{d}{d t}\left(\int_{\Gammah} \gradghs \lifth{\Hs}\cdot \gradgh  \lifth{\eH}-\intgammas \gradgh \Hs\cdot \gradgh  \eH\right)\\
&+\norm{\gradghs \ex}_{\Ldosgammas}\norm{ \gradghs \eH}_{\Ldosgammas}+\norm{\gradghs{\ev}}_{\Ldosgammas}\norm{\gradghs {\eH}}_{\Ldosgammas},
\end{align*}
where we have used~\eqref{eq: bound gradh - gradhs} and~\eqref{eq: bound Bhgradh - Bsgradhs}.

For the left-hand side, using~\eqref{eq: teo transp-grad-grad} we have 
\begin{align*}
&\normcuad{\lifth{\dermat \eH}}_{\Ldosgammah} + \intgammah \gradgh \lifth{\eH}\cdot \gradgh \lifth{\dermat \eH}\\
&\qquad\geq \| \lifth{\dermat \eH} \|^2_{L^2(\Gammah)} +\frac{1}{2}\dt\normcuad{ \gradgh \lifth{\eH}}_{L^2(\Gammah)}-c\| \gradgh \lifth{\eH}\|^2_{L^2(\Gammah)}.
\end{align*}
Thus, the combination of the above estimates yields
\begin{equation*}
\begin{split}
&\hspace{-20pt}\normcuad{\lifth{\dermat \eH} }_{\Ldosgammah} +\frac{1}{2}\dt \normcuad{\gradgh \lifth{\eH}}_{\Ldosgammah}\\
\lesssim {}&
{-\dt\left(\int_{\Gammah} \gradgh \lifth{\Hs}\cdot \gradgh  \lifth{\eH}-\intgammas \gradghs \Hs\cdot \gradgh  \eH\right)}\\
&+ \norm{\gradghs {\ev}}_{\Ldosgammas}\norm{\gradghs {\eH}}_{\Ldosgammas}+ \norm{\gradghs\ex}_{\Ldosgammas}\norm{\gradghs{\eH}}_{\Ldosgammas}\\
&+\left(\normLdosgs{\gradghs\ex}+\normLdosgs{\lifts{|A_h|^2\Hh}-|\funss{A}|^2\Hs}\right)\normLdosgs{\dermat \eH}\\
 &+\normLdosgs{\dH}\normLdosgs{\dermat \eH}+\normcuad{ \gradghs \eH}_{\Ldosgammas}.
\end{split}
\end{equation*}
Finally, using 
\begin{align*}
    \normLdosgs{\lifts{|A_h|^2\Hh}-|\funss{A}|^2\Hs}\lesssim  \normLdosgs{\gradghs\en}+ \normLdosg{\gradghs\eH},
\end{align*}
and 
\begin{equation*}
    \frac{1}{2}\dt\| \lifth{\eH}\|^2_{\Ldosgammah}\leq\| \dermat \lifth{\eH} \|^2_{\Ldosgammah} +c \|\lifth{\eH} \|^2_{\Ldosgammah},
\end{equation*} together with equivalence of norms and   
Young's inequality we arrive at the desired estimate.
\end{proof}

\begin{lemma}\label{Lem: energy estimate en}
    Let $t_0>0$ and $h_0>0$ be such that Lemma~\ref{lem: bound for splines parameterizations} is satisfied. Then, for all $t\in[0,t_0]$ and all $0<h<h_0$ we have
    \begin{equation}\label{eq: energy estimate en}
        \begin{aligned}
            \dt\normcuad{\en}_{\Hunogammas}&+ \dt\normcuad{\dermat \lifth{\en}}_{\Ldosgammah} + \normcuad{\gradgh\dermat \lifth{\en}}_{\Ldosgammah}\\
            \lesssim 
            -&\dt\left(\intgammah \gradgh \lifth{\ns} \cdot \gradgh  \lifth{\en} 
            - \intgammas \gradghs \ns \cdot \gradgh \en\right) \\
            + &\normcuad{\en}_{\Hunogammas} + \normcuad{\ex}_{\Hunogammas} + \normcuad{\ev}_{\Hunogammas} \\
            + &\normLdosgs{\dermat\en}^2+\normcuad{\dermat\dn}_{\Ldosgammas} + \normcuad{\dn}_{\Ldosgammas}.
        \end{aligned}
     \end{equation}
\end{lemma}
\begin{proof}
    We proceed by setting $\testnuhs=\dermat \en$ and $\testnuh= \lifth{\dermat\en}$ in~\eqref{eq: error-normal} and following similar steps to those employed in Lemma~\ref{Lem: energy estimate eH} we arrive at
    \begin{equation*}
        \begin{split}
            \dt\normcuad{ \en}_{\Hunogammas}
            \lesssim {}&
        {-\dt\left(\intgammah \gradgh \lifth{\ns}\cdot \gradgh  \lifth{\en}-\intgammas \gradghs \ns\cdot \gradgh  \en\right)}\\
        &+\normcuad{\en}_{\Hunogammas}+\normcuad{\ex}_{\Hunogammas}+\normcuad{\ev}_{\Hunogammas}\\
        &+ \int_{\boundaryGhcero} \left(\alpha_{\partial,h}\conorh - \alpha_{\partial}^*\conors\right)\cdot \dermat\en 
        +\normcuad{\dn}_{\Ldosgammas}.
        \end{split}
        \end{equation*}
        The main difference with the curvature error estimation is on the boundary terms, where we now focus our attention.
By adding and subtracting $\alpha_{\partial}^*\conorh=\ns\cdot \curvbdryh(\nh\times \btauh)$,
we have
\begin{align*}
    \int_{\boundaryGhcero} \left(\alpha_{\partial,h}\conorh - \alpha_{\partial}^*\conors\right)\cdot \dermat\en
    &=\int_{\boundaryGhcero} \left(\en\cdot \curvbdryh\conorh - \alpha_{\partial}^*(\en\times \btauh)\right)\cdot \dermat\en\\
    &\lesssim \norm{\en}_{\Lp(\boundaryGhcero)}\norm{\dermat\en}_{\Lp(\boundaryGhcero)}.
\end{align*}
Then, using Young's inequality 
and trace inequality to bound $\norm{\en}_{\Lp(\boundaryGhcero)}$, 
we get
\begin{equation}\label{eq: error equation for en parte I}
    \begin{split}
        \dt\normcuad{  \en}_{\Hunogammas}
        \lesssim {}&
    {-\dt\left(\intgammah \gradgh \lifth{\ns}\cdot \gradgh  \lifth{\en}-\intgammas \gradghs \ns\cdot \gradgh  \en\right)}\\
    &+\normcuad{\en}_{\Hunogammas}+\normcuad{\ex}_{\Hunogammas}+\normcuad{\ev}_{\Hunogammas}\\
    &+
    \normcuad{\dermat \lifth{\en}}_{\Lp(\boundaryGhcero)} +\normcuad{\dn}_{\Ldosgammas}.
    \end{split}
    \end{equation}

To deal with the critical term $\norm{\dermat\en}_{\Lp(\boundaryGhcero)}$ we are going to deduce a new energy estimate. For that, we differentiate in time~\eqref{eq: error-normal} and use~\eqref{eq: teo transp} and~\eqref{eq: teo transp-grad-grad} to get the following
\begin{equation}\label{eq: dt error-normal}
    \begin{split}
        \intgammah &\dermat\dermat \lifth{\en} \cdot\testnuh +\intgammah \gradgh\dermat \lifth{\en} :\gradgh\testnuh\\
        &         + \intgammah \dermat\lifth{\en}\cdot\testnuh \Divgh\vh+\intgammah \gradgh \lifth{\en}: \mathcal{B}(\vh)\gradgh \testnuh \\
        =&-\left(\intgammah\gradgh \dermat(\ns)^{\Gammah}: \gradgh \testnuh-\intgammas\gradghs \dermat\ns: \gradgh \testnuhs\right)\\
        &-\left(\intgammah \gradgh \lifth{\ns}: \mathcal{B}(\vh)\gradgh \testnuh-\intgammas \gradghs \ns: \mathcal{B}(\vs)\gradghs \testnuhs\right)\\
    &-\left(\intgammah \dermat\dermat \lifth{\ns}\cdot \testnuh-\intgammas \dermat\dermat \ns \cdot\testnuhs\right)\\
    &-\left(\intgammah \dermat \lifth{\ns}\cdot \testnuh \Divgh\vh-\intgammas \dermat\ns\cdot \testnuhs \Divgs\vs\right)\\
    &+\intgammah \dermat(|A_h|^2 \nh) \cdot\testnuh-\intgammas \dermat(|A_h^*|^2 \ns)\cdot \testnuhs\\
    &+\intgammah|A_h|^2 \nh \cdot\testnuh\Divgh\vh-\intgammas |A_h^*|^2\ns\cdot \testnuhs\Divgs\vs\\
    &+ \int_{\boundaryGhcero} \partial_t(\alpha_{\partial,h}\conorh)\cdot \testnuh-\int_{\boundaryGhcero} \partial_t(\alpha_{\partial}^*\conors)\cdot \testnuhs+\dt\intgammas \dn \cdot\testnuhs\\
    \end{split}
    \end{equation} 
    
        We observe that the first term in the left hand side satisfies
    \begin{align*}
        \intgamma \dermat&\dermat \lifth{\en} \cdot\dermat \lifth{\en}=\dfrac{1}{2}\dt\normcuad{\dermat \lifth{\en}}_{\Ldosgammah}- \dfrac{1}{2}\intgammah|\dermat \lifth{\en}|^2\Divgh\vh.
    \end{align*}
   Then, setting $\testnuhs=\dermat \en$ and $\testnuh= \lifth{\dermat\en}$ and using~\eqref{eq: teo transp} once more to differentiate the last term, we obtain
\begin{equation*}
    \begin{split}
        \dfrac{1}{2}\dt&\normcuad{\dermat \lifth{\en}}_{\Ldosgammah} + \normcuad{\gradgh\dermat \lifth{\en}}_{\Ldosgammah} 
        \\
       ={}& -\intgammah \gradgh \lifth{\en}: \mathcal{B}(\vh)\gradgh\dermat \lifth{\en}
       +\dfrac{1}{2}\intgammah|\dermat \lifth{\en}|^2\Divgh\vh
       \\
        &-\left(\intgammah\gradgh \dermat(\ns)^{\Gammah}: \gradgh \lifth{\dermat\en}-\intgammas\gradghs \dermat\ns: \gradghs \dermat\en\right)\\
        &-\left(\intgammah \gradgh \lifth{\ns}: \mathcal{B}(\vh)\gradgh \lifth{\dermat \en}-\intgammas \gradghs \ns: \mathcal{B}(\vs)\gradghs\dermat \en\right)\\
    &-\left(\intgammah \dermat\dermat \lifth{\ns}\cdot \lifth{\dermat\en}-\intgammas \dermat\dermat \ns \cdot\dermat\en\right)\\
    &-\left(\intgammah \dermat \lifth{\ns}\cdot \lifth{\dermat\en} \Divgh\vh-\intgammas \dermat\ns\cdot \dermat\en \Divgs\vs\right)\\
    &+\intgammah \dermat(|A_h|^2 \nh) \cdot\lifth{\dermat\en}-\intgammas \dermat(|A_h^*|^2 \ns)\cdot \dermat\en\\
    &+\intgammah|A_h|^2 \nh \cdot\lifth{\dermat\en}\Divgh\vh-\intgammas |A_h^*|^2\ns\cdot \dermat\en\Divgs\vs\\
    &+ \int_{\boundaryGhcero} \partial_t(\alpha_{\partial,h}\conorh)\cdot \lifth{\dermat\en}-\int_{\boundaryGhcero} \partial_t(\alpha_{\partial}^*\conors)\cdot \dermat\en\\
    &+\intgammas \dermat\dn \cdot\dermat\en+\intgammas \dn \cdot\dermat\en\Divgs\vs.
    \end{split}
    \end{equation*}

By the uniform bound of $\vh$ and equivalence of norms~\eqref{eq: equivalence discrete norms} we have
\begin{align*}
    -\intgammah &\gradgh \lifth{\en}: \mathcal{B}(\vh)\gradgh\dermat \lifth{\en}
    +\dfrac{1}{2}\intgammah|\dermat \lifth{\en}|\Divgh\vh\\
    &\lesssim \normLdosgh{ \gradghs \lifth{\en}}\normLdosgh{\gradgh\dermat \lifth{\en}}+\normLdosgs{ \dermat\en}^2.
\end{align*}

Applying 
Corollary~\ref{coro: bound gradh gradhs}, Lemma~\ref{Lem: equivalence of discrete norms} and the uniform bounds~\eqref{eq: uniform bound nonlinear dermat Ru},~\eqref{eq: uniform bound numerical approx}, we obtain
\begin{align*}
    &-\left(\intgammah\gradgh \dermat(\ns)^{\Gammah}: \gradgh \lifth{\dermat\en}-\intgammas\gradghs \dermat\ns: \gradghs \dermat\en\right)\\
    &-\left(\intgammah \gradgh \lifth{\ns}: \mathcal{B}(\vh)\gradgh \lifth{\dermat \en}-\intgammas \gradghs \ns: \mathcal{B}(\vs)\gradghs\dermat \en\right)\\
&-\left(\intgammah \dermat\dermat \lifth{\ns}\cdot \lifth{\dermat\en}-\intgammas \dermat\dermat \ns \cdot\dermat\en\right)\\
&-\left(\intgammah \dermat\dermat \lifth{\ns}\cdot \lifth{\dermat\en}-\intgammas \dermat\dermat \ns \cdot\dermat\en\right)\\
&\qquad\leq \left(\normLdosgs{\gradghs\ex}+\normLdosgs{\gradghs\ev}\right)\normLdosgs{\gradghs\dermat\en},
\end{align*}
and 
\begin{align*}
    &\intgammah \dermat(|A_h|^2 \nh) \cdot\lifth{\dermat\en}-\intgammas \dermat(|A_h^*|^2 \ns)\cdot \dermat\en\\
    &+\intgammah|A_h|^2 \nh \cdot\lifth{\dermat\en}\Divgh\vh-\intgammas |A_h^*|^2\ns\cdot \dermat\en\Divgs\vs\\
    &\lesssim\left(\normLdosgs{\lifts{\dermat(|A_h|^2 \nh)}-\dermat(|A_h^*|^2 \ns)}+\normLdosgs{\gradghs\ex}\right)\normLdosgs{\dermat\en}\\
    &+\left(\normLdosgs{\lifts{|A_h|^2\nh\Divgh\vh }-|A_h^*|^2 \ns\Divgs\vs}+\normLdosgs{\gradghs\ex}\right)\normLdosgs{\dermat\en}.
\end{align*}

Thus, by combining all of the aforementioned elements, we obtain
\begin{equation*}
    \begin{split}
        \dfrac{1}{2}&\dt\normcuad{\dermat \lifth{\en}}_{\Ldosgammah} + \normcuad{\gradgh\dermat \lifth{\en}}_{\Ldosgammah}\\
       &\lesssim \normLdosgh{ \gradgh \lifth{\en}}\normLdosgh{\gradgh\dermat \lifth{\en}}\\
          &+\left(\normLdosgs{\gradghs\ex}+\normLdosgs{\gradghs\ev}\right)\normLdosgs{\gradghs\dermat\en}\\
    &+\left(\normLdosgs{\lifts{\dermat(|A_h|^2 \nh)}-\dermat(|A_h^*|^2 \ns)}+\normLdosgs{\gradghs\ex}\right)\normLdosgs{\dermat\en}\\
    &+\left(\normLdosgs{\lifts{|A_h|^2\nh\Divgh\vh }-|A_h^*|^2 \ns\Divgs\vs}+\normLdosgs{\gradghs\ex}\right)\normLdosgs{\dermat\en}\\
    &+ \int_{\boundaryGhcero} \partial_t(\alpha_{\partial,h}\conorh)\cdot \lifth{\dermat\en}-\int_{\boundaryGhcero} \partial_t(\alpha_{\partial}^*\conors)\cdot \dermat\en\\
    &+\norm{\dermat\dn}_{\Ldosgammas}\normHunogs{\dermat\en}+\normLdosgs{ \dn} \normLdosgs{\dermat\en}.
    \end{split}
    \end{equation*}

We now recall that
    \begin{align*}
     \dermat (|A_h|^2)&=2(\gradg (\dermat\nh):A_h-A_hD_h:A_h),\\
     \dermat (|\funss{A}|^2)&=2(\gradghs (\dermat\ns):\funss{A}-\funss{A}\funss{D}:\funss{A}),
 \end{align*}
 where $D_h=((\gradg \vh)^T-\n_{\Gammah}\tensor \n_{\Gammah}\gradg \vh)$ and similarly for $\funss{D}$, with $\n_{\Gammas}$ and $\n_{\Gammah}$ the normal vectors of $\Gammas$ and $\Gammah$ respectively. Therefore,
\begin{align*}
    \normLdosgs{\lifts{\dermat(|A_h|^2 \nh)}-\dermat(|A_h^*|^2 \ns)}\lesssim\normLdosgs{\gradghs\dermat\en}+&\normLdosgs{\gradghs\en}\\
     +\normLdosgs{\gradghs\ev}+&\normLdosgs{\dermat\en},
\end{align*}
and
\[
    \normLdosgs{\lifts{|A_h|^2\Divgh\vh }-|A_h^*|^2 \ns\Divgs\vs}\lesssim\normHunogs{\en}+\normLdosgs{\gradghs\ev}.
\]

Summarizing,
\begin{equation*}
    \begin{split}
        \dfrac{1}{2}\dt&\normcuad{\dermat \lifth{\en}}_{\Ldosgammah} + \normcuad{\gradgh\dermat \lifth{\en}}_{\Ldosgammah}\\
       \lesssim{}& \normLdosgh{ \gradgh \lifth{\en}}\normLdosgh{\gradgh\dermat \lifth{\en}}\\
          &+\left(\normLdosgs{\gradghs\ex}+\normLdosgs{\gradghs\ev}\right)\normLdosgs{\gradghs\dermat\en}\\
    &+\left(\normLdosgs{\gradghs\dermat\en}+\normHunogs{\en}
    +\normLdosgs{\gradghs\ev}+\right)\normLdosgs{\dermat\en}\\
    &+\left(\normLdosgs{\dermat\en}+\normLdosgs{\gradghs\ex}\right)\normLdosgs{\dermat\en}\\
    &+ \int_{\boundaryGhcero} \partial_t(\alpha_{\partial,h}\conorh)\cdot \lifth{\dermat\en}-\int_{\boundaryGhcero} \partial_t(\alpha_{\partial}^*\conors)\cdot \dermat\en\\
    &+\norm{\dermat\dn}_{\Ldosgammas}\norm{\dermat\en}_{\Ldosgammas}+\normLdosgs{ \dn} \normLdosgs{\dermat\en}.
    \end{split}
    \end{equation*}

    We now focus on the boundary terms of the last bound.
    We first recall that $\partial_t\alpha_{\partial,h}=\dermat\nh\cdot\curvbdryh$, $\partial_t\conorh=\dermat\nh\times\btauh$, $\partial_t\alpha_{\partial}^*=\dermat\ns\cdot\curvbdryh$, $\partial_t\conors=\dermat\ns\times\btauh$. Then, adding and subtracting $\partial_t\alpha_{\partial}^*\conorh$ and $\alpha_{\partial,h}\conors$ we have
\begin{align*}
    \int_{\boundaryGhcero} \partial_t(\alpha_{\partial,h}\conorh)&\cdot \lifth{\dermat\en}-\int_{\boundaryGhcero} \partial_t(\alpha_{\partial}^*\conors)\cdot \dermat\en\\
    ={}&\int_{\boundaryGhcero} \big((\dermat\en\cdot \curvbdryh)\conorh + \partial_t\alpha_{\partial}^*(\en\times \btauh)\big)\cdot \dermat\en\\
    &+\int_{\boundaryGhcero} \big(\alpha_{\partial,h}(\dermat\en\times\btauh)+ (\en\cdot\curvbdryh)\partial_t\conors\big)\cdot \dermat\en\\
    \lesssim{}& \norm{\en}_{\Lp(\boundaryGhcero)}\norm{\dermat\en}_{\Lp(\boundaryGhcero)}+c_1\normcuad{\dermat\en}_{\Lp(\boundaryGhcero)}.
\end{align*}
Finally, using equivalence of norms, applying Young's inequality and absorbing the $\normLdosgs{\gradghs\dermat\en}$ term on the left-hand side we deduce that
\begin{equation}\label{eq: error equation for en parte II}
    \begin{aligned}
        \dfrac{1}{2}\dt&\normcuad{\dermat \lifth{\en}}_{\Ldosgammah} + \dfrac{1}{2}\normcuad{\gradgh\dermat \lifth{\en}}_{\Ldosgammah} \\
       &\qquad\lesssim \normLdosgh{ \gradgh \lifth{\en}}^2\normLdosgs{\gradghs\ex}+\normLdosgs{\gradghs\ev}\\
       &\qquad\quad +\normHunogs{\en}^2+\normLdosgs{\dermat\en}^2+\normcuad{\en}_{\Hunogammas} \\
       &\qquad\quad +
       \normcuad{\dermat\en}_{\Lp(\boundaryGhcero)}+c_1\normcuad{\dermat\en}_{\Lp(\boundaryGhcero)} \\
       &\qquad\quad +\normcuad{\dermat\dn}_{\Ldosgammas}+\normLdosgs{ \dn}^2+ \normLdosgs{\dermat\en}^2.
    \end{aligned}
\end{equation}

Finally, adding~\eqref{eq: error equation for en parte I} and~\eqref{eq: error equation for en parte II} we deduce
\begin{equation}\label{eq: error equation for en}
    \begin{aligned}
        \dt&\normcuad{  \en}_{\Hunogammas}+ \dfrac{1}{2}\dt\normcuad{\dermat \lifth{\en}}_{\Ldosgammah} + \dfrac{1}{2}\normcuad{\gradgh\dermat \lifth{\en}}_{\Ldosgammah}\\
        &\lesssim 
        -\dt\left(\intgammah \gradgh \lifth{\ns} \cdot \gradgh  \lifth{\en} 
        - \intgammas \gradghs \ns \cdot \gradgh \en\right) \\
        &\quad + \normcuad{\en}_{\Hunogammas} + \normcuad{\ex}_{\Hunogammas} + \normcuad{\ev}_{\Hunogammas} +\normcuad{\dermat\dn}_{\Ldosgammas}\\
        &\quad +\normcuad{\dermat\en}_{\Lp(\boundaryGhcero)}+\normLdosgs{\dermat\en}^2 + \normcuad{\dn}_{\Ldosgammas}.
    \end{aligned}
\end{equation}
Using trace inequality~\eqref{eq: trace ineq} with $0<\epsilon<1$, 
we absorb the term $\normcuad{\gradgh\dermat \lifth{\en}}_{\Ldosgammah}$ to the left-hand side and the estimate~\eqref{eq: energy estimate en} follows.
\end{proof}

The following proposition is the main result of this section and entails the global existence of the discrete solution and optimal error estimates when compared to the auxiliary discrete functions.

\begin{proposition}\label{prop: stability}
    Let $p\geq 2$, and assume that the exact solution is sufficiently smooth such that the defects satisfy
    \begin{equation}\label{eq: bound of defects}
        \begin{split}
            \norm{\dv\argu}\lesssim{}h^p,\quad
            \| \dH\argu\|_{\Ldosgammast}&\lesssim{} h^p,\;\;\;\\
            \| \dn\argu\|_{\Ldosgammast}+\| \dermat\dn\argu\|_{\Ldosgammast}&\lesssim{} h^p,  
        \end{split}
	\end{equation}
    for all $t\in [0,T]$ and $h>0$ sufficiently small, and the initial errors satisfy
\begin{equation}\label{eq: inicial error}
        \begin{split}
            \norm{\ex\argu[0]}_{\Hunogammast[0]}&
            +\norm{\eH(0)}_{\Hunogammas}
            \\
            &
            +\norm{\en(0)}_{\Hunogammas}
            +\norm{\dermat\en(0)}_{\Ldosgammas}
\lesssim{}h^p.
        \end{split}
    \end{equation}
Then, there exists $h_0>0$ such that for all $0<h\leq h_0$ the semidiscrete solutions exist for all $t\in [0,T]$ and the following error estimates are satisfied
	\begin{equation}\label{eq: stability bound}
        \begin{split}
        \norm{\ex(t)}_{\Hk(\Gammast))}&+\norm{\ev(t)}_{\Hk(\Gammast))}
        \\
        &+\norm{\eH(t)}_{\Hk(\Gammast))}+\norm{\en(t)}_{\Hk(\Gammast))}
		\lesssim{}h^p.
        \end{split}
	\end{equation}
\end{proposition}

\begin{proof}
We start the proof considering $t_0>0$ and $h_0> 0$ such that Lemmas~\ref{Lem: energy estimate ex ev},~\ref{Lem: energy estimate eH} and~\ref{Lem: energy estimate en} are satisfied. 
Since $t$ will be fixed on $[0,t_0]$, we will omit this argument $t$ to simplify the notation.

Replacing~\eqref{eq: energy estimate ev} in~\eqref{eq: energy estimate ex}, and integrating the resulting inequalities in time, we have
\begin{equation}\label{eq: int energy ex}
    \begin{aligned}
        &\normcuad{\ex\argu}_{\Hunogammast}
        \\
        &\quad\lesssim\int_{0}^{t}\left(\normcuad{ \en\argu[s]}_{\Hunogammast[s]} 
        + \normcuad{\eH\argu[s]}_{\Hunogammast[s]}
        + \normcuad{\dv\argu[s]}_{\Hunogammast[s]}
        \right)ds\\
        &\quad\quad+\int_{0}^{t}\normcuad{\ex\argu[s]}_{\Hunogammast[s]}ds+\normcuad{\ex\argu[0]}_{\Hunogammast[0]}.
    \end{aligned}
\end{equation}
We now want to replace~\eqref{eq: energy estimate ev} in~\eqref{eq: energy estimate eH}. First we observe that
Lemma~\ref{Lem: equivalence of discrete norms} and Young's inequality imply, for some constant $C>0$ and all $\epsilon > 0$,
\begin{align*}
  \int_{\Gammaht[t]} \gradght \lifth{\Hs}\argu[t]&\cdot \gradght  \lifth{\eH}\argu[t]-\int_{\Gammast} \gradghst \Hs\argu\cdot \gradghst  \eH\argu\\
  &\le \frac{C}{\epsilon} \normcuad{\ex\argu}_{\Hunogammast}+ \epsilon \normcuad{\eH\argu}_{\Hunogammast}.
\end{align*}
for all $t\in[0,t_0]$. 
This allows us to bound, after replacing~\eqref{eq: energy estimate ev} in~\eqref{eq: energy estimate eH}, and absorbing the term $\normcuad{\eH\argu}_{\Hunogammast}$ on the left-hand side, the curvature error as follows:
\begin{equation}\label{eq: int energy eH}
    \begin{aligned}
        &\normcuad{ \eH\argu}_{\Hunogammast}
        \lesssim  \normcuad{\ex\argu[t]}_{\Hunogammast}+\normcuad{\ex\argu[0]}_{\Hunogammast[0]}+\normcuad{\eH\argu[0]}_{\Hunogammast[0]}\\ 
        &\quad+\int_{0}^{t}\left( \normcuad{\eH\argu[s]}_{\Hunogammast[s]}+\normcuad{\en\argu[s]}_{\Hunogammast[s]}\right)ds\\
        &\quad+\int_{0}^{t}\left(\normcuad{\ex\argu[s]}_{\Hunogammast[s]}
        +\normcuad{\dH\argu[s]}_{\Ldosgammast[s]}
        +\normcuad{\dv\argu[s]}_{\Hunogammast[s]}
        \right)ds.
    \end{aligned}
\end{equation}
Analogously, from~\eqref{eq: energy estimate en} we have
\begin{equation}\label{eq: int energy en}
    \begin{aligned}
        &\hspace{-20pt}\normcuad{ \en\argu}_{\Hunogammast}+ \normcuad{\dermat \en\argu}_{\Ldosgammast} + \int_{0}^{t}\normcuad{\gradghs\dermat\en\argu[s]}_{\Ldosgammast[s]}ds\\
        \lesssim{}&\normcuad{\ex\argu[t]}_{\Hunogammast}+\normcuad{\eH\argu[t]}_{\Hunogammast}\\
        &+\normcuad{\ex\argu[0]}_{\Hunogammast[0]}+\normcuad{\en\argu[0]}_{\Hunogammast[0]} +\normcuad{\dermat \en\argu[0]}_{\Ldosgammast[0]}\\
        &+ \int_0^t\left(\normcuad{\en\argu[s]}_{\Hunogammast[s]} + \normcuad{\eH\argu[s]}_{\Hunogammast[s]} \normcuad{\ex\argu[s]}_{\Hunogammast[s]} \right)ds \\
        &+\int_0^t\Big(\normLdosgs{\dermat\en\argu[s]}^2+\normcuad{\dermat\dn\argu[s]}_{\Ldosgammast}\\
        &\qquad\quad + \normcuad{\dn\argu[s]}_{\Ldosgammast[s]}
        +\normcuad{\dv\argu[s]}_{\Hunogammast[s]}\Big)ds.
    \end{aligned}
\end{equation}
Adding~\eqref{eq: int energy ex},~\eqref{eq: int energy eH} and~\eqref{eq: int energy en}, we get
\begin{equation}
	\begin{split}
			\normcuad{\ex\argu[t]}_{\Hunogammast}
            &+	\normcuad{\eH\argu[t]}_{\Hunogammast}
            \\
            &\qquad + \normcuad{\en\argu[t]}_{\Hunogammast}
            +\normcuad{\dermat\en\argu[t]}_{\Ldosgammast}\\
		\lesssim{}&
        \normcuad{\ex\argu[0]}_{\Hunogammast[0]}
        +\normcuad{\eH\argu[0]}_{\Hunogammast[0]}
        \\
        &+\normcuad{\en\argu[0]}_{\Hunogammast[0]}
        +\normcuad{\dermat \en\argu[0]}_{\Ldosgammast[0]}
        \\
		&+\int_{0}^{t}\Big(
            \normcuad{\ex\argu[s]}_{\Hunogammast[s]}
            +
            \normcuad{\eH\argu[s]}_{\Hunogammast[s]} 
            \\ 
        &\qquad\quad +
        \normcuad{\en\argu[s]}_{\Hunogammast[s]} 
        +
        \normcuad{\dermat\en\argu[s]}_{\Ldosgammast}
        \Big)ds
        \\
		&+\int_0^t\Big(\normcuad{\dv\argu[s]}_{\Hunogammast[s]}+\normcuad{\dH\argu[s]}_{\Ldosgammast[s]}
        \\
        &\qquad\quad+\normcuad{\dermat\dn\argu[s]}_{H_h^{-1}(\Gammast[s])} + \normcuad{\dn\argu[s]}_{\Ldosgammast[s]}
       \Big)ds.
	\end{split}
\end{equation}
By Gronwall's inequality, using the assumptions~\eqref{eq: bound of defects}--\eqref{eq: inicial error}, we obtain, for all $0\le t \le t_0$ and $0<h\le h_0$,
\begin{equation*}
	\begin{split}
			\norm{\ex\argu[t]}_{\Hunogammast}
            +	\norm{\eH\argu[t]}_{\Hunogammast}+	\norm{\en\argu[t]}_{\Hunogammast}
            &+\norm{\dermat\en\argu[t]}_{\Ldosgammast}
            \\
            &\qquad\le C_1 e^{C_2t_0}h^p,
    \end{split}
\end{equation*}
for some positive constants $C_1$, $C_2$, independent of $h$ and $t$.
Thus, using inverse inequality, this last estimate, and once again~\eqref{eq: energy estimate ev}, we have for $t\in[0,t_0]$, and $0<h\le h_0$,
\begin{equation*}
	\begin{split}
			\normWunost{\ex\argu[t]}
            &+	\normWunost{\eH\argu[t]}
            \\
            &\qquad +	\normWunost{\en\argu[t]}
            +\normWunost{\ev\argu[t]}
            \\
            \le{}& \frac{C_3}h \Big(\norm{\ex\argu[t]}_{\Hunogammast}
            +	\norm{\eH\argu[t]}_{\Hunogammast}
            \\
            &\qquad+	\norm{\en\argu[t]}_{\Hunogammast}
            +\norm{\ev\argu[t]}_{\Hunogammast}\Big)
            \\
            \le{}& C_1\,C_3 e^{C_2t_0}h^{p-1}.
    \end{split}
\end{equation*}
We now redefine $h_0$ to satisfy $C_1\,C_3 e^{C_2T}h_0^{p-1} \le h_0^{(p-1)/2}$, whence~\eqref{eq: bounds for numerical approx} holds for $t_0=T$, thereby implying the bound~\eqref{eq: stability bound} for $t\in[0,T]$. The assertion of the Proposition thus follows.
\end{proof}

\subsection{Convergence}
\label{sec: convergence}

In this section we combine the previous error estimates to finally deduce the error bounds for the solution of the semidiscrete problem stated as Problem~\ref{weak discrete problem MCF}.

\begin{theorem}\label{teo: convergencia}
    Assume that the initial data $\X^0$ is such that Problem~\ref{weak problem MCF} admits an exact solution $(\X\argu, \vel\argu, \Hm\argu, \n\argu)$ sufficiently smooth, with $\Gammat=\X\argu(\Omega)$ a non-degenerate and regular surface with a piecewise smooth boundary $\partial\Gammat=\partial\Gamma_0$ for all $t\in[0,T]$. 
    Consider the semi-discrete Problem~\ref{weak discrete problem MCF}, with the spline spaces $\Sparam^3$, $\Shcero$, $\Shcero^3$ and $\Otauh$ as defined in Sections~\ref{sec: surface approximation} and~\ref{sec: spline approximation estimates}, with polynomial degree $p\ge 2$ and smoothness $C^\ell$, $0\le \ell \le p-1$.
    Then, there exists $h_0>0$ such that for all $0<h\leq h_0$, Problem~\ref{weak discrete problem MCF} admits a unique solution  
     $(\Xh\argu, \vh\argu, \Hh\argu, \nh\argu)$ on $[0,T]$. Moreover
    we have the following estimate for the spline parametrization
    \begin{equation*}
    \norm{\X\argu-\Xh\argu}_{\Hk(\Omega)^3}\lesssim h^p,
    \end{equation*}
    and also, 
    \begin{align*}
	\norm{\liftgamma{\idh}-\idg}_{\Hunogammat^3}&\lesssim h^p,\quad \norm{\liftgamma{\vh}-\vel}_{\Hunogammat^3}\lesssim h^p,\\
	\norm{\liftgamma{\nh}-\n}_{\Hunogammat^3}&\lesssim h^p,\quad \norm{\liftgamma{\Hh}-\Hm}_{\Hunogammat^3}\lesssim h^p,
	\end{align*}
 where the hidden constants depend on the regularity of the solution, but are independent of $h$.
\end{theorem}
\begin{proof}
For any variable $u\in\{\X,\idg,\vel,\Hm,\n\}$ we have
\begin{align*}
    \norm{\liftgamma{u_h}-u}_{\Hunogamma}
    &\leq 
    \norm{\liftgamma{u_h}-\liftgamma{u^*}}_{\Hunogamma}+\norm{\liftgamma{u^*}-u}_{\Hunogamma}.
    \end{align*}
Where $u^*\in\{\Xs,\ids,\vs,\Hs,\ns\}$ y $u_h\in\{\Xh,\idh,\vh,\Hh,\nh\}$.

In each case, the second term is bounded by \( ch^p \) due to the interpolation estimates in Lemma~\ref{Lem: global error estimate quasi int on gamma} and in Propositions~\ref{prop: error estimate dermat linear Ru} and~\ref{prop: error estimate non-linear Ru}. 

Now we will see that the assumptions~\eqref{eq: bound of defects} and~\eqref{eq: inicial error} of Proposition~\ref{prop: stability} are satisfied, which will thereby imply that the first terms are also bounded by $c\,h^p$, arriving at the assertion of the theorem.

By Proposition~\ref{prop: consistency estimates}, there exists $h_0>0$ such that assumption~\eqref{eq: bound of defects} holds for all $t\in[0,T]$ and all $0<h\le h_0$. 
So it only remains to check the initial errors satisfy the assumption~\eqref{eq: inicial error}. 

Due to interpolation estimates, the initial error $\ex\argu[0]$ for the position satisfies $\norm{\ex\argu[0]}_{\Hunogammast[0]}\lesssim h^p$.
For $\eH$, in $t=0$ we have
\begin{align*}
\norm{\eH\argu[0]}&_{\Hk(\Gammast[0])}=\norm{\lifts{\Hh\argu[0]}-\Hs\argu[0]}_{\Hk(\Gammast[0])}\\
&\lesssim\norm{\Qs[{\Gammaht[0]}]\Hmcero-\lifth{\Hmcero}}_{\Hk(\Gammaht[0])}+\norm{\lifts{\Hmcero}-\Rs[{\Gammast[0]}]\Hmcero}_{\Hk(\Gammast[0])}\\
&\lesssim h^p,
\end{align*}
thanks again to the interpolation error and the error in the Ritz projection.
Since we define $\n_{h,0}=\ns_{0}$, we know that $\en\argu[0]=0$,
whence it only remains to show that $\norm{\dermat \en(0)}\lesssim h^p$.
To do this we will again consider a decomposition of the initial error
\begin{align*}
\norm{\dermat\en\argu[0]}_{\Lp(\Gammast[0])}&=\norm{\lifts{\dermat\nh\argu[0]}-\dermat\ns\argu[0]}_{\Lp(\Gammast[0])}.
\end{align*} 
Since we do not have initial data for $\dermat\nh(t)$, we must obtain an estimate for this error by considering the definition of $\dermat\nh(0)$. This is defined using equation~\eqref{eq: weakh-normal} at time $t=0$, which reads
\begin{equation*}
\begin{aligned}
\int_{\Gammaht[0]} \dermat   \nh\argu[0] \cdot \testnuh =&- \int_{\Gammaht[0]}\nabla_{\Gammaht[0]}\nhcero \cdot \nabla_{\Gammaht[0]} \testnuh \\
&+ \ \int_{\Gammaht[0]} |A_{h,0}|^2 \nhcero \cdot \testnuh+\int_{\boundaryGhcero} \alpha_{\partial,h,0}\conor_{h,0}\cdot \testnuh,
\end{aligned}
\end{equation*}
for all $\testnuh\in\Otauh$. Then, using again that $\nhcero=\ns(0)=\Rs[{\Gammast[0]}]\n(0)$, with $\Rs$ the nonlinear Ritz projection from Definition~\ref{def: non linear Ru trace in Otauh}, we obtain 
\begin{equation}
    \begin{aligned}
    \int_{\Gammaht[0]} \dermat   \nh\argu[0] \cdot \testnuh =&- \int_{\Gammat[0]}\nabla_{\Gammat[0]}\ncero \cdot \nabla_{\Gammat[0]} (\testnuh)^{\Gammat[0]} \\
    &+ \ \int_{\boundaryG} \alpha_{\partial,0}\conor_{0}\cdot (\testnuh)^{\Gammat[0]}+ \ \int_{\Gammaht[0]} |A_{h,0}|^2 \n_{h,0} \cdot \testnuh\\
    &+\lambda\left(\int_{\Gammaht[0]}\ncero \cdot  \testnuh-\int_{\Gammat[0]}\ncero\cdot (\testnuh)^{\Gammat[0]}\right). \label{eq: weakh-normal-cero}
    \end{aligned}
    \end{equation}
Since $\n$ is solution of the continuous problem $\dermat \n(0)=\Delta_{\Gammat[0]}\ncero+|A_{0}|^2\ncero$, we know that for any $\testnu\in \Hk(\Gammat[0])$ we have
\begin{equation*}
    \begin{split}
        \int_{\Gammat[0]}\dermat \n(0)\cdot\testnu =&-\int_{\Gammat[0]}\nabla_{\Gammat[0]}\ncero : \nabla_{\Gammat[0]} \testnu+ \int_{\boundaryG}\alpha_{\partial,0}\conor_{0} \cdot \testnu\\
        &+ \int_{\boundaryG}(\partial_{\conor_{0}}\ncero \cdot \btau)\btau\cdot\testnu+\int_{\Gammat[0]} |A_{0}|^2 \ncero \cdot \testnu.
    \end{split}
\end{equation*}
Then, we rewrite~\eqref{eq: weakh-normal-cero} as
\begin{equation*}
    \begin{aligned}
    \int_{\Gammaht[0]} \dermat   \nh\argu[0] \cdot \testnuh =&\int_{\Gammat[0]}\dermat \n(0)\cdot(\testnuh)^{\Gammat[0]} -\int_{\boundaryG}(\nabla_{\Gammat[0]}\ncero\conor_{0} \cdot \btau)\btau\cdot(\testnuh)^{\Gammat[0]}\\
    &+  \ \int_{\Gammaht[0]} |A_{h,0}|^2 \n_{h,0} \cdot \testnuh-\int_{\Gammat[0]} |A_{0}|^2 \ncero \cdot (\testnuh)^{\Gammat[0]}\\
    &+\lambda\left(\int_{\Gammaht[0]}\ncero \cdot  \testnuh-\int_{\Gammat[0]}\ncero\cdot (\testnuh)^{\Gammat[0]}\right). 
    \end{aligned}
    \end{equation*}
Adding and subtracting $(\dermat\ns(0))^{\Gammaht[0]}=(\dermat\Rs[{\Gammast[0]}]\n(0))^{\Gammaht[0]}$ in the integral of the left-hand side, we have
\begin{equation*}
\begin{aligned}
 \int_{\Gammaht[0]}&\left(\dermat   \nh\argu[0]-(\dermat\ns(0))^{\Gammaht[0]}\right)  \cdot \testnuh \\
 =&\int_{\Gammat[0]}\dermat \n(0)\cdot(\testnuh)^{\Gammat[0]} - \int_{\Gammaht[0]}\lifth{\dermat\Rs[{\Gammast[0]}]\n(0)}  \cdot \testnuh\\
 &+  \ \int_{\Gammaht[0]} |A_{h,0}|^2 \n_{h,0} \cdot \testnuh-\int_{\Gammat[0]} |A_{0}|^2 \ncero \cdot (\testnuh)^{\Gammat[0]}\\
 &+\lambda\left(\int_{\Gammaht[0]}\ncero \cdot  \testnuh-\int_{\Gammat[0]}\ncero\cdot (\testnuh)^{\Gammat[0]}\right)\\
 &-\int_{\boundaryG}(\partial_{\conor_{0}}\ncero \cdot \btau)\btau\cdot{(\testnuh)}^{\Gammat[0]}.
\end{aligned}
\end{equation*}
The first three differences in the right-hand side can be estimated as in Section~\ref{sec: consistency}, and are all bounded by $ch^p\norm{\testnuh}_{\Lp(\Gammaht[0])}$. For the last term we can also use the same technique we have used in~\ref{sec: consistency}, more precisely, in the proof of Proposition~\ref{prop: consistency estimates}. Finally, taking $\testnuh=\dermat \nh\argu[0]-(\dermat\ns(0))^{\Gammaht[0]}$, we get that $\dermat\en\argu[0]$ satisfies the desired estimate~\eqref{eq: inicial error}.

We have thus proved that the assumptions of Proposition~\ref{prop: stability} hold, and the assertion of the theorem follows.
\end{proof}

\section{Numerical Experiments}\label{sec: num exp}

\subsection{Matrix vector form}
We present here the matrix-vector formulation of scheme~\eqref{weak discrete problem MCF}. We collect the coefficients of the solutions $\Xh\argu\in \Sparam^3$, $\vh\argu\in \Shcero^3$, $\Hh\argu\in \Shcero$ and $\nh \argu\in\Otauh$, respectively, in the following column vectors
\begin{align*}
    \bs x \in \mathbb{R}^{3N},\quad \vel  \in \mathbb{R}^{3N},
    \quad \bs{\kappa} \in \mathbb{R}^{N},\quad \n  \in \mathbb{R}^{3N}.
\end{align*}
Let $I_{3}$ the $3\times 3$ identity matrix, we denote by 
$\bs M^{\Gammah}=I_{3}\tensor M^{\Gammah}$ ($\bs M^{\Gammah}_0=I_{3}\tensor M^{\Gammah}_0$) and 
$\bs A^{\Gammah}=I_{3}\tensor A^{\Gammah}$ ($\bs A^{\Gammah}_0=I_{3}\tensor A^{\Gammah}_0$),
where 
$M^{\Gammah}$ ($M^{\Gammah}_0$) 
and 
$A^{\Gammah}$ ($A^{\Gammah}_0$)
are the usual mass and stiffness matrices with basis functions in 
$\Sh$ ($\Shcero$).
More precisely, the surface-dependent mass matrix $M^{\Gammah}$ ($M^{\Gammah}_0$) and stiffness matrix $A^{\Gammah}$ ($A^{\Gammah}_0$) on the surface ${\Gammah}$, defined by the spline function $\Xh$ with coefficients $\bs x$, are given by
\begin{align*}
    M^{\Gammah}_{ij}=\int_{\Gammah} \basisfg[j]\basisfg[i],&\qquad A^{\Gammah}_{ij}=\int_{\Gammah} \gradgh\basisfg[j]\cdot \gradgh\basisfg[i],& i,j=1,2,\dots,N\ (N_0),
\end{align*}
where $\{\basisfg\}_{j=1}^{N \ (N_0)}$ is a basis of $\Sh$ ($\Shcero$).

We define the following nonlinear functions $f_1(\bs x,\bs \Hm,\n)\in \mathbb{R}^{N}$, $f_2(\bs x,\n)\in \mathbb{R}^{3N}$ and $f_b(\bs x,\n)\in \mathbb{R}^{3N}$ by
\begin{align*}
    f_1(\bs x,\bs \Hm,\n)_{i} &= \int_{\Gammah} |A_h|^2 \,\Hh \,\basisfg[i],  \\
    f_2(\bs x,\n)_{k} &=\int_{\Gammah} |A_h|^2 \,\nh \cdot {\bs b}_\partial^k\\
    f_b(\bs x,\n)_k &=\int_{\boundaryGhcero} \alpha_{\partial,h}\,\conorh \cdot {\bs b}_\partial^k,
\end{align*}
for $i=1,\dots, N$ and $k=1,\dots,N_\partial$, where $\alpha_{\partial,h}=\nh\cdot\curvbdryh$ and $ (\conorh)_{\ell}=(\nh)_{\ell}\times \btauh$, with $\Hh$ and $\nh$ the spline functions corresponding to the vectors $\bs \Hm$, $\n$, respectively. Also, $\Gammah$ is the image of the reference domain $\Omega$ by $\Xh(t)$, the spline function corresponding to ${\bs x}(t)$.

Then,~\eqref{weak discrete problem MCF} can be written in the following matrix-vector form:

\begin{problem} 
    Find ${\bs x}:[0,T]\to \R^{3N}$, ${\bs \kappa}:[0,T]\to \R^{N_0}$, ${\n}:[0,T]\to \R^{N_\partial}$ such that, 
    \begin{align}
    M^{\Gammah}_0 \dot{\bs{\Hm}}+ A^{\Gammah}_0\bs\Hm ={}& f_1(\bs x,\bs \Hm,\n) \label{eq: mat-vec kappa} \\
    \bs M^{\Gammah}_\partial \dot{\n}+\bs A^{\Gammah}_\partial\n ={}& f_2(\bs x,\n)+f_b(\bs x,\n)  \label{eq: mac-vec nu}  \\
            \dot{\bs x} ={}& \vel \label{eq: mac-vec vel}
    \end{align}
    where $\vel$ is given by the coefficients of the spline function $\Qs[\Gammah](-\Hh\nh)$,
    and $\Gammah(t)$ is the image of the reference domain $\Omega$ by $\Xh(t)$, the spline function corresponding to ${\bs x}(t)$.
    Here, we let $\bs M^{\Gammah}_\partial$ and 
$\bs A^{\Gammah}_\partial$  denote the usual mass and stiffness matrices obtained when using a set
$\{ {\bs b}_\partial^i \}_{i=1}^{N_\partial}$ of basis functions of $\Otauh$. Obtaining such a basis is very complicated, and instead of doing this, we resort to a Lagrange multiplier formulation which we describe below.
    \end{problem}

The equation for $\bs\Hm$ has the same structure as the analogous one in~\cite{KLL2019}, the only difference being that we use discrete spaces with zero boundary values.
The main changes are in the equation corresponding to the evolution of the normal vector $\nh$. 
There are at least two ways to implement the spaces $\Otauh$: using constraints or constructing a local basis of  $\Otauh$. 
We have decided to use constraints, which is simpler and more realistic than building a set of local basis functions of $\Otauh$.
In order to do this, we use techniques similar to those that are normally used for various saddle point problems in partial differential equations.

The evolution equation for $\nh$ consists in finding $\nh:[0,T]\to \Otauh$ such that, for all $\testnuh\in\Otauh$,
\begin{equation*}
    \int_{\Gammah} \dermat   \nh \cdot \testnuh + \int_{\Gammah}\gradgh\nh \cdot \gradgh \testnuh = \ \int_{\Gammah} |A_h|^2 \nh \cdot \testnuh+\int_{\boundaryGhcero} \alpha_{\partial,h}\conorh\cdot \testnuh.
\end{equation*}
Recalling from~\eqref{def:Otauh} the definition of $\Otauh$:
\begin{equation*}
    \Otauh:= \left\{ \bs\phi_h \in \Sh: \int_{\partial\Gammah^0}(\bs\phi_h\cdot \btauhunit)w_h=0, \,\text{for all}\ w_h\in \mathcal{S}_{h}^\partial  \right\},
\end{equation*}
this equation is equivalent to finding $\nh:[0,T]\to \Sh^3$ and a Lagrange multiplier $\lambda_h:[0,T]\to\Sh$ such that
\begin{align*}
    \int_{\Gammah} \dermat   \nh \cdot \bs z_h + \int_{\Gammah}\gradgh\nh \cdot \gradgh \bs z_h + \int_{\boundaryGhcero} (\btauhunit\cdot \bs z_h)\lambda_h&=
     \ \int_{\Gammah} |A_h|^2 \nh \cdot \testnuh&\\
      & +\int_{\boundaryGhcero} \alpha_{\partial,h}\conorh\cdot \testnuh,&\\
    \int_{\boundaryGhcero} (\btauhunit\cdot \nh)w_h&=0&
\end{align*}
for all $\bs z_h\in \Sh^3$ and all $w_h\in\Sh$. 

The equivalent matrix-vector formulation of~\eqref{eq: mac-vec nu} will be the following, for $\n : [0,T]\to \R^{3N}$:
\begin{align*}
    \bs M^{\Gammah} \dot{\n}+\bs A^{\Gammah}\n + \bs S^T \bs\lambda=&\, f_2(\bs x,\n)+f_b(\bs x,\n)\\
    \bs S \n  =&\,\bs 0,
\end{align*}
where the $3N\times3N$ matrix $\bs S$ is given by
\begin{equation*}
    \bs S = 
    \begin{bmatrix}
        S^1 &  0_{N} &  0_{N}\\
         0_{N} & S^2 &  0_{N}\\
         0_{N} &  0_{N} & S^3
        \end{bmatrix} ,
\end{equation*}
with $0_{N}$ the $N\times N$-zero matrix and $S^k_{ij}=\int_{\boundaryGhcero} \basisfg[i]\basisfg[j](\bs e_k\cdot \btauhunit)=\int_{\boundaryGhcero} \basisfg[i]\basisfg[j](\btauhunit)_k$, for $i,j=1,\dots,N$ and $k=1,2,3$, with $\bs
 e_k$ are the standard unit vectors of $\mathbb{R}^3$.

The final matrix-vector form of the problem now reads

 \begin{problem}\label{matrix-vec discrete problem MCF}
    Find ${\bs x}:[0,T]\to \R^{3N}$, ${\bs \kappa}:[0,T]\to \R^{N_0}$, ${\n}:[0,T]\to \R^{3N}$, ${\bs\lambda}:[0,T]\to \R^{N}$ such that, 
    \begin{align*}
    M^{\Gammah}_0 \dot{\bs{\Hm}}+ A^{\Gammah}_0\bs\Hm ={}& f_1(\bs x,\bs \Hm,\n)  \\
    \bs M^{\Gammah} \dot{\n}+\bs A^{\Gammah}\n + \bs S^T \bs\lambda={}& f_2(\bs x,\n)+f_b(\bs x,\n)\\
    \bs S \n  ={}&\bs 0,\\
                    \dot{\bs x} ={}& \vel
    \end{align*}
    where $\vel$ is given by the coefficients of the spline function $\Qs[\Gammah](-\Hh\nh)$,
    and $\Gammah(t)$ is the image of the reference domain $\Omega$ by $\Xh(t)$, the spline function corresponding to ${\bs x}(t)$.
    
    \end{problem}

\begin{remark}It is important to clarify that the weak formulation with Lagrange multiplier is used only as a tool for the implementation of the numerical scheme, and is the one we have used in the numerical experiments reported below.
We believe that it is important to mention this Lagrange multiplier formulation, in order to show that the scheme~\eqref{weak discrete problem MCF} can be implemented, and it is not just theoretical.
\end{remark}

\subsection{Numerical results}
\graphicspath{{imagenes}}
The experiments in this section are implemented using the Octave software, particularly employing the GeoPDEs package~\cite{VAZQUEZ2016}. Two numerical experiments are presented, both using biquadratic spline spaces globally $C^1$. We have used a mesh with $400$ elements and $484$ DOFs for both experiments.
For the time discretization of the problem~\eqref{weak discrete problem MCF} we consider the linearly implicit $q$-step Backward Differentiation Formulas of order $q=2$.
An analysis of the fully discretized method falls beyond the scope of this article, and will be subject of future work.

\paragraph{Example 1} 
In Figure~\ref{fig: superficie incial plano perturbado}, we plot the initial surface with its mean curvature, and we show the evolution of the area of approximated surface during the time interval $[0,0.8]$. The area of $\Gamma_{h,0}$ was $4.0442$ and at time $0.8$ was $4.0000036$. 
\begin{figure}[!ht]
  \centering
  \begin{subfigure}[b]{0.49\linewidth}
    \includegraphics[width=\linewidth]{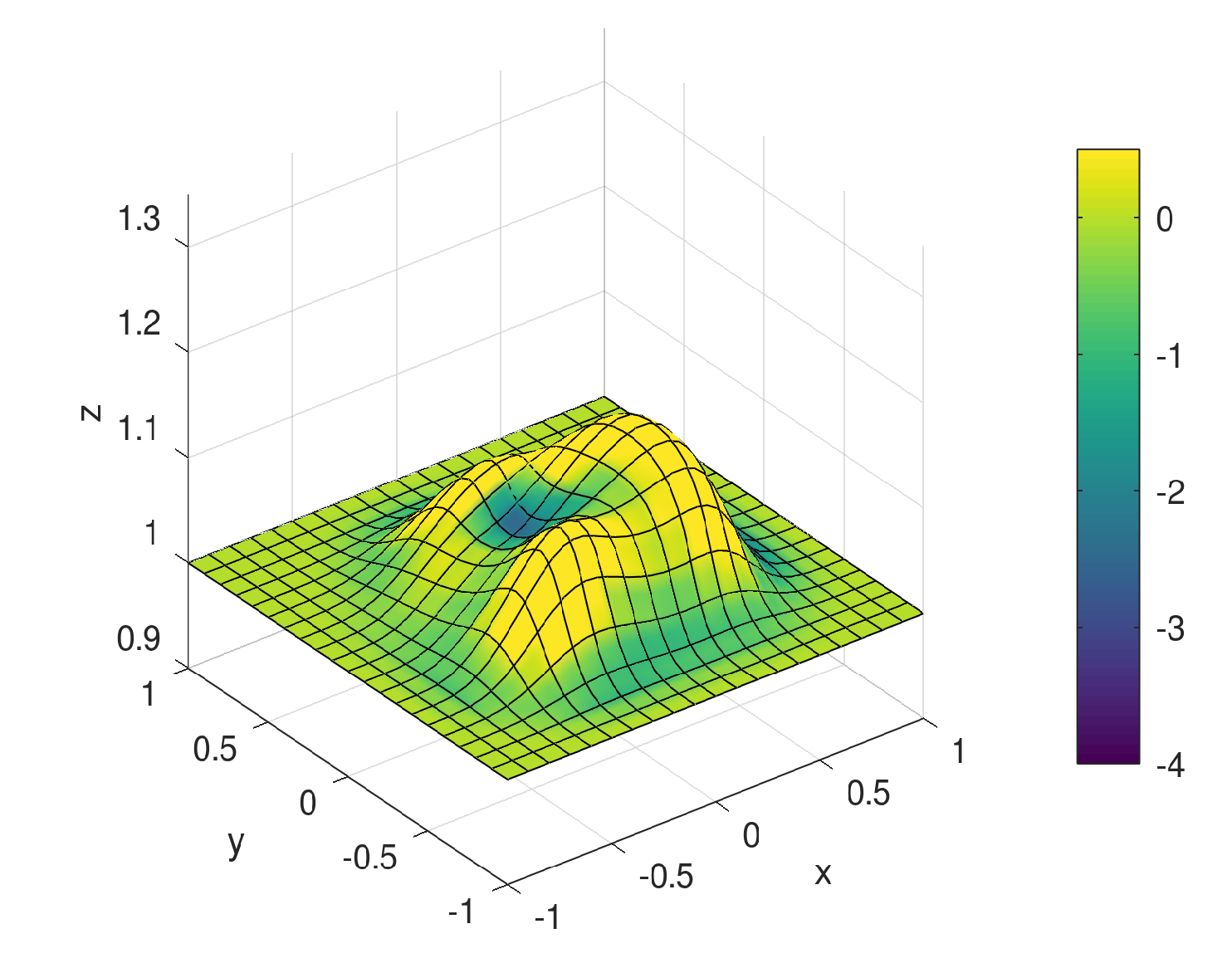}
    \label{fig:westminster_lateral}
  \end{subfigure}\hspace{5mm}%
  \begin{subfigure}[b]{0.4\linewidth}
    \includegraphics[width=\linewidth]{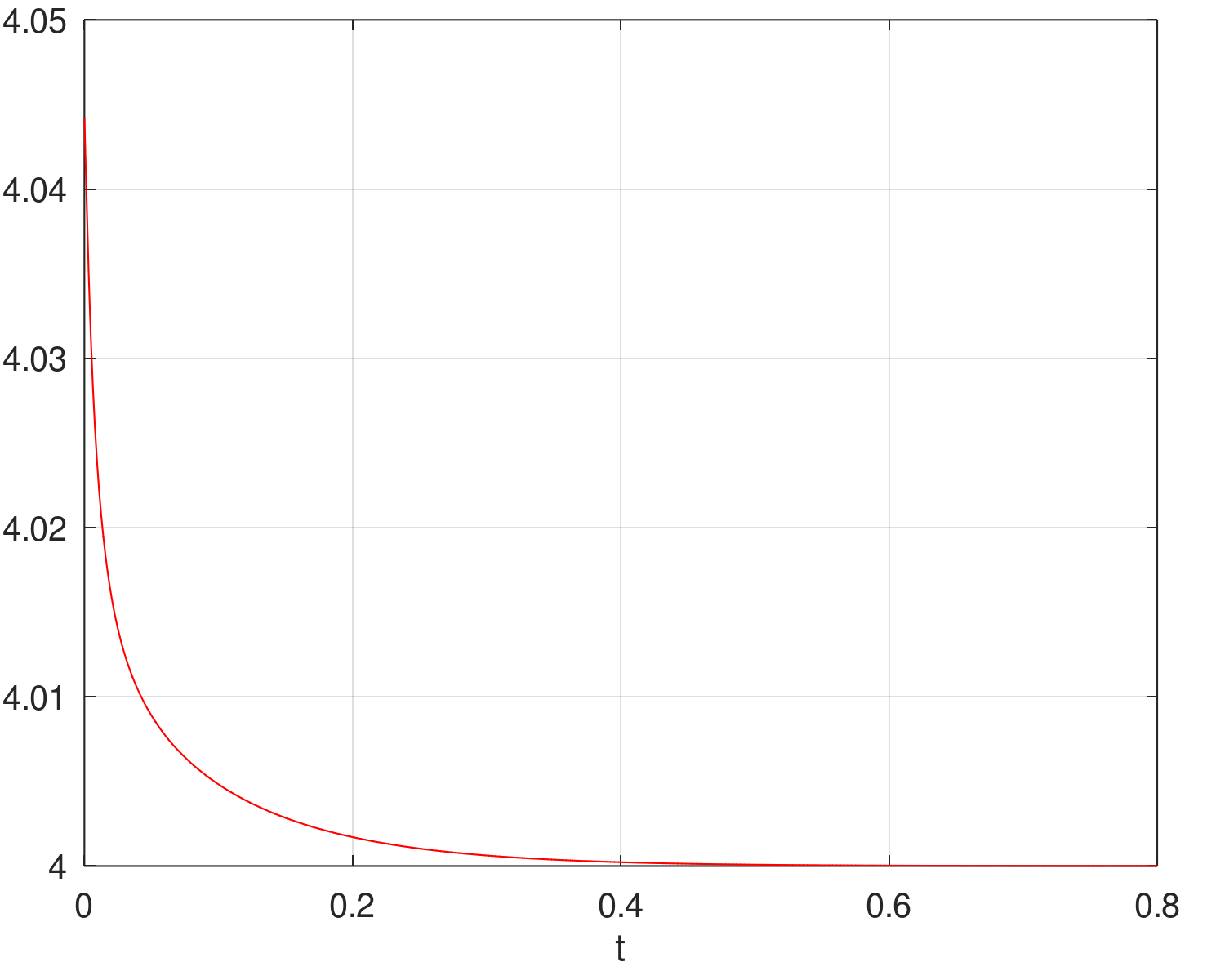} 
    \label{fig:westminster_lateral_2}
  \end{subfigure}
  \caption{Initial surface (left), with the mean curvature represented by its colors. Evolution of approximate area $|\Gammaht|$ against time during the time interval $[0,0.8]$ (right). The area of $\Gamma_{h,0}$ was $4.0442$ and at time $0.8$ was $4.0000036$.}
  \label{fig: superficie incial plano perturbado}
\end{figure}

In Figure~\ref{fig: superficies aproximadas plano perturbado} we show the evolution of the surface from Example 1 from two angles. The time step size was $\deltat= 0.0015625$. 



\begin{figure}[!ht]
  \centering
  \begin{subfigure}[b]{0.49\linewidth}
    \includegraphics[width=\linewidth]{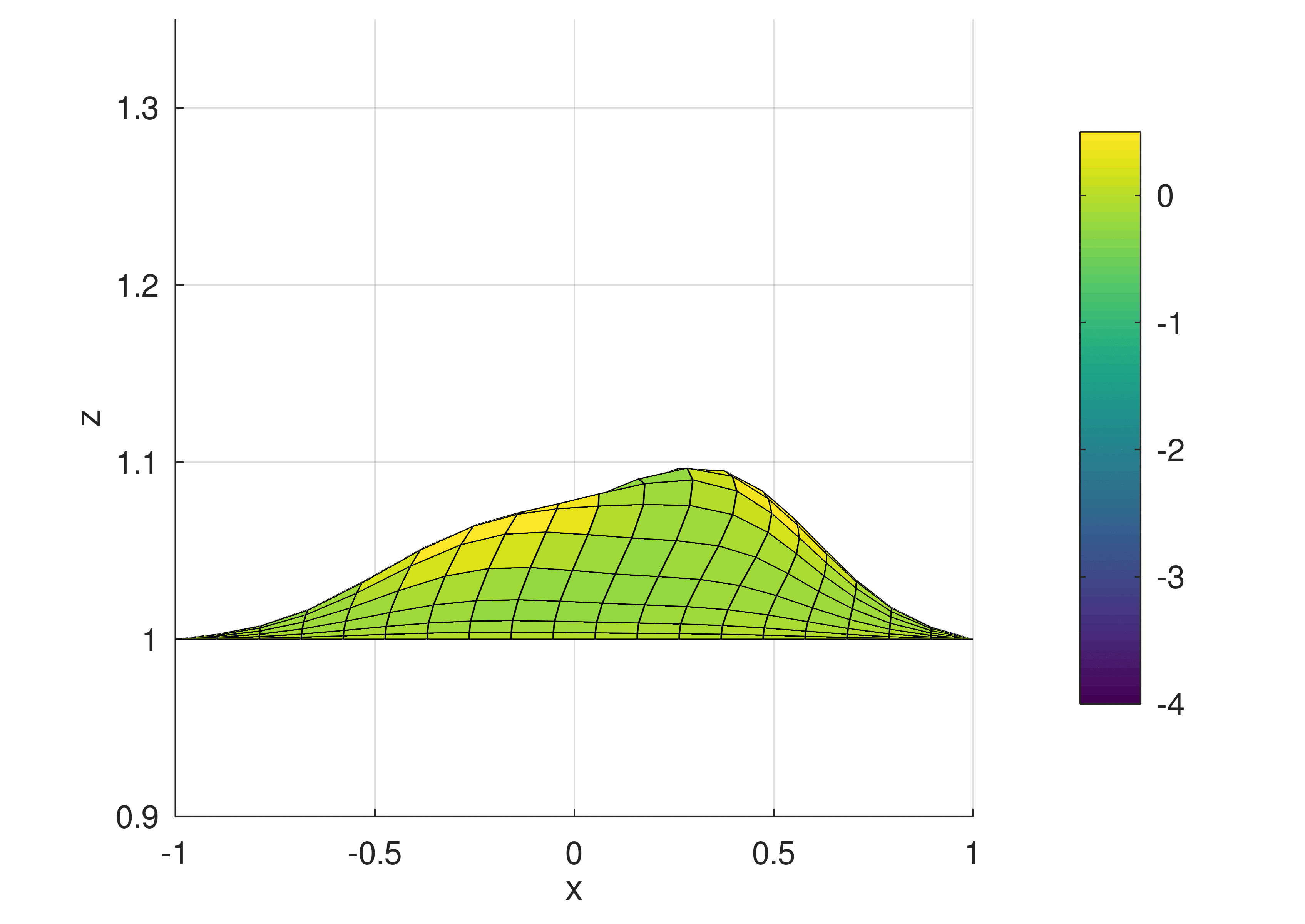}
  \end{subfigure}
  \begin{subfigure}[b]{0.49\linewidth}
    \includegraphics[width=\linewidth]{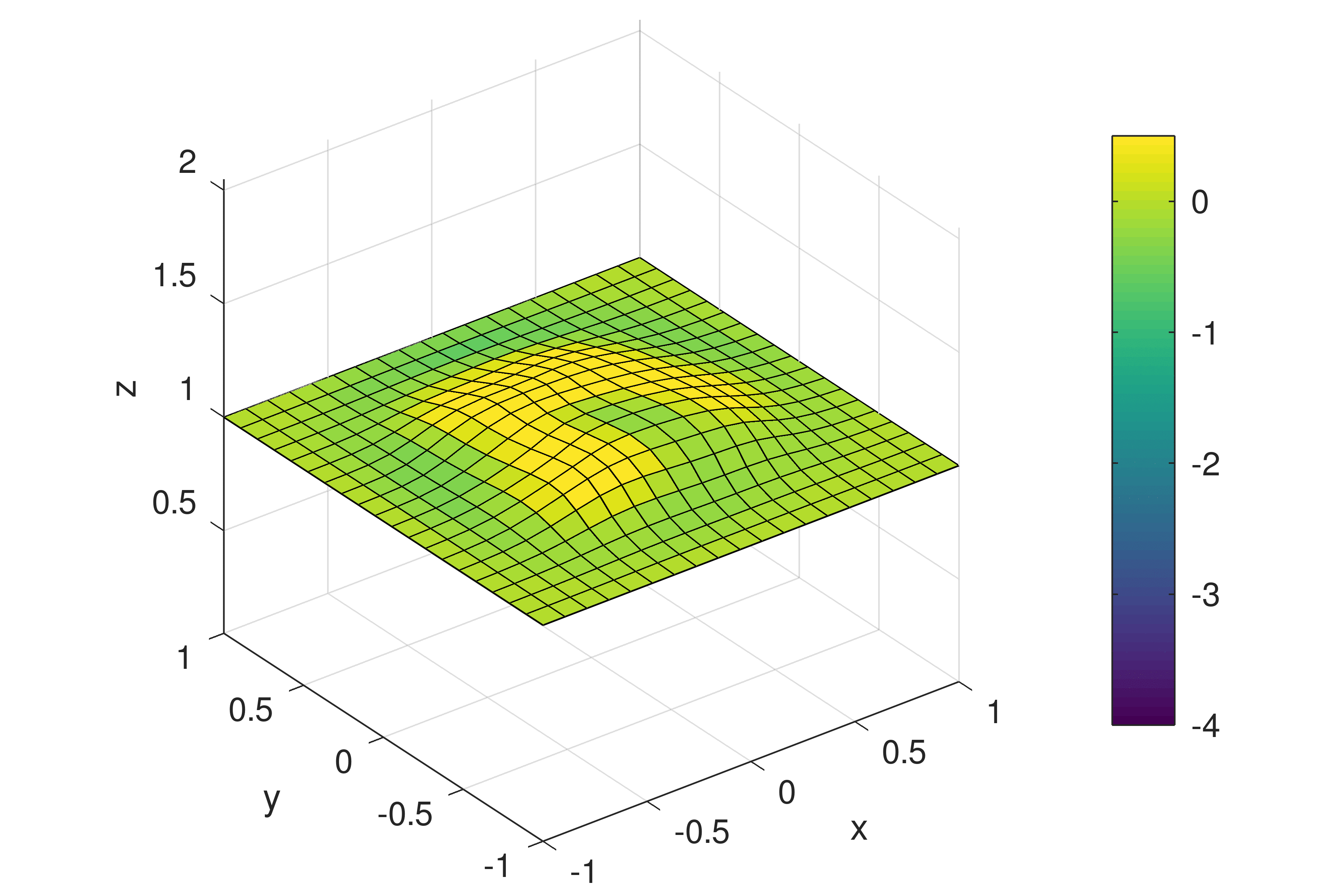}
  \end{subfigure}
  \vspace{0.5cm} 
  \begin{subfigure}[b]{0.49\linewidth}
    \includegraphics[width=\linewidth]{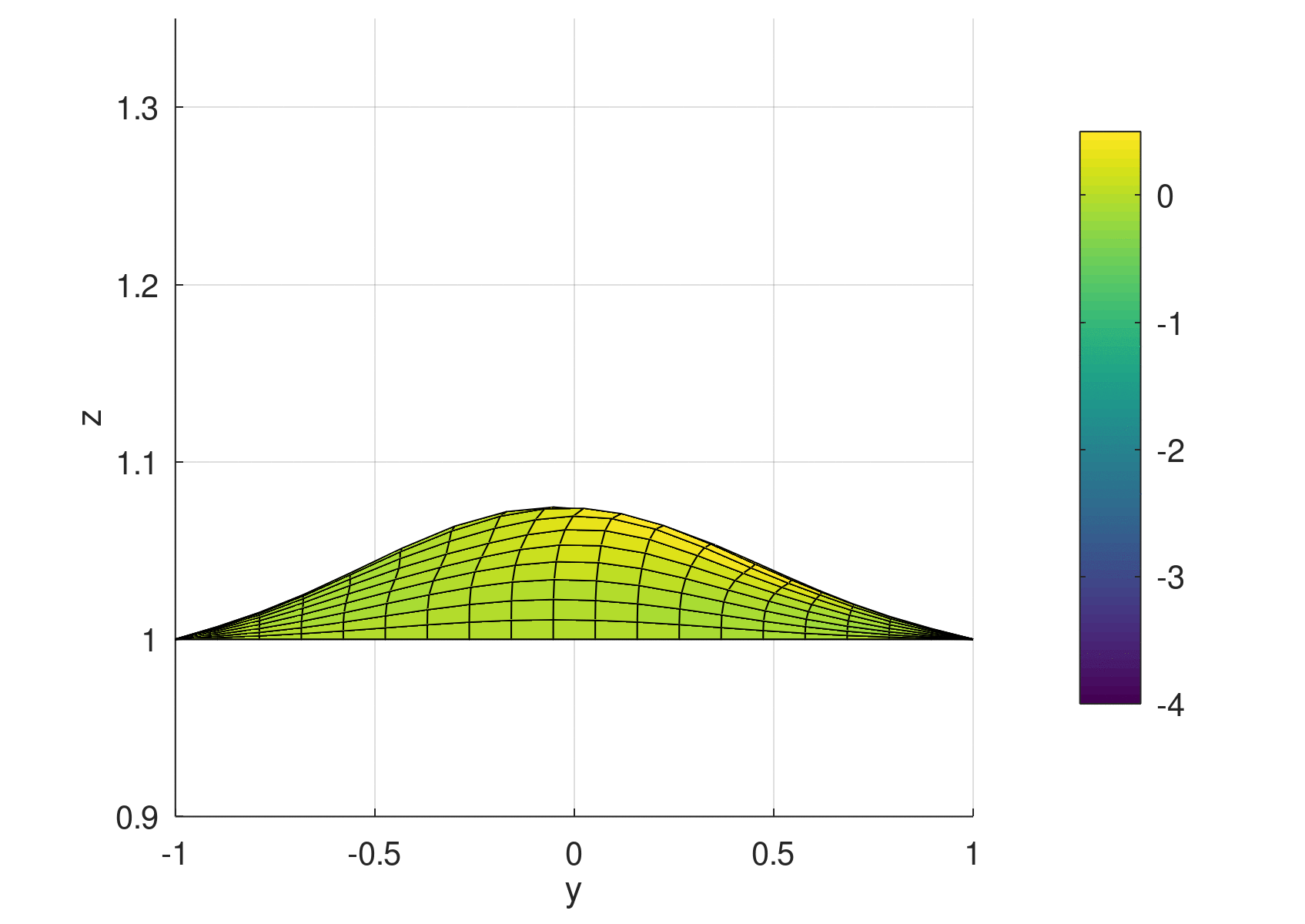} 
  \end{subfigure}
  \begin{subfigure}[b]{0.49\linewidth}
    \includegraphics[width=\linewidth]{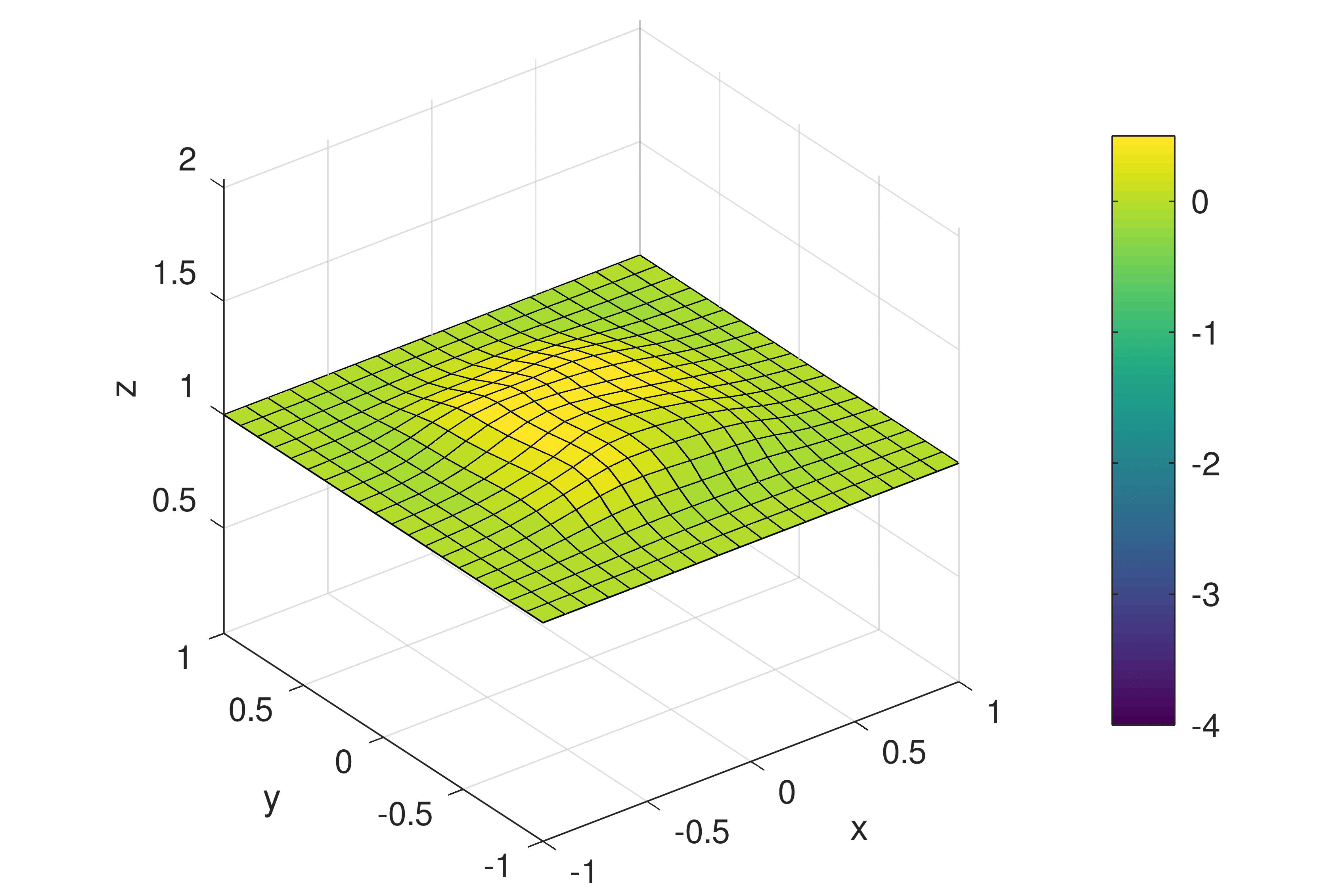} 
  \end{subfigure}
  \vspace{0.5cm}
  \begin{subfigure}[b]{0.49\linewidth}
    \includegraphics[width=\linewidth]{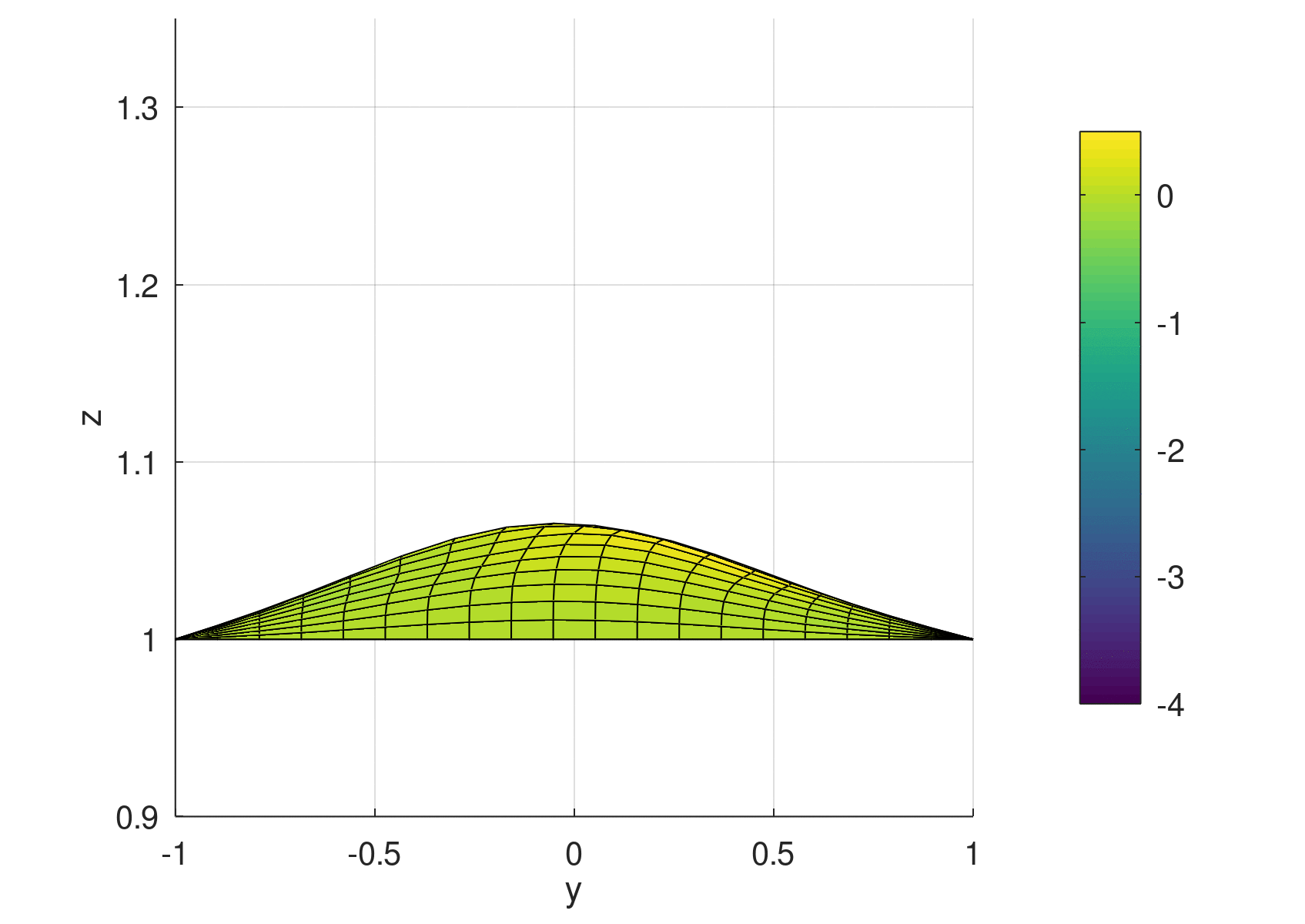}
  \end{subfigure}
  \begin{subfigure}[b]{0.49\linewidth}
    \includegraphics[width=\linewidth]{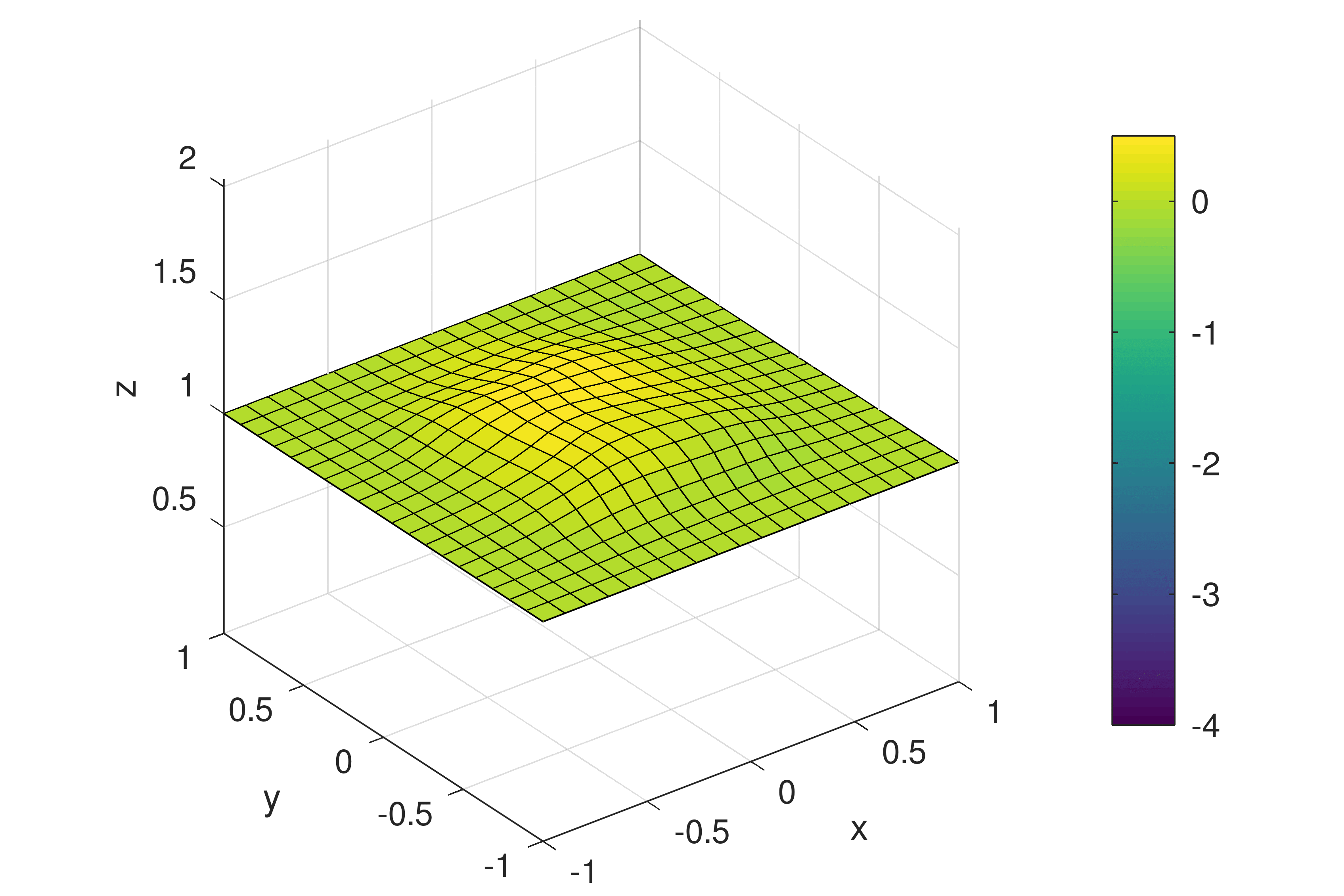}
  \end{subfigure}
  \caption{Evolution of discrete surface from Example 1. Solution at time $t=0.015$ (top), $t=0.03$ (middle), $t=0.05$ (bottom). The colors indicate the value of the discrete variable $\Hh$. It is worth noticing that even though the solution gets flatter with time, the mesh keeps some distortion.}
  \label{fig: superficies aproximadas plano perturbado}
\end{figure}

\paragraph{Example 2} Now we consider a surface whose boundary is not planar. It is a square patch of a sphere, with an initial area of $5.859$. In Figure~\ref{fig: superficie inicial esfera}, we plot $\Gamma_{h,0}$ along with its initial mean curvature. On the right, we show the evolution of the area of approximated surface during the time interval $[0,0.9]$. The area at time $0.9$ was $4.354$.

\begin{figure}[!ht]
    \centering
    \begin{subfigure}[b]{0.49\linewidth}
    \includegraphics[width=\linewidth]{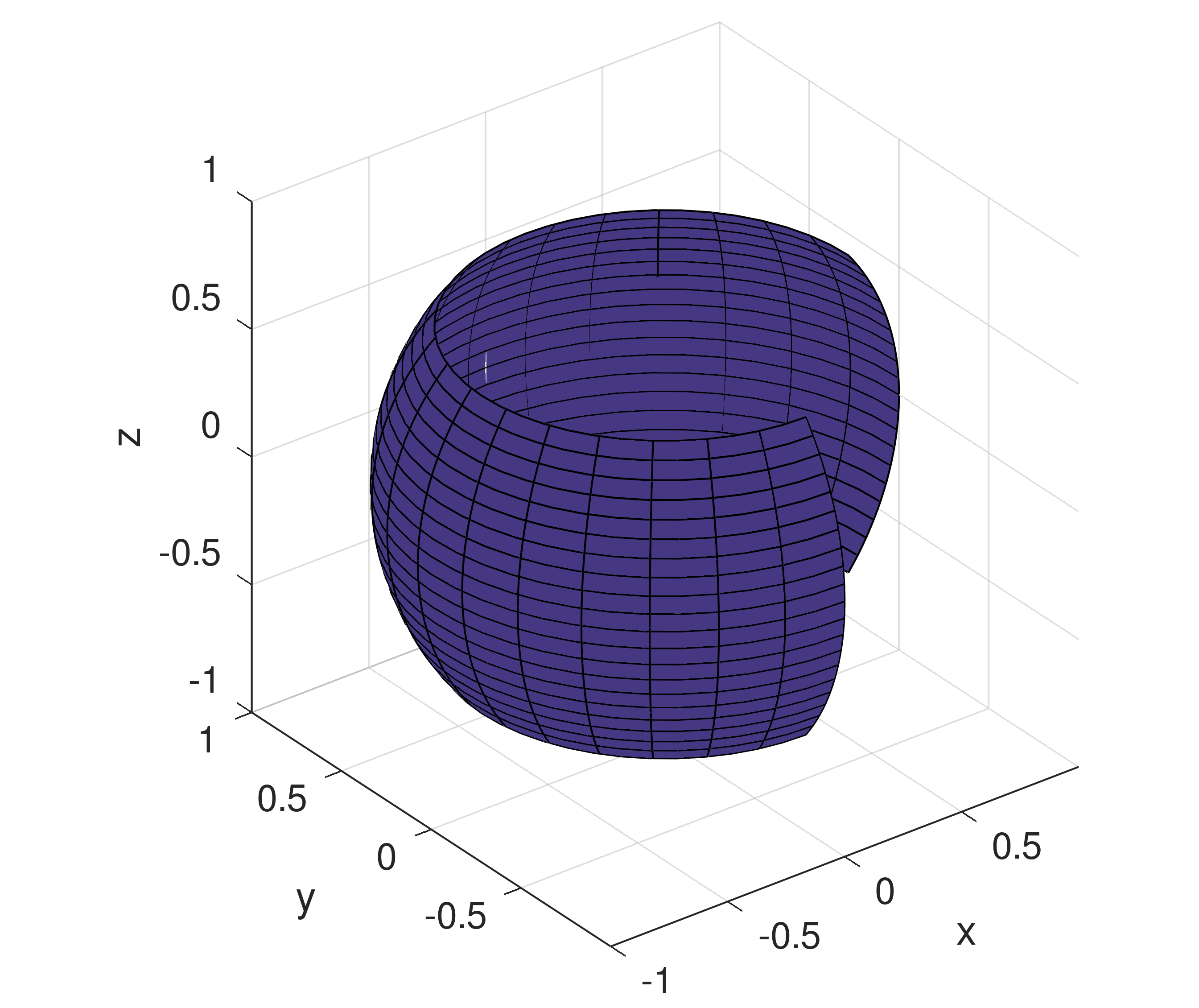}
    \end{subfigure}\hfill
    \begin{subfigure}[b]{0.49\linewidth}
    \includegraphics[width=\linewidth]{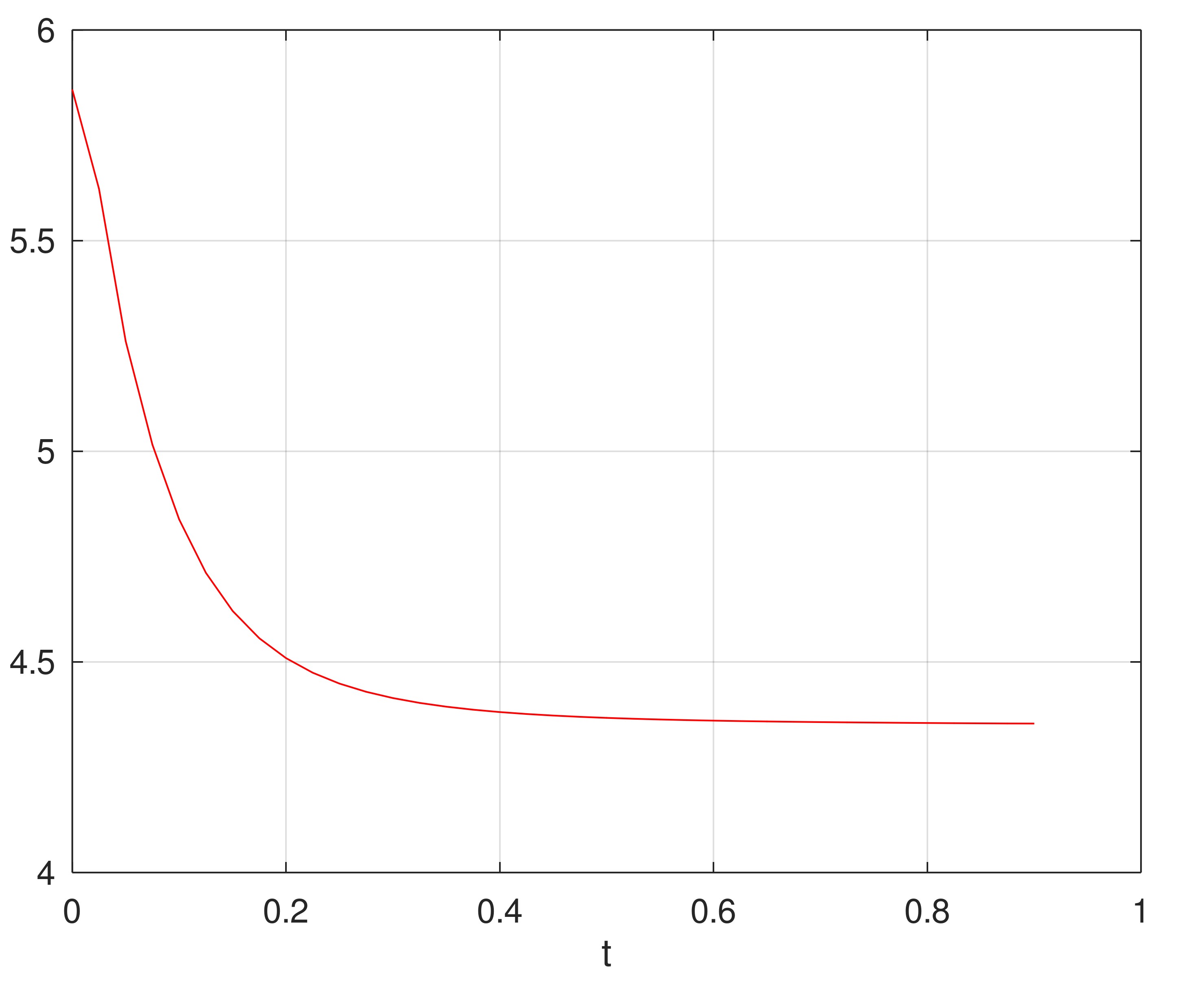}
    \end{subfigure}
    \caption{Initial surface of Example 2, with constant mean curvature $-2$ (left). Evolution of area of $\Gammaht$ against time (right), starting at $5.859$ at $t=0$, and ending at $4.354$ at $t=0.9$.}
    \label{fig: superficie inicial esfera}
    \end{figure}  

    In Figure~\ref{fig: superficies aproximadas esfera}, we present the evolution of the surface $\Gammah$ for different time steps, using its color to show the value of the variable $\Hh$. We consider the interval $[0,0.9]$ with time step size of $\deltat=0.025$. 

  \begin{figure}[!ht]
    \centering
    \begin{subfigure}[b]{0.49\linewidth}
      \includegraphics[width=\linewidth]{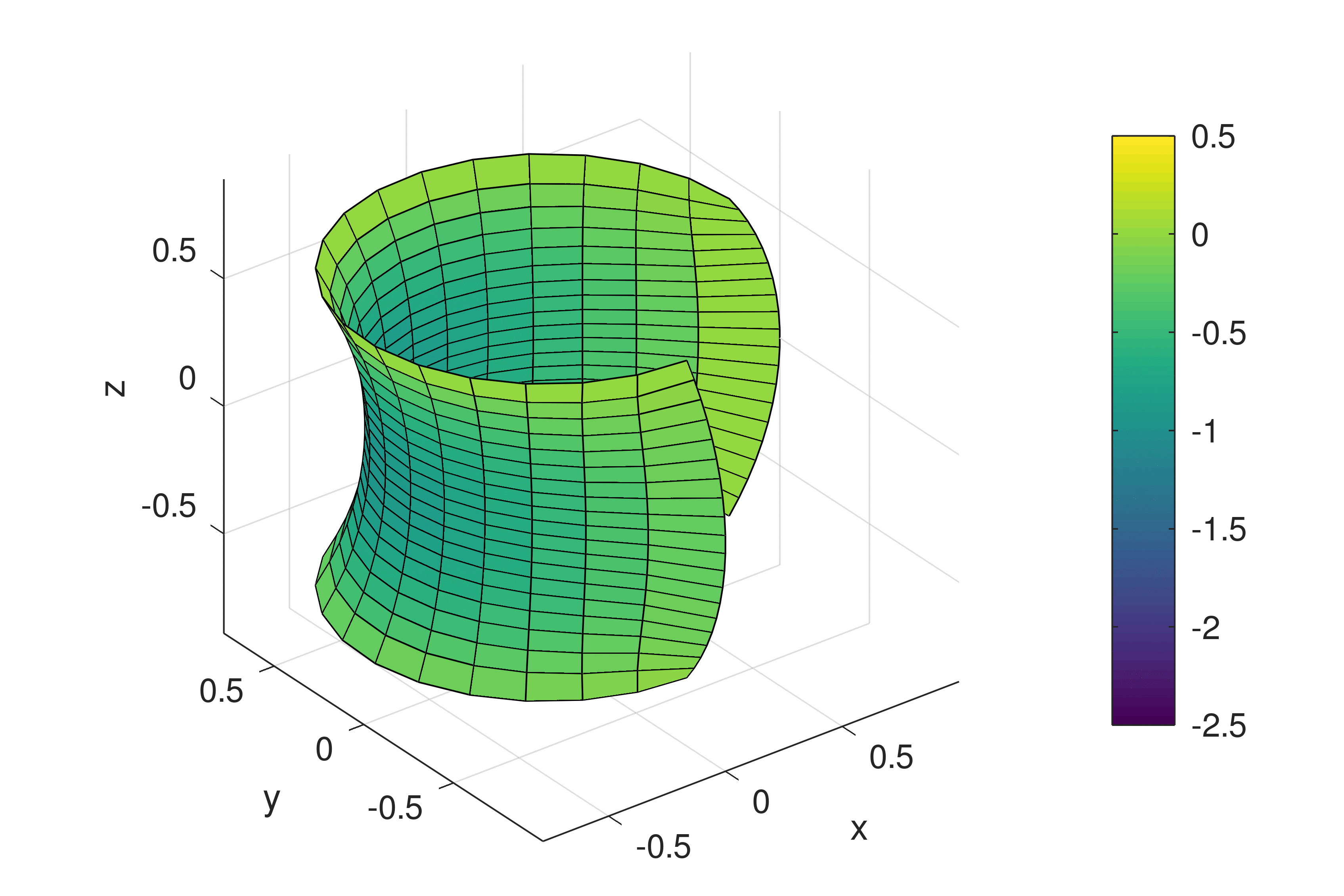}
    \end{subfigure}
    \begin{subfigure}[b]{0.49\linewidth}
      \includegraphics[width=\linewidth]{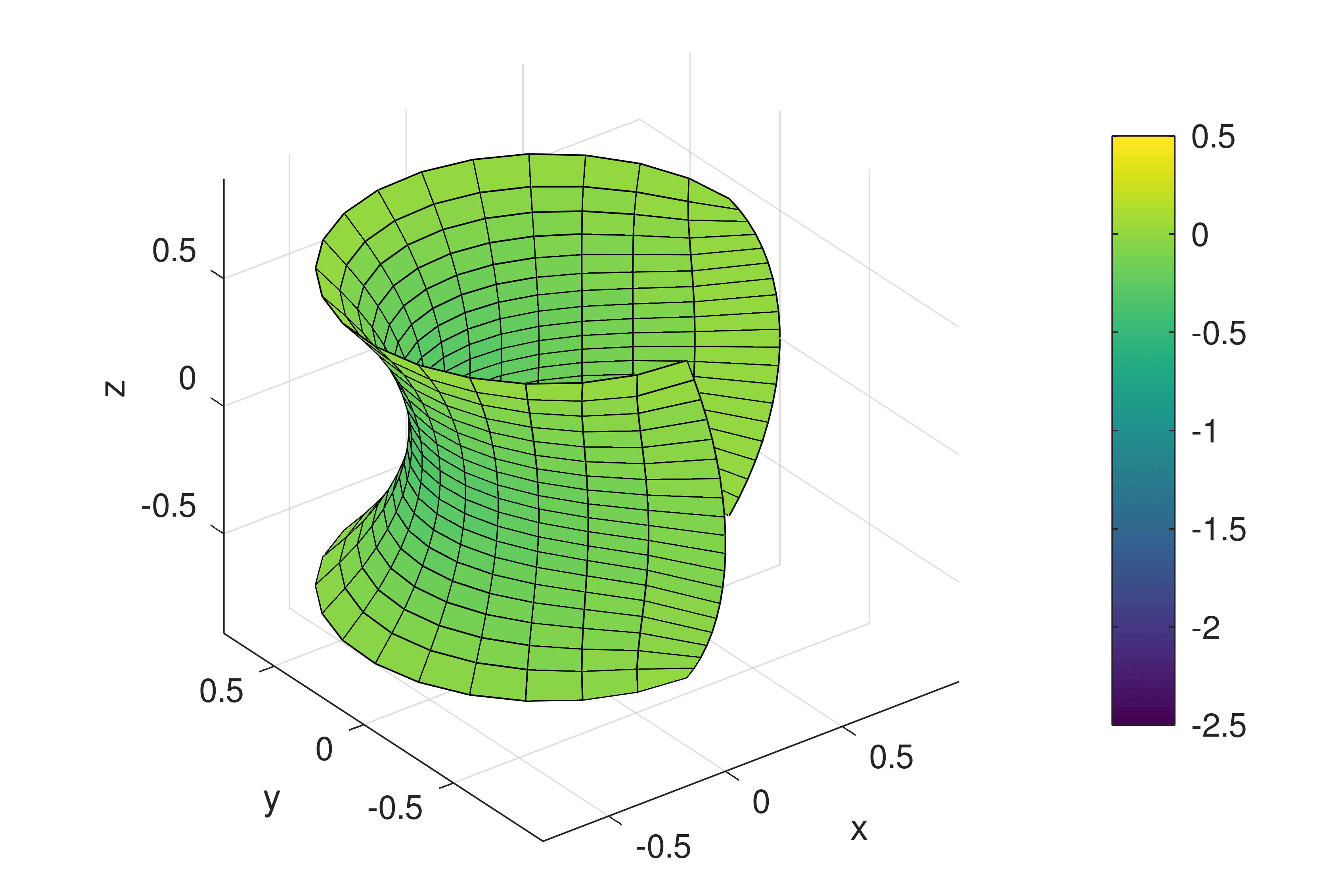}
    \end{subfigure}
    \vspace{0.5cm} 
    \begin{subfigure}[b]{0.49\linewidth}
      \includegraphics[width=\linewidth]{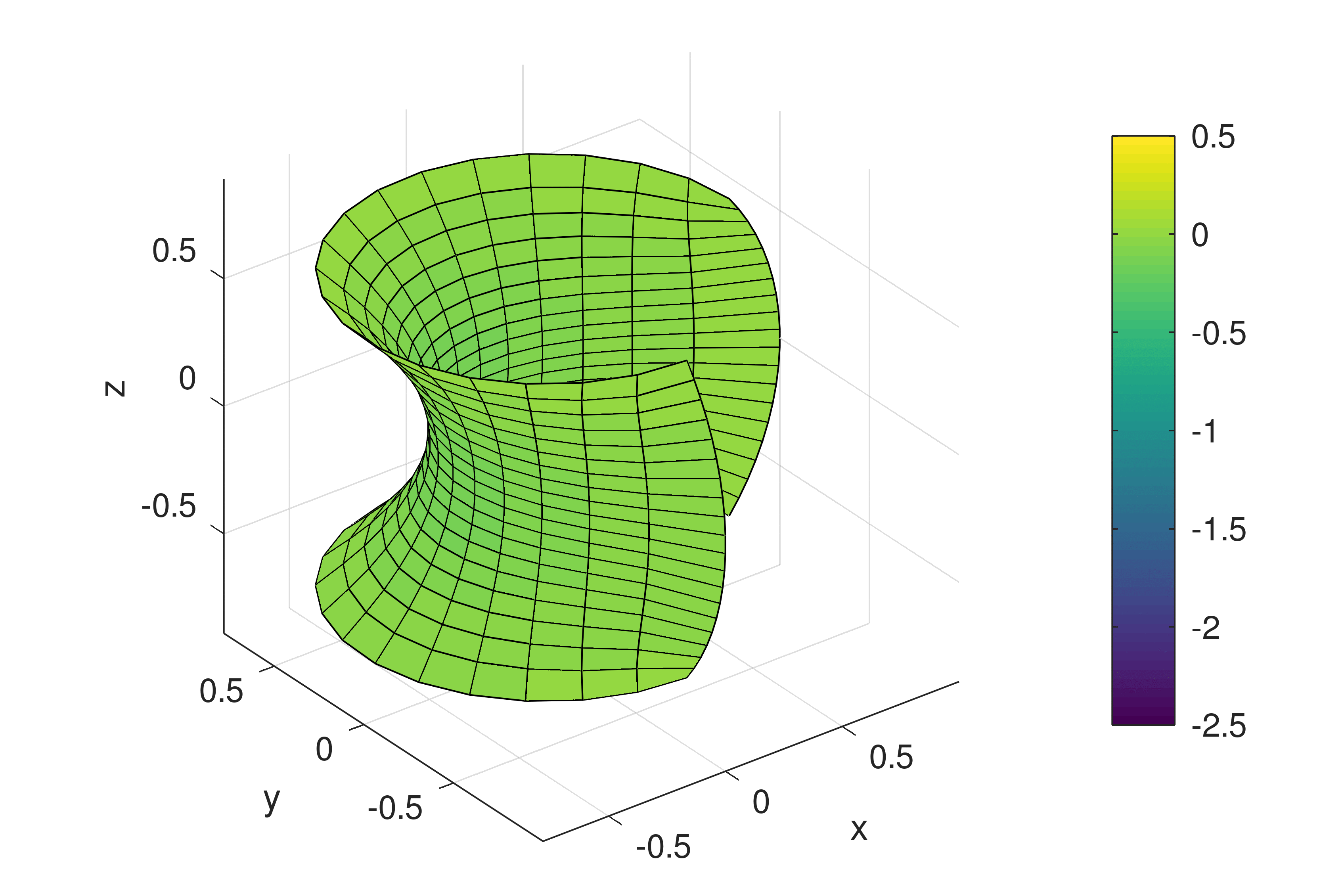} 
    \end{subfigure}
    \begin{subfigure}[b]{0.49\linewidth}
      \includegraphics[width=\linewidth]{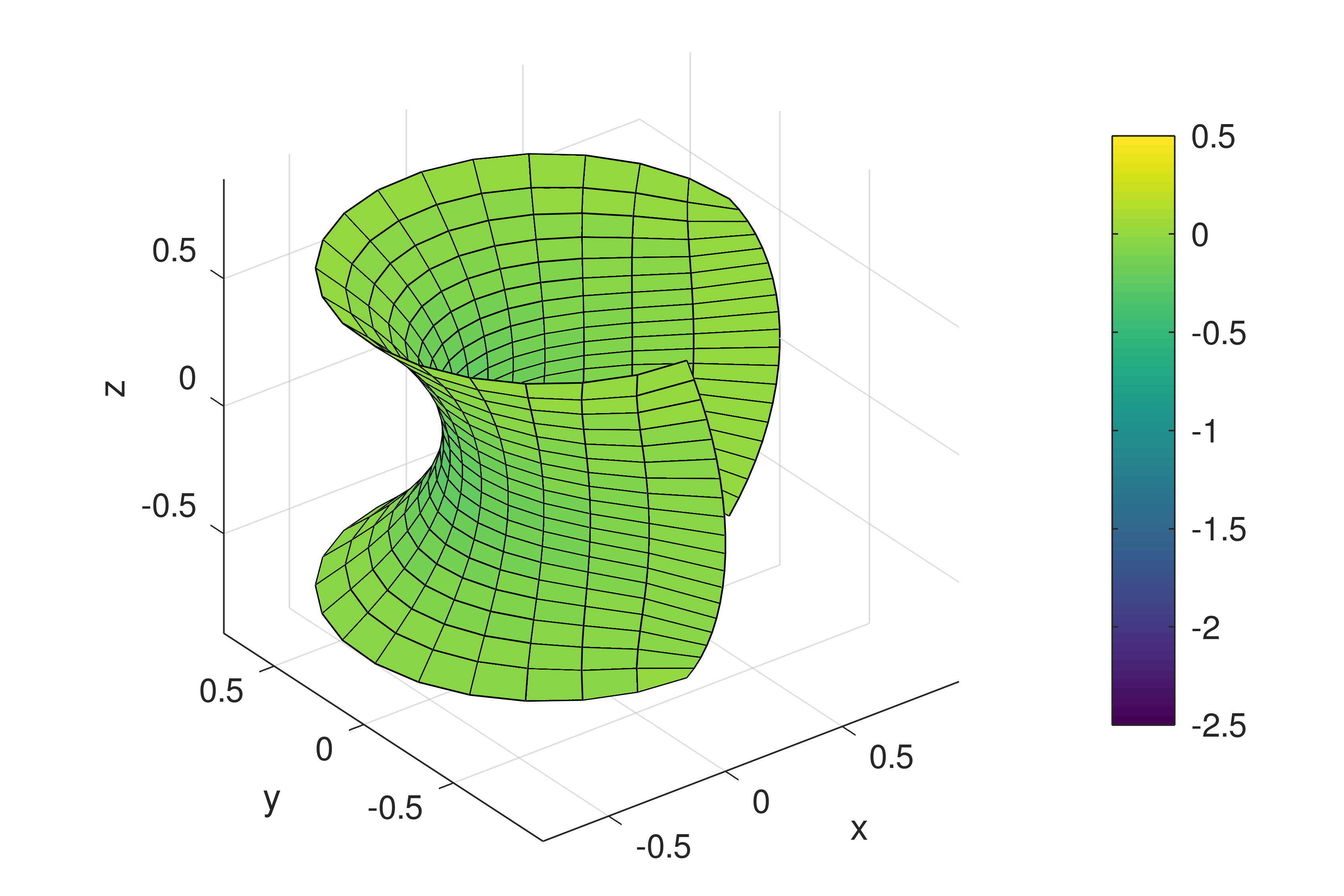} 
    \end{subfigure}
    \caption{Evolving surface at times $t=0.25$ (top-left), $t=0.5$ (top-right), $t=0.75$ (bottom-left) and $t=0.9$ (bottom-right). The colors indicate the value of the discrete variable $\Hh$. }
    \label{fig: superficies aproximadas esfera}
  \end{figure}

  \subsection*{Acknowledgements}
  The authors were partially supported by Agencia Nacional de Promoci\'on Cient\'ifica y Tecnol\'ogica through grant PICT-2020-SERIE A-03820, and by Universidad Nacional del Litoral through grants CAI+D-2020 50620190100136LI and CAI+D-2024 85520240100018LI.

\bibliographystyle{abbrv}
\bibliography{references}

\end{document}